\documentclass[twoside]{article}
\usepackage[letterpaper, left=3.5cm, right=3.5cm, top=4cm, bottom=2.5cm]{geometry}
\usepackage{graphicx}
\usepackage{amsmath,amssymb,amsthm}
\usepackage{authblk}
\usepackage[utf8]{inputenc}
\usepackage{microtype}
\usepackage{mathrsfs} 
\usepackage{enumitem}
\usepackage{calc}
\usepackage{esint}
\usepackage{fancyhdr}
\usepackage{xcolor}
\usepackage[labelfont = bf]{caption}
\usepackage{doi}
\usepackage{subfigure}
\usepackage[numbers]{natbib}
\usepackage{float}
\usepackage{stmaryrd}
\usepackage{mathtools}
\usepackage{booktabs}
\usepackage{colortbl}
\usepackage{tabulary}
\usepackage{diagbox}
\usepackage{caption}
\usepackage{cleveref}
\usepackage{commath}

\usepackage{physics}
\usepackage{amsmath}
\usepackage{tikz}
\usepackage{mathdots}
\usepackage{yhmath}
\usepackage{cancel}
\usepackage{color}
\usepackage{siunitx}
\usepackage{array}
\usepackage{multirow}
\usepackage{amssymb}
\usepackage{gensymb}
\usepackage{tabularx}
\usepackage{extarrows}
\usepackage{booktabs}
\usetikzlibrary{fadings}
\usetikzlibrary{patterns}
\usetikzlibrary{shadows.blur}
\usetikzlibrary{shapes}

\newtheorem{theorem}{Theorem}[section]
\numberwithin{equation}{section}

\newtheorem{remark}[theorem]{Remark}
\newtheorem{lemma}[theorem]{Lemma}

\newtheorem{algorithm}[theorem]{Algorithm}

\usepackage{titlesec}
\titleformat{\section}{\normalfont\scshape\centering}{\thesection.}{0.5em}{}
\titleformat*{\subsection}{\itshape}
\titleformat*{\subsubsection}{\itshape}

\providecommand{\keywords}[1]
{
	{\small\emph{Keywords:} #1}
}

\providecommand{\MSC}[1]
{
	{\small\emph{AMS MSC (2020):~~} #1}
}

\definecolor{denim}{rgb}{0.08, 0.38, 0.74}
\definecolor{byzantium}{rgb}{0.44, 0.16, 0.39} 
\definecolor{shamrockgreen}{rgb}{0.0, 0.62, 0.38} 

\usepackage{hyperref}
\hypersetup{
	colorlinks=true,
	linkcolor=denim,
	citecolor = shamrockgreen,
	filecolor=magenta, 
	urlcolor=byzantium,
} 

\providecommand{\jumptmp}[2]{#1\llbracket{#2}#1\rrbracket}
\providecommand{\jump}[1]{\jumptmp{}{#1}}

\DeclareFontFamily{U}{stix2bb}{}
\DeclareFontShape{U}{stix2bb}{m}{n} {<-> stix2-mathbb}{}

\NewDocumentCommand{\indicator}{}{\text{\usefont{U}{stix2bb}{m}{n}1}}

\begin{document}
	\setlength{\abovedisplayskip}{5.5pt}
	\setlength{\belowdisplayskip}{5.5pt}
	\setlength{\abovedisplayshortskip}{5.5pt}
	\setlength{\belowdisplayshortskip}{5.5pt}

	\title{\vspace{-15mm}Duality-Based Algorithm and Numerical Analysis for Optimal Insulation Problems on Non-Smooth Domains\thanks{This work is partially supported by the Office of Naval Research (ONR) under Award NO: N00014-24-1-2147, NSF grant DMS-2408877, the Air Force Office of Scientific Research (AFOSR) under Award NO: FA9550-22-1-0248.}\vspace{-1mm}}
	\author[1]{Harbir Antil\thanks{Email: \url{hantil@gmu.edu}}}
	\author[2]{Alex Kaltenbach\thanks{Email: \url{kaltenbach@math.tu-berlin.de}}}
	\author[3]{Keegan L. A. Kirk\thanks{Email: \url{kkirk6@gmu.edu}}\vspace{-4mm}}
	\date{\today\vspace{-5mm}} 
	\affil[1,3]{\small{Department of Mathematical Sciences and the Center for Mathematics and Artificial Intelligence (CMAI), George Mason University, Fairfax, VA 22030, USA}}
	\affil[2]{\small{Institute of Mathematics, Technical University of Berlin, Stra\ss e des 17.\ Juni 135, 10623 Berlin\vspace{-1mm}}}
	\maketitle

	\pagestyle{fancy}
	\fancyhf{}
	\fancyheadoffset{0cm}
	\addtolength{\headheight}{-0.25cm}
	\renewcommand{\headrulewidth}{0pt} 
	\renewcommand{\footrulewidth}{0pt}
	\fancyhead[CO]{\textsc{Duality-based algorithm for optimal insulation}}
	\fancyhead[CE]{\textsc{H. Antil, A. Kaltenbach, and K. Kirk}}
	\fancyhead[R]{\thepage}
	\fancyfoot[R]{}
	
	\begin{abstract}
        This article develops  a numerical approximation of a convex non-local~and~\mbox{non-smooth} minimization problem. The physical problem involves determining the optimal~\mbox{distribution}, given by $h\colon \Gamma_I\to [0,+\infty)$, of a given amount $m\in \mathbb{N}$ of insulating~\mbox{material}~attached~to~a boundary part $\Gamma_I\subseteq  \partial\Omega$ of a thermally conducting body $\Omega\hspace{-0.1em}\subseteq\hspace{-0.1em}\mathbb{R}^d$,~$d\hspace{-0.1em}\in\hspace{-0.1em}  \mathbb{N}$,
        subject to conductive heat transfer.
         To~tackle~the~\hspace{-0.1mm}\mbox{non-local} \hspace{-0.1mm}and \hspace{-0.1mm}non-smooth \hspace{-0.1mm}character \hspace{-0.1mm}of \hspace{-0.1mm}the \hspace{-0.1mm}problem,~\hspace{-0.1mm}the~\hspace{-0.1mm}\mbox{article}~\hspace{-0.1mm}introduces~\hspace{-0.1mm}a~\hspace{-0.1mm}(Fenchel)~\hspace{-0.1mm}\mbox{duality}~\hspace{-0.1mm}\mbox{framework}:

        \textit{(a)} 
        At the continuous level, 
        using (Fenchel) duality relations, we derive an \textit{a~posteriori} error identity that can handle arbitrary admissible approximations of the primal and dual formulations of the convex non-local and non-smooth~minimization~problem;
        
        \textit{(b)} At the discrete level, using discrete (Fenchel) duality relations, we derive an~\textit{a~priori} error identity that applies
		to a Crouzeix--Raviart discretization of the primal formulation and a Raviart--Thomas discretization of the dual formulation. The proposed framework leads to error decay rates that are optimal with~respect~to~the~\mbox{specific}~\mbox{regularity}~of~a~minimizer. In addition, we prove convergence of the numerical approximation under minimal regularity assumptions.
        Since the discrete dual formulation can be written as a quadratic program, it is solved using a primal-dual active set strategy interpreted as semi-smooth~Newton~method. A solution of the discrete primal formulation is reconstructed from the solution of the discrete dual formulation by means of an inverse generalized Marini formula. This is the first such formula for this 
         class of convex non-local~and~\mbox{non-smooth}~minimization~problems.
	\end{abstract}
	
	\keywords{optimal insulation; Crouzeix--Raviart element; Raviart--Thomas element, \emph{a priori} error identity; \emph{a posteriori} error identity; (Fenchel) duality theory; semi-smooth Newton method.}
	
	\MSC{35J20; 49J40; 49M29; 65N30; 65N15; 65N50.}
	
	\section{Introduction}\thispagestyle{empty}\vspace{-1.5mm}

    \hspace{5mm}The present paper is interested in determining the \textit{`best'} distribution of a~given~amount~of an insulating material attached to parts of a thermally conducting body $\Omega\subseteq\mathbb{R}^d$,~${d\in  \mathbb{N}}$.~To~this~end,\linebreak we study a non-local and non-smooth  convex minimization problem first proposed by \mbox{\textsc{Buttazzo}}   (\textit{cf}.~\cite{Buttazzo1988}) and recently extended by the authors to the case of bounded polyhedral Lipschitz~domains as well as to a mixed boundary setting (\textit{i.e.}, Dirichlet, Neumann, and insulated boundary, \textit{cf}.~\cite{AKKInsulationModel}): Let $\Omega\subseteq \mathbb{R}^d$, $d\in \mathbb{N}$, be a bounded polyhedral Lipschitz domain
    representing~the~\textit{\mbox{thermally}~conduc-ting body}, with (topological) boundary $\partial\Omega$ decomposed into an \textit{insulation part} (\textit{i.e.}, $\Gamma_I$)  (to which the insulating material is attached), a Dirichlet part (\textit{i.e.}, $\Gamma_D$),~and~a~Neumann~part~(\textit{i.e.},~$\Gamma_N$). Then, for a given \textit{amount of insulating material} $m>0$, a given 
    \textit{heat~source~density}~${f\in L^2(\Omega)}$, a given \textit{heat flux} $g\hspace{-0.1em}\in\hspace{-0.1em}  H^{-\smash{\frac{1}{2}}}(\Gamma_N)$, and given \textit{Dirichlet boundary temperature distribution}~${u_D\hspace{-0.1em}\in  \hspace{-0.1em}H^{\smash{\frac{1}{2}}}(\Gamma_D)}$ with~boundary~lift~${\widehat{u}_D\in  H^1(\Omega)}$, we seek a \textit{temperature distribution} $u\in \widehat{u}_D+H^1_D(\Omega)$ that minimizes the energy functional $I\colon \widehat{u}_D+H^1_D(\Omega)\to \mathbb{R}$, for every $v\in H^1(\Omega) $ defined by\enlargethispage{17.5mm}
\begin{align}\label{eq:primal_intro} 
		\smash{I(v) \coloneqq  \tfrac{1}{2}\| \nabla v\|_{\Omega}^2+\tfrac{1}{2m}\|v\|_{1,\Gamma_I}^2-(f,v)_{\Omega}-\langle g,v\rangle_{\Gamma_N}\,.} 
\end{align}\newpage
%The associated Euler--Lagrange equations read
%\begin{align*}
%    \begin{aligned}
%    -\Delta u &= f&&\quad \text{ a.e.\ in }\Omega\,,\\[-0.5mm]
%    -\nabla u\cdot n&\in 
%   \smash{\tfrac{1}{m}} (\partial\vert\cdot\vert)(u)\|u\|_{1,\Gamma_I}&&\quad\text{ a.e.\ on }\Gamma_I\,,\\[-0.5mm]
%    u&=u_D&&\quad \text{ a.e.\ on }\Gamma_D\,,\\[-0.5mm]
%    \nabla u\cdot n&=g&&\quad \text{ a.e.\ on }\Gamma_N\,,
%    \end{aligned}
%\end{align*}
%where $ (\partial\,\vert\cdot\vert)(t)=\textup{sign}(t)$ if $t\neq 0$ and  $(\partial\,\vert\cdot\vert)(t)=[-1,1]$.%$\partial\vert\cdot\vert\colon \hspace{-0.1em}\mathbb{R}\hspace{-0.1em}\to\hspace{-0.1em} 2^{\mathbb{R}}$ is the subdifferential of the absolute value $\vert\cdot\vert\colon\hspace{-0.1em} \mathbb{R}\hspace{-0.1em}\to\hspace{-0.1em}[0,+\infty)$, \textit{i.e.},~for~every~$t\hspace{-0.1em}\in\hspace{-0.1em} \mathbb{R}$, we have that
%\begin{align*}
%    (\partial\,\vert\cdot\vert)(t)\coloneqq\begin{cases}
%        \textup{sign}(t)&\text{ if }t\neq 0\,,\\
%        [-1,1]&\text{ if }t=0\,.
%    \end{cases}
%\end{align*}

Once a temperature distribution $u\in H^1(\Omega)$ minimizing the functional \eqref{eq:primal_intro} is determined, the optimal distribution of the given amount insulating material can be calculated~as~follows: 
\begin{itemize}[noitemsep,topsep=2pt,leftmargin=!,labelwidth=\widthof{2. Step:}]
    \item[\textit{Step 1:}]  Identify a Lipschitz continuous  (globally) transversal vector field $k\in (C^{0,1}(\partial\Omega))^d$ of unit-length, \textit{i.e.}, there exists a constant $\kappa\in (0,1]$ (the transversality~constant)~such~that
    \begin{align}\label{intro:transversal}
        \begin{aligned} 
            %\vert k\vert&=1&&\quad \text{ on }\partial\Omega\,,\\
            k\cdot n &\ge \kappa&&\quad \text{ a.e.\ on }\partial\Omega\,.
        \end{aligned}
    \end{align}
    Note that, for each bounded Lipschitz domain, one can establish the existence of a smooth (globally) transversal vector field of unit-length (\textit{cf}.\ \cite[Cor.\ 2.13]{HMT07}). If $\Omega$ is star-shaped with respect to a ball $B_r^d(x_0)\subseteq \Omega$,  a smooth~(globally)~transversal~vector~field~of unit-length is given via $k\coloneqq \frac{\textrm{id}_{\smash{\mathbb{R}^d}}-x_0}{\vert \textrm{id}_{\smash{\mathbb{R}^d}}-x_0\vert}\in(C^{\infty}(\partial\Omega))^d$ (\textit{cf}.\ \cite[Cor.~4.21]{HMT07}); %If $\partial\Omega\in \smash{C^{1,1}}$, the outward unit normal vector field $n\colon \partial\Omega\to \smash{\mathbb{S}^{d-1}}$~is~a~Lipschitz~\mbox{continuous}~\mbox{(globally)} transversal vector field (with transversality constant $1$) (\textit{cf}.\ \cite[Thm.~2.19,~(2.75)]{HMT07});

    \item[\textit{Step 2:}] Compute the optimal distribution of the  insulating material via the explicit formula 
    \begin{align}\label{intro:distribution_function}
    h_u\coloneqq \tfrac{m}{\|u\|_{1,\Gamma_I}}\tfrac{\vert u\vert}{k\cdot n}\in L^1(\Gamma_I)\,.
    \end{align}
    More precisely, the distribution function \eqref{intro:distribution_function} represents the distribution in direction of the transversal vector field $k\in (C^{0,1}(\partial\Omega))^d$ (rather that in direction of $n\in (L^\infty(\partial\Omega))^d$).\linebreak This \hspace{-0.1mm}enables \hspace{-0.1mm}to \hspace{-0.1mm}determine \hspace{-0.1mm}the \hspace{-0.1mm}optimal \hspace{-0.1mm}distribution \hspace{-0.1mm}of \hspace{-0.1mm}the \hspace{-0.1mm}insulating~\hspace{-0.1mm}material,~\hspace{-0.1mm}in~\hspace{-0.1mm}\mbox{particular},  at kinks and edges of the thermally conducting body and to avoid gaps (\textit{i.e.}, no insulating material is attached) and self-intersections (\textit{i.e.}, insulating material~is~attached~twice) in the arbitrarily thin insulated boundary layer, see \cite{AKKInsulationModel} for a more~detailed~discussion.
\end{itemize}

In this paper, we are interested in the numerical approximation of the minimization of \eqref{eq:primal_intro}. Here, the main challenge arises from the 
non-local and non-smooth character~of~the~functional~\eqref{eq:primal_intro}.
To tackle this, we resort to a (Fenchel) duality framework. The main contributions~of~the~present paper as well as related contributions are summarized next:\vspace{-1.75mm}\enlargethispage{4mm}

\subsection{Main contributions}\vspace{-1mm}

\begin{enumerate}[noitemsep,topsep=2pt,leftmargin=!,labelwidth=\widthof{6.},font=\itshape]

    \item A (Fenchel) dual problem (in the sense of \cite[Rem.\ 4.2, p.\ 60/61]{ET99}) to the minimization~of~\eqref{eq:primal_intro} as well as convex optimality relations and a strong duality relation are identified in Theorem~\ref{thm:duality};

    \item On the basis of (Fenchel) duality relations, an \textit{a posteriori} error identity, which is applicable to  arbitrary admissible approximations of the primal and dual problem, is derived in  Theorem~\ref{thm:prager_synge_identity};

    \item  A Crouzeix--Raviart approximation of \eqref{eq:primal_intro} is proposed and an associated (Fenchel) dual problem defined on the Raviart--Thomas element is identified. In particular, discrete convex optimality and discrete  strong duality  relations~are~derived~in~\mbox{Theorem}~\ref{prop:discrete_duality};

    \item  On the basis of  discrete convex (Fenchel) duality relations, a discrete \textit{a posteriori} error identity,\linebreak which applies to arbitrary discrete admissible approximations of the discrete~primal~and discrete dual problem, is derived~in Theorem~\ref{thm:discrete_prager_synge_identity}. This is followed by \emph{a priori} error estimates~in Theorem~\ref{thm:apriori_identity}, whose \hspace{-0.1mm}optimality \hspace{-0.1mm}is \hspace{-0.1mm}confirmed \hspace{-0.1mm}in \hspace{-0.1mm}some \hspace{-0.1mm}cases~\hspace{-0.1mm}via~\hspace{-0.1mm}numerical~\hspace{-0.1mm}\mbox{experiments}~\hspace{-0.1mm}in~\hspace{-0.1mm}\mbox{Section}~\hspace{-0.1mm}\ref{sec:experiments};

    \item The discrete dual problem is equivalent to a quadratic program  solved with~a~\mbox{semi-smooth} Newton \hspace{-0.1mm}scheme. \hspace{-0.1mm}A \hspace{-0.1mm}Lagrange \hspace{-0.1mm}multiplier \hspace{-0.1mm}obtained \hspace{-0.1mm}as \hspace{-0.1mm}a \hspace{-0.1mm}by-product \hspace{-0.1mm}is \hspace{-0.1mm}used \hspace{-0.1mm}in~\hspace{-0.1mm}a reconstruction formula for a discrete primal solution 
     from the discrete~dual~solution~in~\mbox{Lemma}~\ref{lem:marini};
     
     \item Numerical %test cases and 
     experiments confirming the theoretical findings are presented~in~Section~\ref{sec:experiments}.~Moreover, Section \ref{sec:experiments} considers
     a real-world test case of a 3D geometry
modelling~a~simple~house~with~garage.\vspace{-1.75mm}
\end{enumerate}

 \subsection{Related contributions}\vspace{-1mm}

\hspace{5mm}The existing literature either focuses on the theoretical analysis of the minimization~of~\eqref{eq:primal_intro}~or on the numerical analysis of a related eigenvalue problem (each in the case $\partial\Omega\in C^{1,1}$~and~${\Gamma_I=\partial\Omega}$):

\begin{itemize} 
[noitemsep,topsep=2pt,leftmargin=!,labelwidth=\widthof{6.},font=\itshape]

    \item \textit{Theoretical analysis:} There are several contributions addressing the derivation of the functional \eqref{eq:primal_intro} in a suitable sense as asymptotic limit (as the thickness of the insulating layer tends~to~zero) (\textit{cf}.\ \cite{BrezisCaffarelliFriedman1980,CaffarelliFriedman1980,AcerbiButtazzo1986b,ButtazzoKohn1987,ButtazzoDalMasoMosco1989,BoutkridaMossinoMoussa1999,BoutkridaGrenonMossinoMoussa2002,MossinoVanninathan2002,PietraNitschScalaTrombetti2021,AcamporaCristoforoniNitschTrombetti2024}) and contributions addressing related analytical studies of the functional \eqref{eq:primal_intro} (\textit{cf}.\ \cite{Buttazzo1988c,Buttazzo1988,ButtazzoZeine1997,EspositoRiey2003,Buttazzo2017,HuangLiLi2022});

    \item \textit{Numerical \hspace{-0.1mm}analysis:} \hspace{-0.1mm}A \hspace{-0.1mm}purely \hspace{-0.1mm}experimental \hspace{-0.1mm}study \hspace{-0.1mm}of \hspace{-0.1mm}the \hspace{-0.1mm}minimization \hspace{-0.1mm}of \hspace{-0.1mm}\eqref{eq:primal_intro} \hspace{-0.1mm}can~\hspace{-0.1mm}be~\hspace{-0.1mm}found~\hspace{-0.1mm}in~\hspace{-0.1mm}\cite{Munoz2007}.  Thorough 
    numerical analyses of a related eigenvalue problem can found in the papers 
    \cite{BartelsButtazzo2019,BartelsKellerWachsmuth2022,BartelsButtazzoKeller2024}.
\end{itemize}

\newpage
	\section{Preliminaries}\label{sec:preliminaries}
    
    \subsection{Assumptions on the thermally conducting body and insulated boundary}\enlargethispage{10mm}
	
	\hspace{5mm}Throughout the  article,  let ${\Omega\subseteq \mathbb{R}^d}$, ${d\in\mathbb{N}}$, be a bounded simplicial Lipschitz domain such that $\partial\Omega$ is split into three 
	 (relatively) open boundary parts: an insulated boundary~${\Gamma_I\hspace{-0.15em}\subseteq\hspace{-0.15em} \partial\Omega}$~with~${\Gamma_I\hspace{-0.15em}\neq\hspace{-0.15em}\emptyset}$, a~\mbox{Dirichlet}~boundary~${\Gamma_D\subseteq \partial\Omega}$, and a Neumann boundary $\Gamma_N\subseteq \partial\Omega$ such that $\partial\Omega=\overline{\Gamma}_D\cup\overline{\Gamma}_N\cup\overline{\Gamma}_C$.\vspace{-1.5mm}

     \begin{figure}[H]
         \centering
         
\tikzset{every picture/.style={line width=0.75pt}} %set default line width to 0.75pt        

\begin{tikzpicture}[x=0.95pt,y=0.95pt,yscale=-1,xscale=1]

%Shape: Grid [id:dp8507570588732107] 
\draw  [draw opacity=0] (59.63,51.01) -- (209.63,51.01) -- (209.63,201.01) -- (59.63,201.01) -- cycle ; \draw  [color={rgb, 255:red, 155; green, 155; blue, 155 }  ,draw opacity=0.25 ] (69.63,51.01) -- (69.63,201.01)(79.63,51.01) -- (79.63,201.01)(89.63,51.01) -- (89.63,201.01)(99.63,51.01) -- (99.63,201.01)(109.63,51.01) -- (109.63,201.01)(119.63,51.01) -- (119.63,201.01)(129.63,51.01) -- (129.63,201.01)(139.63,51.01) -- (139.63,201.01)(149.63,51.01) -- (149.63,201.01)(159.63,51.01) -- (159.63,201.01)(169.63,51.01) -- (169.63,201.01)(179.63,51.01) -- (179.63,201.01)(189.63,51.01) -- (189.63,201.01)(199.63,51.01) -- (199.63,201.01) ; \draw  [color={rgb, 255:red, 155; green, 155; blue, 155 }  ,draw opacity=0.25 ] (59.63,61.01) -- (209.63,61.01)(59.63,71.01) -- (209.63,71.01)(59.63,81.01) -- (209.63,81.01)(59.63,91.01) -- (209.63,91.01)(59.63,101.01) -- (209.63,101.01)(59.63,111.01) -- (209.63,111.01)(59.63,121.01) -- (209.63,121.01)(59.63,131.01) -- (209.63,131.01)(59.63,141.01) -- (209.63,141.01)(59.63,151.01) -- (209.63,151.01)(59.63,161.01) -- (209.63,161.01)(59.63,171.01) -- (209.63,171.01)(59.63,181.01) -- (209.63,181.01)(59.63,191.01) -- (209.63,191.01) ; \draw  [color={rgb, 255:red, 155; green, 155; blue, 155 }  ,draw opacity=0.25 ] (59.63,51.01) -- (209.63,51.01) -- (209.63,201.01) -- (59.63,201.01) -- cycle ;
%Shape: Square [id:dp040553399727301764] 
\draw   (59.83,50.75) -- (209.92,50.75) -- (209.92,200.83) -- (59.83,200.83) -- cycle ;
%Straight Lines [id:da3901014496760229] 
\draw [color={denim}  ,draw opacity=0.5 ]   (210.25,50.8) -- (220.25,60) ;
%Straight Lines [id:da15181644894843793] 
\draw [color={denim}  ,draw opacity=0.5 ]   (210,60.75) -- (220,69.95) ;
%Straight Lines [id:da15340904997210703] 
\draw [color={denim}  ,draw opacity=0.5 ]   (210,55.75) -- (220,64.95) ;
%Straight Lines [id:da2983628874519848] 
\draw [color={denim}  ,draw opacity=0.5 ]   (210.25,65.8) -- (220.25,75) ;
%Straight Lines [id:da15674498817696292] 
\draw [color={denim}  ,draw opacity=0.5 ]   (210,75.75) -- (220,84.95) ;
%Straight Lines [id:da8482927965969174] 
\draw [color={denim}  ,draw opacity=0.5 ]   (210,70.75) -- (220,79.95) ;
%Straight Lines [id:da8864742610135565] 
\draw [color={denim}  ,draw opacity=0.5 ]   (210.25,81.3) -- (220.25,90.5) ;
%Straight Lines [id:da09826637938690475] 
\draw [color={denim}  ,draw opacity=0.5 ]   (210,91.25) -- (220,100.45) ;
%Straight Lines [id:da14725576270386553] 
\draw [color={denim}  ,draw opacity=0.5 ]   (210,86.25) -- (220,95.45) ;
%Straight Lines [id:da09553207684375264] 
\draw [color={denim}  ,draw opacity=0.5 ]   (210.25,96.3) -- (220.25,105.5) ;
%Straight Lines [id:da9625563377854753] 
\draw [color={denim}  ,draw opacity=0.5 ]   (210,106.25) -- (220,115.45) ;
%Straight Lines [id:da820270091243533] 
\draw [color={denim}  ,draw opacity=0.5 ]   (210,101.25) -- (220,110.45) ;
%Straight Lines [id:da30636232065989155] 
\draw [color={denim}  ,draw opacity=0.5 ]   (210,110.75) -- (220,119.95) ;
%Straight Lines [id:da8621862247392325] 
\draw [color={denim}  ,draw opacity=0.5 ]   (209.75,121.5) -- (219.75,130.7) ;
%Straight Lines [id:da30517612256098436] 
\draw [color={denim}  ,draw opacity=0.5 ]   (209.75,116.5) -- (214,120.5) ;
%Straight Lines [id:da34483454137696823] 
\draw [color={denim}  ,draw opacity=0.5 ]   (210,126.55) -- (220,135.75) ;
%Straight Lines [id:da1894939126899633] 
\draw [color={denim}  ,draw opacity=0.5 ]   (209.75,136.5) -- (219.75,145.7) ;
%Straight Lines [id:da9904995659979547] 
\draw [color={denim}  ,draw opacity=0.5 ]   (209.75,131.5) -- (219.75,140.7) ;
%Straight Lines [id:da8795800415091686] 
\draw [color={denim}  ,draw opacity=0.5 ]   (210,142.05) -- (220,151.25) ;
%Straight Lines [id:da11775863291433852] 
\draw [color={denim}  ,draw opacity=0.5 ]   (209.75,152) -- (219.75,161.2) ;
%Straight Lines [id:da46670967733996904] 
\draw [color={denim}  ,draw opacity=0.5 ]   (209.75,147) -- (219.75,156.2) ;
%Straight Lines [id:da20493952797910064] 
\draw [color={denim}  ,draw opacity=0.5 ]   (210,157.05) -- (220,166.25) ;
%Straight Lines [id:da17525284728956314] 
\draw [color={denim}  ,draw opacity=0.5 ]   (209.75,167) -- (219.75,176.2) ;
%Straight Lines [id:da47827486826918864] 
\draw [color={denim}  ,draw opacity=0.5 ]   (209.75,162) -- (219.75,171.2) ;
%Straight Lines [id:da2121850703780379] 
\draw [color={denim}  ,draw opacity=0.5 ]   (210,177) -- (220,186.2) ;
%Straight Lines [id:da9866588938538248] 
\draw [color={denim}  ,draw opacity=0.5 ]   (210,172) -- (220,181.2) ;
%Straight Lines [id:da6326102517608956] 
\draw [color={denim}  ,draw opacity=0.5 ]   (210.25,182.05) -- (220.25,191.25) ;
%Straight Lines [id:da07526234450660985] 
\draw [color={denim}  ,draw opacity=0.5 ]   (210,192) -- (219.63,201) ;
%Straight Lines [id:da044809785580841144] 
\draw [color={denim}  ,draw opacity=0.5 ]   (210,187) -- (220,196.2) ;
%Straight Lines [id:da06864703152329721] 
\draw [color={denim}  ,draw opacity=0.5 ]   (210.25,197.5) -- (214.13,201.5) ;
%Straight Lines [id:da2637690722116206] 
\draw [color={denim}  ,draw opacity=0.5 ]   (210.05,50.5) -- (200.9,40.47) ;
%Straight Lines [id:da6385948590147901] 
\draw [color={denim}  ,draw opacity=0.5 ]   (200.12,50.72) -- (190.97,40.69) ;
%Straight Lines [id:da9619933472603119] 
\draw [color={denim}  ,draw opacity=0.5 ]   (205.11,50.73) -- (195.96,40.7) ;
%Straight Lines [id:da7555298303878732] 
\draw [color={denim}  ,draw opacity=0.5 ]   (195.08,50.45) -- (185.94,40.43) ;
%Straight Lines [id:da9034009950301467] 
\draw [color={denim}  ,draw opacity=0.5 ]   (185.16,50.67) -- (176.01,40.65) ;
%Straight Lines [id:da04248370826653036] 
\draw [color={denim}  ,draw opacity=0.5 ]   (190.15,50.69) -- (181,40.66) ;
%Straight Lines [id:da30066354636301473] 
\draw [color={denim}  ,draw opacity=0.5 ]   (179.62,50.41) -- (170.48,40.38) ;
%Straight Lines [id:da7042598570906102] 
\draw [color={denim}  ,draw opacity=0.5 ]   (169.7,50.63) -- (160.55,40.6) ;
%Straight Lines [id:da9096989173319716] 
\draw [color={denim}  ,draw opacity=0.5 ]   (174.69,50.64) -- (165.54,40.61) ;
%Straight Lines [id:da05211905483129642] 
\draw [color={denim}  ,draw opacity=0.5 ]   (164.66,50.36) -- (155.52,40.33) ;
%Straight Lines [id:da9097247185825477] 
\draw [color={denim}  ,draw opacity=0.5 ]   (154.74,50.58) -- (145.59,40.56) ;
%Straight Lines [id:da9146153999324316] 
\draw [color={denim}  ,draw opacity=0.5 ]   (159.73,50.6) -- (150.58,40.57) ;
%Straight Lines [id:da48229115907619513] 
\draw [color={denim}  ,draw opacity=0.5 ]   (149.45,50.57) -- (140.31,40.54) ;
%Straight Lines [id:da7702684795427426] 
\draw [color={denim}  ,draw opacity=0.5 ]   (139.53,50.79) -- (130.38,40.76) ;
%Straight Lines [id:da5825209819731263] 
\draw [color={denim}  ,draw opacity=0.5 ]   (144.51,50.8) -- (135.37,40.77) ;
%Straight Lines [id:da14411383352993012] 
\draw [color={denim}  ,draw opacity=0.5 ]   (134.49,50.52) -- (125.34,40.49) ;
%Straight Lines [id:da7906051257777778] 
\draw [color={denim}  ,draw opacity=0.5 ]   (124.57,50.74) -- (115.42,40.72) ;
%Straight Lines [id:da006188146341178591] 
\draw [color={denim}  ,draw opacity=0.5 ]   (129.55,50.76) -- (120.41,40.73) ;
%Straight Lines [id:da22238635417074581] 
\draw [color={denim}  ,draw opacity=0.5 ]   (119.03,50.48) -- (109.88,40.45) ;
%Straight Lines [id:da2203146224443524] 
\draw [color={denim}  ,draw opacity=0.5 ]   (109.11,50.7) -- (99.96,40.67) ;
%Straight Lines [id:da4037030226642271] 
\draw [color={denim}  ,draw opacity=0.5 ]   (114.09,50.71) -- (104.95,40.68) ;
%Straight Lines [id:da19093721353956683] 
\draw [color={denim}  ,draw opacity=0.5 ]   (104.07,50.43) -- (94.92,40.4) ;
%Straight Lines [id:da8893398885042718] 
\draw [color={denim}  ,draw opacity=0.5 ]   (94.15,50.65) -- (85,40.62) ;
%Straight Lines [id:da42992541066902357] 
\draw [color={denim}  ,draw opacity=0.5 ]   (99.13,50.67) -- (89.99,40.64) ;
%Straight Lines [id:da06016272435003489] 
\draw [color={denim}  ,draw opacity=0.5 ]   (84.17,50.37) -- (75.03,40.35) ;
%Straight Lines [id:da4981213952079808] 
\draw [color={denim}  ,draw opacity=0.5 ]   (89.16,50.39) -- (80.01,40.36) ;
%Straight Lines [id:da7287456739815834] 
\draw [color={denim}  ,draw opacity=0.5 ]   (79.14,50.11) -- (69.99,40.08) ;
%Straight Lines [id:da1716841379983325] 
\draw [color={denim}  ,draw opacity=0.5 ]   (68.71,50.33) -- (59.56,40.3) ;
%Straight Lines [id:da516445289034698] 
\draw [color={denim}  ,draw opacity=0.5 ]   (74.2,50.34) -- (65.05,40.32) ;
%Straight Lines [id:da19108720085568387] 
\draw [color={denim}  ,draw opacity=0.5 ]   (63.75,51) -- (59.13,46.5) ;
%Straight Lines [id:da3916499394065682] 
\draw [color={denim}  ,draw opacity=0.5 ]   (209.63,210.75) -- (200.15,200.72) ;
%Straight Lines [id:da14849644278690866] 
\draw [color={denim}  ,draw opacity=0.5 ]   (199.37,210.97) -- (190.22,200.94) ;
%Straight Lines [id:da480054404884968] 
\draw [color={denim}  ,draw opacity=0.5 ]   (204.36,210.98) -- (195.21,200.95) ;
%Straight Lines [id:da5816008783965672] 
\draw [color={denim}  ,draw opacity=0.5 ]   (194.33,210.7) -- (185.19,200.68) ;
%Straight Lines [id:da4665681838349083] 
\draw [color={denim}  ,draw opacity=0.5 ]   (184.41,210.92) -- (175.26,200.9) ;
%Straight Lines [id:da9111100484955967] 
\draw [color={denim}  ,draw opacity=0.5 ]   (189.4,210.94) -- (180.25,200.91) ;
%Straight Lines [id:da6664092786777638] 
\draw [color={denim}  ,draw opacity=0.5 ]   (178.87,210.66) -- (169.73,200.63) ;
%Straight Lines [id:da38747872439604514] 
\draw [color={denim}  ,draw opacity=0.5 ]   (168.95,210.88) -- (159.8,200.85) ;
%Straight Lines [id:da1091464545341072] 
\draw [color={denim}  ,draw opacity=0.5 ]   (173.94,210.89) -- (164.79,200.86) ;
%Straight Lines [id:da7578957950563772] 
\draw [color={denim}  ,draw opacity=0.5 ]   (163.91,210.61) -- (154.77,200.58) ;
%Straight Lines [id:da18179833492884212] 
\draw [color={denim}  ,draw opacity=0.5 ]   (153.99,210.83) -- (144.84,200.81) ;
%Straight Lines [id:da7187215515613472] 
\draw [color={denim}  ,draw opacity=0.5 ]   (158.98,210.85) -- (149.83,200.82) ;
%Straight Lines [id:da3752127855437537] 
\draw [color={denim}  ,draw opacity=0.5 ]   (148.7,210.82) -- (139.56,200.79) ;
%Straight Lines [id:da05359373721132288] 
\draw [color={denim}  ,draw opacity=0.5 ]   (138.78,211.04) -- (129.63,201.01) ;
%Straight Lines [id:da3401618276862941] 
\draw [color={denim}  ,draw opacity=0.5 ]   (143.76,211.05) -- (134.62,201.02) ;
%Straight Lines [id:da16192203195371047] 
\draw [color={denim}  ,draw opacity=0.5 ]   (133.74,210.77) -- (124.59,200.74) ;
%Straight Lines [id:da015004216547931382] 
\draw [color={denim}  ,draw opacity=0.5 ]   (123.82,210.99) -- (114.67,200.97) ;
%Straight Lines [id:da8217024770563819] 
\draw [color={denim}  ,draw opacity=0.5 ]   (128.8,211.01) -- (119.66,200.98) ;
%Straight Lines [id:da6508840877440396] 
\draw [color={denim}  ,draw opacity=0.5 ]   (118.28,210.73) -- (109.13,200.7) ;
%Straight Lines [id:da1147947803115954] 
\draw [color={denim}  ,draw opacity=0.5 ]   (108.36,210.95) -- (99.21,200.92) ;
%Straight Lines [id:da41375885260483236] 
\draw [color={denim}  ,draw opacity=0.5 ]   (113.34,210.96) -- (104.2,200.93) ;
%Straight Lines [id:da8582222085594853] 
\draw [color={denim}  ,draw opacity=0.5 ]   (103.32,210.68) -- (94.17,200.65) ;
%Straight Lines [id:da4133381673758856] 
\draw [color={denim}  ,draw opacity=0.5 ]   (93.4,210.9) -- (84.25,200.87) ;
%Straight Lines [id:da0646034045738908] 
\draw [color={denim}  ,draw opacity=0.5 ]   (98.38,210.92) -- (89.24,200.89) ;
%Straight Lines [id:da8851708899090724] 
\draw [color={denim}  ,draw opacity=0.5 ]   (83.42,210.62) -- (74.28,200.6) ;
%Straight Lines [id:da22648951965692854] 
\draw [color={denim}  ,draw opacity=0.5 ]   (88.41,210.64) -- (79.26,200.61) ;
%Straight Lines [id:da16298686845463073] 
\draw [color={denim}  ,draw opacity=0.5 ]   (78.39,210.36) -- (69.24,200.33) ;
%Straight Lines [id:da1227467130791684] 
\draw [color={denim}  ,draw opacity=0.5 ]   (68.9,210.23) -- (59.75,200.2) ;
%Straight Lines [id:da7197780081015817] 
\draw [color={denim}  ,draw opacity=0.5 ]   (73.63,210.5) -- (64.8,200.82) ;
%Straight Lines [id:da7047752231655877] 
\draw [color={denim}  ,draw opacity=0.5 ]   (64,210.25) -- (62.83,208.89) -- (59.75,205.3) ;
%Straight Lines [id:da2693035941891506] 
\draw [color={denim}  ,draw opacity=0.5 ]   (210.38,205.25) -- (205.75,200.63) ;
%Straight Lines [id:da37415739117851166] 
\draw [color={denim}  ,draw opacity=0.5 ]   (50,49.8) -- (60,59) ;
%Straight Lines [id:da8703642813087511] 
\draw [color={denim}  ,draw opacity=0.5 ]   (49.75,59.75) -- (59.75,68.95) ;
%Straight Lines [id:da25380000559698535] 
\draw [color={denim}  ,draw opacity=0.5 ]   (49.75,54.75) -- (59.75,63.95) ;
%Straight Lines [id:da014488532027728152] 
\draw [color={denim}  ,draw opacity=0.5 ]   (50,64.8) -- (60,74) ;
%Straight Lines [id:da7287154168620511] 
\draw [color={denim}  ,draw opacity=0.5 ]   (49.75,75) -- (59.75,84.2) ;
%Straight Lines [id:da850921318366465] 
\draw [color={denim}  ,draw opacity=0.5 ]   (49.75,69.75) -- (59.75,78.95) ;
%Straight Lines [id:da7190213032465287] 
\draw [color={denim}  ,draw opacity=0.5 ]   (50,80.3) -- (60,89.5) ;
%Straight Lines [id:da48798300193534416] 
\draw [color={denim}  ,draw opacity=0.5 ]   (49.75,90.25) -- (59.75,99.45) ;
%Straight Lines [id:da6096271690573143] 
\draw [color={denim}  ,draw opacity=0.5 ]   (49.75,85) -- (59.75,94.2) ;
%Straight Lines [id:da834179426282011] 
\draw [color={denim}  ,draw opacity=0.5 ]   (50,95.3) -- (60,104.5) ;
%Straight Lines [id:da9822312327704144] 
\draw [color={denim}  ,draw opacity=0.5 ]   (49.5,105.5) -- (59.5,114.7) ;
%Straight Lines [id:da8856184174474138] 
\draw [color={denim}  ,draw opacity=0.5 ]   (49.75,100.25) -- (59.75,109.45) ;
%Straight Lines [id:da7738257496914653] 
\draw [color={denim}  ,draw opacity=0.5 ]   (49.75,110.55) -- (59.75,119.75) ;
%Straight Lines [id:da29462581116621656] 
\draw [color={denim}  ,draw opacity=0.5 ]   (49.5,120.5) -- (59.5,129.7) ;
%Straight Lines [id:da08515222031774505] 
\draw [color={denim}  ,draw opacity=0.5 ]   (49.5,115.5) -- (59.5,124.7) ;
%Straight Lines [id:da6296424863685477] 
\draw [color={denim}  ,draw opacity=0.5 ]   (49.75,125.55) -- (59.75,134.75) ;
%Straight Lines [id:da8311487884398499] 
\draw [color={denim}  ,draw opacity=0.5 ]   (50,136) -- (60,145.2) ;
%Straight Lines [id:da5054446405610953] 
\draw [color={denim}  ,draw opacity=0.5 ]   (49.5,130.5) -- (59.5,139.7) ;
%Straight Lines [id:da5432208621392958] 
\draw [color={denim}  ,draw opacity=0.5 ]   (49.75,141.05) -- (59.75,150.25) ;
%Straight Lines [id:da36290315844855936] 
\draw [color={denim}  ,draw opacity=0.5 ]   (49.5,151) -- (59.5,160.2) ;
%Straight Lines [id:da4500076345946531] 
\draw [color={denim}  ,draw opacity=0.5 ]   (49.5,146) -- (59.5,155.2) ;
%Straight Lines [id:da4904233600700443] 
\draw [color={denim}  ,draw opacity=0.5 ]   (49.75,156.05) -- (59.75,165.25) ;
%Straight Lines [id:da27249017893015703] 
\draw [color={denim}  ,draw opacity=0.5 ]   (49.5,166.25) -- (59.5,175.45) ;
%Straight Lines [id:da4211146028314845] 
\draw [color={denim}  ,draw opacity=0.5 ]   (49.5,161) -- (59.5,170.2) ;
%Straight Lines [id:da7605569272928743] 
\draw [color={denim}  ,draw opacity=0.5 ]   (49.75,176) -- (59.75,185.2) ;
%Straight Lines [id:da779084256878755] 
\draw [color={denim}  ,draw opacity=0.5 ]   (49.75,171) -- (59.75,180.2) ;
%Straight Lines [id:da5197529200510105] 
\draw [color={denim}  ,draw opacity=0.5 ]   (50,181.05) -- (60,190.25) ;
%Straight Lines [id:da9878071242062256] 
\draw [color={denim}  ,draw opacity=0.5 ]   (49.75,191) -- (59.75,200.2) ;
%Straight Lines [id:da8437225437090057] 
\draw [color={denim}  ,draw opacity=0.5 ]   (49.75,186) -- (59.75,195.2) ;
%Straight Lines [id:da5341847705978953] 
\draw [color={denim}  ,draw opacity=0.5 ]   (50.13,196.5) -- (54.75,200.25) ;
%Straight Lines [id:da7419899284159641] 
\draw [color={denim}  ,draw opacity=0.5 ]   (55.88,50.25) -- (59.75,53.75) ;
%Rounded Rect [id:dp46867517823252336] 
\draw  [color={rgb, 255:red, 0; green, 0; blue, 0 }  ,draw opacity=1 ] (154.63,67.01) .. controls (154.63,66.18) and (155.3,65.5) .. (156.14,65.5) -- (193.36,65.5) .. controls (194.2,65.5) and (194.88,66.18) .. (194.88,67.01) -- (194.88,84.15) .. controls (194.88,84.99) and (194.2,85.67) .. (193.36,85.67) -- (156.14,85.67) .. controls (155.3,85.67) and (154.63,84.99) .. (154.63,84.15) -- cycle ;
%Straight Lines [id:da4661901878722403] 
\draw [color={denim}  ,draw opacity=0.5 ]   (159.75,70.75) -- (169.75,80.75) ;
%Straight Lines [id:da8634775717218319] 
\draw [color={denim}  ,draw opacity=0.5 ]   (164.5,70.5) -- (169.88,76) ;
%Straight Lines [id:da963510161967424] 
\draw [color={denim}  ,draw opacity=0.5 ]   (159.5,75.5) -- (164.88,81) ;
%Straight Lines [id:da8197176995936073] 
\draw [color={denim}  ,draw opacity=0.5 ]   (214,120.5) -- (220.25,126.2) ;
%Straight Lines [id:da35959297209005214] 
\draw [color={denim}  ,draw opacity=0.5 ]   (214.75,50.05) -- (220.5,55.25) ;
%Straight Lines [id:da7793392289127927] 
\draw [color={denim}  ,draw opacity=0.5 ]   (205.5,40.3) -- (209.75,45.25) ;
%Straight Lines [id:da4894862057448417] 
\draw [color={denim}  ,draw opacity=0.5 ]   (209.25,40.05) -- (210.25,41.25) ;
%Straight Lines [id:da8120271513133643] 
\draw [color={denim}  ,draw opacity=0.5 ]   (219.75,50.3) -- (220.75,51.25) ;
%Shape: Path Data [id:dp43165472823785733] 
\draw  [draw opacity=0][fill={denim}  ,fill opacity=0.2 ] (60.24,30.8) -- (60.24,30.79) -- (129.39,30.96) -- (129.39,31.21) -- (209.45,31.21) -- (209.45,31.3) .. controls (209.59,31.3) and (209.74,31.29) .. (209.88,31.29) .. controls (209.89,31.29) and (209.9,31.29) .. (209.92,31.29) -- (209.92,31.29) .. controls (210.02,31.29) and (210.13,31.29) .. (210.24,31.29) .. controls (221.05,31.29) and (229.82,40.1) .. (229.82,50.97) .. controls (229.82,51.03) and (229.82,51.08) .. (229.82,51.14) -- (229.82,51.14) -- (229.58,93.7) -- (229.58,119.79) -- (229.44,119.79) -- (228.99,200.42) -- (228.87,200.42) .. controls (228.87,200.51) and (228.88,200.6) .. (228.89,200.69) -- (228.67,200.69) .. controls (228.67,200.81) and (228.68,200.93) .. (228.68,201.04) .. controls (228.68,211.72) and (219.99,220.38) .. (209.28,220.38) .. controls (209.25,220.38) and (209.21,220.38) .. (209.17,220.37) -- (209.17,220.38) -- (203.35,220.34) -- (60.3,220.67) -- (60.3,220.64) .. controls (59.99,220.66) and (59.68,220.67) .. (59.37,220.67) .. controls (48.74,220.67) and (40.12,211.98) .. (40.12,201.25) .. controls (40.12,201.18) and (40.12,201.1) .. (40.12,201.03) -- (40.12,201.03) -- (39.96,132.14) -- (40.3,132.14) -- (40.5,50.85) -- (40.62,50.85) .. controls (40.54,50.15) and (40.5,49.44) .. (40.5,48.72) .. controls (40.5,48.62) and (40.5,48.52) .. (40.5,48.43) .. controls (41.71,38.58) and (50.07,30.95) .. (60.2,30.95) .. controls (60.2,30.95) and (60.2,30.95) .. (60.2,30.95) -- (60.2,30.8) -- (60.24,30.8) -- cycle (59.79,201.03) -- (59.37,201.03) -- (59.37,201.25) -- (59.79,201.25) -- (59.79,201.03) -- cycle (209.92,50.85) -- (210.01,50.85) -- (210.01,50.69) -- (209.92,50.69) -- (209.92,50.85) -- cycle (209.28,200.69) -- (209.28,200.69) -- (209.39,200.69) -- (209.39,200.69) -- (209.28,200.69) -- cycle (209.92,50.85) -- (80.04,50.85) .. controls (80.04,50.87) and (80.04,50.88) .. (80.04,50.89) -- (60.2,50.89) -- (60.15,70.83) .. controls (60.15,70.83) and (60.15,70.83) .. (60.15,70.83) -- (59.83,200.02) -- (59.79,200.06) -- (59.79,201.03) -- (140.13,200.85) -- (140.13,200.39) -- (199.14,200.71) -- (209.28,200.69) -- (209.92,87.7) -- (209.92,50.85) -- cycle ;
%Straight Lines [id:da36015868325790246] 
\draw [color={denim}  ,draw opacity=1 ]   (61.02,30.63) -- (209.78,30.79) ;
%Shape: Arc [id:dp2772439029742704] 
\draw  [draw opacity=0] (209.14,30.81) .. controls (209.39,30.8) and (209.66,30.79) .. (209.92,30.79) .. controls (220.57,30.79) and (229.24,39.41) .. (229.52,50.16) -- (209.92,50.69) -- cycle ; \draw  [color={denim}  ,draw opacity=1 ] (209.14,30.81) .. controls (209.39,30.8) and (209.66,30.79) .. (209.92,30.79) .. controls (220.57,30.79) and (229.24,39.41) .. (229.52,50.16) ;  
%Shape: Arc [id:dp8114325388675698] 
\draw  [draw opacity=0] (40.5,50.85) .. controls (40.49,50.73) and (40.49,50.61) .. (40.49,50.48) .. controls (40.49,39.49) and (49.09,30.59) .. (59.7,30.59) .. controls (60.14,30.59) and (60.58,30.6) .. (61.02,30.63) -- (59.7,50.48) -- cycle ; \draw  [color={denim}  ,draw opacity=1 ] (40.5,50.85) .. controls (40.49,50.73) and (40.49,50.61) .. (40.49,50.48) .. controls (40.49,39.49) and (49.09,30.59) .. (59.7,30.59) .. controls (60.14,30.59) and (60.58,30.6) .. (61.02,30.63) ;  
%Straight Lines [id:da701636105793817] 
\draw [color={denim}  ,draw opacity=1 ]   (228.3,201.08) -- (229.52,50.21) ;
%Straight Lines [id:da4242406393737195] 
\draw [color={denim}  ,draw opacity=1 ]   (40.17,200.88) -- (40.46,50.57) ;
%Shape: Arc [id:dp31688739829491985] 
\draw  [draw opacity=0] (228.37,200.15) .. controls (228.38,200.38) and (228.38,200.61) .. (228.38,200.83) .. controls (228.38,211.53) and (220.51,220.25) .. (210.65,220.66) -- (209.92,200.83) -- cycle ; \draw  [color={denim}  ,draw opacity=1 ] (228.37,200.15) .. controls (228.38,200.38) and (228.38,200.61) .. (228.38,200.83) .. controls (228.38,211.53) and (220.51,220.25) .. (210.65,220.66) ;  
%Shape: Arc [id:dp9274192794190577] 
\draw  [draw opacity=0] (58.72,220.67) .. controls (48.41,220.4) and (40.13,211.62) .. (40.13,200.84) .. controls (40.13,200.8) and (40.13,200.77) .. (40.13,200.73) -- (59.21,200.84) -- cycle ; \draw  [color={denim}  ,draw opacity=1 ] (58.72,220.67) .. controls (48.41,220.4) and (40.13,211.62) .. (40.13,200.84) .. controls (40.13,200.8) and (40.13,200.77) .. (40.13,200.73) ;  
%Straight Lines [id:da7778452467995949] 
\draw [color={denim}  ,draw opacity=1 ]   (58.43,220.67) -- (210.65,220.66) ;
%Shape: Rectangle [id:dp5220774784888251] 
\draw   (59.83,50.75) -- (209.92,50.75) -- (209.92,200.83) -- (59.83,200.83) -- cycle ;
%Shape: Path Data [id:dp24133627747919872] 
\draw  [draw opacity=0][fill={denim}  ,fill opacity=0.2 ] (260.16,30.07) -- (260.16,30.07) -- (329.31,30.24) -- (329.31,30.49) -- (409.37,30.49) -- (409.37,30.58) .. controls (409.51,30.57) and (409.65,30.57) .. (409.79,30.56) .. controls (409.81,30.56) and (409.82,30.56) .. (409.83,30.56) -- (409.83,30.56) .. controls (409.94,30.56) and (410.05,30.56) .. (410.15,30.56) .. controls (420.97,30.56) and (429.74,39.37) .. (429.74,50.25) .. controls (429.74,50.3) and (429.74,50.36) .. (429.73,50.41) -- (429.74,50.41) -- (429.5,92.98) -- (429.5,119.06) -- (429.35,119.06) -- (429.35,119.42) -- (409.65,119.42) -- (409.83,86.97) -- (409.83,50.13) -- (279.95,50.13) .. controls (279.95,50.14) and (279.95,50.15) .. (279.95,50.17) -- (260.33,50.17) -- (246.17,35.98) .. controls (249.75,32.42) and (254.67,30.23) .. (260.11,30.23) .. controls (260.11,30.23) and (260.11,30.23) .. (260.11,30.23) -- (260.11,30.07) -- (260.16,30.07) -- cycle (409.83,50.13) -- (409.92,50.13) -- (409.92,49.96) -- (409.83,49.96) -- (409.83,50.13) -- cycle ;
%Shape: Grid [id:dp2934542325905436] 
\draw  [draw opacity=0] (260.5,50.22) -- (410.5,50.22) -- (410.5,200.22) -- (260.5,200.22) -- cycle ; \draw  [color={rgb, 255:red, 155; green, 155; blue, 155 }  ,draw opacity=0.25 ] (270.5,50.22) -- (270.5,200.22)(280.5,50.22) -- (280.5,200.22)(290.5,50.22) -- (290.5,200.22)(300.5,50.22) -- (300.5,200.22)(310.5,50.22) -- (310.5,200.22)(320.5,50.22) -- (320.5,200.22)(330.5,50.22) -- (330.5,200.22)(340.5,50.22) -- (340.5,200.22)(350.5,50.22) -- (350.5,200.22)(360.5,50.22) -- (360.5,200.22)(370.5,50.22) -- (370.5,200.22)(380.5,50.22) -- (380.5,200.22)(390.5,50.22) -- (390.5,200.22)(400.5,50.22) -- (400.5,200.22) ; \draw  [color={rgb, 255:red, 155; green, 155; blue, 155 }  ,draw opacity=0.25 ] (260.5,60.22) -- (410.5,60.22)(260.5,70.22) -- (410.5,70.22)(260.5,80.22) -- (410.5,80.22)(260.5,90.22) -- (410.5,90.22)(260.5,100.22) -- (410.5,100.22)(260.5,110.22) -- (410.5,110.22)(260.5,120.22) -- (410.5,120.22)(260.5,130.22) -- (410.5,130.22)(260.5,140.22) -- (410.5,140.22)(260.5,150.22) -- (410.5,150.22)(260.5,160.22) -- (410.5,160.22)(260.5,170.22) -- (410.5,170.22)(260.5,180.22) -- (410.5,180.22)(260.5,190.22) -- (410.5,190.22) ; \draw  [color={rgb, 255:red, 155; green, 155; blue, 155 }  ,draw opacity=0.25 ] (260.5,50.22) -- (410.5,50.22) -- (410.5,200.22) -- (260.5,200.22) -- cycle ;
%Shape: Square [id:dp314166732665379] 
\draw   (260.33,50.17) -- (410.42,50.17) -- (410.42,200.25) -- (260.33,200.25) -- cycle ;
%Straight Lines [id:da7659367171118718] 
\draw [color={denim}  ,draw opacity=0.5 ]   (410.75,50.22) -- (420.75,59.42) ;
%Straight Lines [id:da6510697839942092] 
\draw [color={denim}  ,draw opacity=0.5 ]   (410.5,60.17) -- (420.5,69.37) ;
%Straight Lines [id:da6276923129493011] 
\draw [color={denim}  ,draw opacity=0.5 ]   (410.5,55.17) -- (420.5,64.37) ;
%Straight Lines [id:da4351004853348148] 
\draw [color={denim}  ,draw opacity=0.5 ]   (410.75,65.22) -- (420.75,74.42) ;
%Straight Lines [id:da3876185030246686] 
\draw [color={denim}  ,draw opacity=0.5 ]   (410.5,75.17) -- (420.5,84.37) ;
%Straight Lines [id:da6674565852265526] 
\draw [color={denim}  ,draw opacity=0.5 ]   (410.5,70.17) -- (420.5,79.37) ;
%Straight Lines [id:da36771757517682513] 
\draw [color={denim}  ,draw opacity=0.5 ]   (410.75,80.72) -- (420.75,89.92) ;
%Straight Lines [id:da7521837001908436] 
\draw [color={denim}  ,draw opacity=0.5 ]   (410.5,90.67) -- (420.5,99.87) ;
%Straight Lines [id:da7462606586761806] 
\draw [color={denim}  ,draw opacity=0.5 ]   (410.5,85.67) -- (420.5,94.87) ;
%Straight Lines [id:da28156032765677463] 
\draw [color={denim}  ,draw opacity=0.5 ]   (410.75,95.72) -- (420.75,104.92) ;
%Straight Lines [id:da5472546196656172] 
\draw [color={denim}  ,draw opacity=0.5 ]   (410.5,105.67) -- (420.5,114.87) ;
%Straight Lines [id:da4159335710397434] 
\draw [color={denim}  ,draw opacity=0.5 ]   (410.5,100.67) -- (420.5,109.87) ;
%Straight Lines [id:da18656978459515772] 
\draw [color={denim}  ,draw opacity=0.5 ]   (410.5,110.17) -- (420.5,119.37) ;
%Straight Lines [id:da7228869184448987] 
\draw [color={shamrockgreen}  ,draw opacity=0.5 ]   (410.25,120.92) -- (420.25,130.12) ;
%Straight Lines [id:da9138694402045633] 
\draw [color={denim}  ,draw opacity=0.5 ]   (410.25,115.92) -- (414.5,119.92) ;
%Straight Lines [id:da04480699986031267] 
\draw [color={shamrockgreen}  ,draw opacity=0.5 ]   (410.5,125.97) -- (420.5,135.17) ;
%Straight Lines [id:da22988985121150107] 
\draw [color={shamrockgreen}  ,draw opacity=0.5 ]   (410.25,135.92) -- (420.25,145.12) ;
%Straight Lines [id:da5262920517006697] 
\draw [color={shamrockgreen}  ,draw opacity=0.5 ]   (410.25,130.92) -- (420.25,140.12) ;
%Straight Lines [id:da10517795689102272] 
\draw [color={shamrockgreen}  ,draw opacity=0.5 ]   (410.5,141.47) -- (420.5,150.67) ;
%Straight Lines [id:da8652075832534545] 
\draw [color={shamrockgreen}  ,draw opacity=0.5 ]   (410.25,151.42) -- (420.25,160.62) ;
%Straight Lines [id:da5952800959183444] 
\draw [color={shamrockgreen}  ,draw opacity=0.5 ]   (410.25,146.42) -- (420.25,155.62) ;
%Straight Lines [id:da7096147465983154] 
\draw [color={shamrockgreen}  ,draw opacity=0.5 ]   (410.5,156.47) -- (420.5,165.67) ;
%Straight Lines [id:da2812641141025789] 
\draw [color={shamrockgreen}  ,draw opacity=0.5 ]   (410.25,166.42) -- (420.25,175.62) ;
%Straight Lines [id:da32015021720963466] 
\draw [color={shamrockgreen}  ,draw opacity=0.5 ]   (410.25,161.42) -- (420.25,170.62) ;
%Straight Lines [id:da21754040488490833] 
\draw [color={shamrockgreen}  ,draw opacity=0.5 ]   (410.5,176.42) -- (420.5,185.62) ;
%Straight Lines [id:da4780548942000711] 
\draw [color={shamrockgreen}  ,draw opacity=0.5 ]   (410.5,171.42) -- (420.5,180.62) ;
%Straight Lines [id:da12682073453834364] 
\draw [color={shamrockgreen}  ,draw opacity=0.5 ]   (410.75,181.47) -- (420.75,190.67) ;
%Straight Lines [id:da6357723441165608] 
\draw [color={shamrockgreen}  ,draw opacity=0.5 ]   (410.5,191.42) -- (420.13,200.42) ;
%Straight Lines [id:da2263152279649685] 
\draw [color={shamrockgreen}  ,draw opacity=0.5 ]   (410.5,186.42) -- (420.5,195.62) ;
%Straight Lines [id:da9658494500610555] 
\draw [color={shamrockgreen}  ,draw opacity=0.5 ]   (410.75,196.92) -- (414.63,200.92) ;
%Straight Lines [id:da7736166755894005] 
\draw [color={denim}  ,draw opacity=0.5 ]   (410.55,49.91) -- (401.4,39.89) ;
%Straight Lines [id:da8175737708930739] 
\draw [color={denim}  ,draw opacity=0.5 ]   (400.62,50.13) -- (391.47,40.11) ;
%Straight Lines [id:da4911079904967903] 
\draw [color={denim}  ,draw opacity=0.5 ]   (405.61,50.15) -- (396.46,40.12) ;
%Straight Lines [id:da5784355666870553] 
\draw [color={denim}  ,draw opacity=0.5 ]   (395.58,49.87) -- (386.44,39.84) ;
%Straight Lines [id:da3659101370743032] 
\draw [color={denim}  ,draw opacity=0.5 ]   (385.66,50.09) -- (376.51,40.06) ;
%Straight Lines [id:da9818288324774307] 
\draw [color={denim}  ,draw opacity=0.5 ]   (390.65,50.1) -- (381.5,40.08) ;
%Straight Lines [id:da8176763614865523] 
\draw [color={denim}  ,draw opacity=0.5 ]   (380.12,49.82) -- (370.98,39.8) ;
%Straight Lines [id:da789757952166102] 
\draw [color={denim}  ,draw opacity=0.5 ]   (370.2,50.04) -- (361.05,40.02) ;
%Straight Lines [id:da26410512294891975] 
\draw [color={denim}  ,draw opacity=0.5 ]   (375.19,50.06) -- (366.04,40.03) ;
%Straight Lines [id:da11228651789040045] 
\draw [color={denim}  ,draw opacity=0.5 ]   (365.16,49.78) -- (356.02,39.75) ;
%Straight Lines [id:da36662507195728744] 
\draw [color={denim}  ,draw opacity=0.5 ]   (355.24,50) -- (346.09,39.97) ;
%Straight Lines [id:da898336413770435] 
\draw [color={denim}  ,draw opacity=0.5 ]   (360.23,50.01) -- (351.08,39.99) ;
%Straight Lines [id:da2940965537373925] 
\draw [color={denim}  ,draw opacity=0.5 ]   (349.95,49.98) -- (340.81,39.96) ;
%Straight Lines [id:da18644624996920212] 
\draw [color={denim}  ,draw opacity=0.5 ]   (340.03,50.2) -- (330.88,40.18) ;
%Straight Lines [id:da8299381362848626] 
\draw [color={denim}  ,draw opacity=0.5 ]   (345.01,50.22) -- (335.87,40.19) ;
%Straight Lines [id:da22822949366195777] 
\draw [color={denim}  ,draw opacity=0.5 ]   (334.99,49.94) -- (325.84,39.91) ;
%Straight Lines [id:da3760050005990292] 
\draw [color={denim}  ,draw opacity=0.5 ]   (325.07,50.16) -- (315.92,40.13) ;
%Straight Lines [id:da08000508618710733] 
\draw [color={denim}  ,draw opacity=0.5 ]   (330.05,50.17) -- (320.91,40.15) ;
%Straight Lines [id:da14129535111941904] 
\draw [color={denim}  ,draw opacity=0.5 ]   (319.53,49.89) -- (310.38,39.87) ;
%Straight Lines [id:da44920328384293806] 
\draw [color={denim}  ,draw opacity=0.5 ]   (309.61,50.11) -- (300.46,40.09) ;
%Straight Lines [id:da5506354394079329] 
\draw [color={denim}  ,draw opacity=0.5 ]   (314.59,50.13) -- (305.45,40.1) ;
%Straight Lines [id:da10345541148308146] 
\draw [color={denim}  ,draw opacity=0.5 ]   (304.57,49.85) -- (295.42,39.82) ;
%Straight Lines [id:da6727456127759501] 
\draw [color={denim}  ,draw opacity=0.5 ]   (294.65,50.07) -- (285.5,40.04) ;
%Straight Lines [id:da7687122163170721] 
\draw [color={denim}  ,draw opacity=0.5 ]   (299.63,50.08) -- (290.49,40.06) ;
%Straight Lines [id:da7214531676187297] 
\draw [color={denim}  ,draw opacity=0.5 ]   (284.67,49.79) -- (275.53,39.76) ;
%Straight Lines [id:da9832076466558728] 
\draw [color={denim}  ,draw opacity=0.5 ]   (289.66,49.8) -- (280.51,39.78) ;
%Straight Lines [id:da7029062439465927] 
\draw [color={denim}  ,draw opacity=0.5 ]   (279.64,49.52) -- (270.49,39.5) ;
%Straight Lines [id:da553460809368356] 
\draw [color={denim}  ,draw opacity=0.5 ]   (269.21,49.74) -- (260.06,39.72) ;
%Straight Lines [id:da19242847669525154] 
\draw [color={denim}  ,draw opacity=0.5 ]   (274.7,49.76) -- (265.55,39.73) ;
%Straight Lines [id:da27915799413288833] 
\draw [color={denim}  ,draw opacity=0.5 ]   (264.25,50.42) -- (259.63,45.92) ;
%Straight Lines [id:da6891599625103695] 
\draw [color={byzantium}  ,draw opacity=0.5 ]   (410.13,210.17) -- (400.65,200.14) ;
%Straight Lines [id:da43520827239917215] 
\draw [color={byzantium}  ,draw opacity=0.5 ]   (399.87,210.38) -- (390.72,200.36) ;
%Straight Lines [id:da8575234434170025] 
\draw [color={byzantium}  ,draw opacity=0.5 ]   (404.86,210.4) -- (395.71,200.37) ;
%Straight Lines [id:da4724607226018156] 
\draw [color={byzantium}  ,draw opacity=0.5 ]   (394.83,210.12) -- (385.69,200.09) ;
%Straight Lines [id:da10645565027786041] 
\draw [color={byzantium}  ,draw opacity=0.5 ]   (384.91,210.34) -- (375.76,200.31) ;
%Straight Lines [id:da0894510100296333] 
\draw [color={byzantium}  ,draw opacity=0.5 ]   (389.9,210.35) -- (380.75,200.33) ;
%Straight Lines [id:da28404572842158693] 
\draw [color={byzantium}  ,draw opacity=0.5 ]   (379.37,210.07) -- (370.23,200.05) ;
%Straight Lines [id:da21873502884103524] 
\draw [color={byzantium}  ,draw opacity=0.5 ]   (369.45,210.29) -- (360.3,200.27) ;
%Straight Lines [id:da9907893552087681] 
\draw [color={byzantium}  ,draw opacity=0.5 ]   (374.44,210.31) -- (365.29,200.28) ;
%Straight Lines [id:da4080701629508745] 
\draw [color={byzantium}  ,draw opacity=0.5 ]   (364.41,210.03) -- (355.27,200) ;
%Straight Lines [id:da641802930095642] 
\draw [color={byzantium}  ,draw opacity=0.5 ]   (354.49,210.25) -- (345.34,200.22) ;
%Straight Lines [id:da641008783195927] 
\draw [color={byzantium}  ,draw opacity=0.5 ]   (359.48,210.26) -- (350.33,200.24) ;
%Straight Lines [id:da6978155015678866] 
\draw [color={byzantium}  ,draw opacity=0.5 ]   (349.2,210.23) -- (340.06,200.21) ;
%Straight Lines [id:da10883629945163054] 
\draw [color={byzantium}  ,draw opacity=0.5 ]   (339.28,210.45) -- (330.13,200.43) ;
%Straight Lines [id:da351963275757357] 
\draw [color={byzantium}  ,draw opacity=0.5 ]   (344.26,210.47) -- (335.12,200.44) ;
%Straight Lines [id:da7238599027190344] 
\draw [color={byzantium}  ,draw opacity=0.5 ]   (334.24,210.19) -- (325.09,200.16) ;
%Straight Lines [id:da14679286198386943] 
\draw [color={byzantium}  ,draw opacity=0.5 ]   (324.32,210.41) -- (315.17,200.38) ;
%Straight Lines [id:da5468833524975325] 
\draw [color={byzantium}  ,draw opacity=0.5 ]   (329.3,210.42) -- (320.16,200.4) ;
%Straight Lines [id:da44854162413583465] 
\draw [color={byzantium}  ,draw opacity=0.5 ]   (318.78,210.14) -- (309.63,200.12) ;
%Straight Lines [id:da33905430516312896] 
\draw [color={byzantium}  ,draw opacity=0.5 ]   (308.86,210.36) -- (299.71,200.34) ;
%Straight Lines [id:da3244804698313719] 
\draw [color={byzantium}  ,draw opacity=0.5 ]   (313.84,210.38) -- (304.7,200.35) ;
%Straight Lines [id:da5777219041355182] 
\draw [color={byzantium}  ,draw opacity=0.5 ]   (303.82,210.1) -- (294.67,200.07) ;
%Straight Lines [id:da9948747448012218] 
\draw [color={byzantium}  ,draw opacity=0.5 ]   (293.9,210.32) -- (284.75,200.29) ;
%Straight Lines [id:da2383337394982119] 
\draw [color={byzantium}  ,draw opacity=0.5 ]   (298.88,210.33) -- (289.74,200.31) ;
%Straight Lines [id:da32095915898648397] 
\draw [color={byzantium}  ,draw opacity=0.5 ]   (283.92,210.04) -- (274.78,200.01) ;
%Straight Lines [id:da6659208262880123] 
\draw [color={byzantium}  ,draw opacity=0.5 ]   (288.91,210.05) -- (279.76,200.03) ;
%Straight Lines [id:da7229845827793053] 
\draw [color={byzantium}  ,draw opacity=0.5 ]   (278.89,209.77) -- (269.74,199.75) ;
%Straight Lines [id:da7912168325300053] 
\draw [color={byzantium}  ,draw opacity=0.5 ]   (269.4,209.64) -- (260.25,199.62) ;
%Straight Lines [id:da6173782262866117] 
\draw [color={byzantium}  ,draw opacity=0.5 ]   (274.13,209.92) -- (265.3,200.23) ;
%Straight Lines [id:da31239987698798344] 
\draw [color={byzantium}  ,draw opacity=0.5 ]   (410.88,204.67) -- (406.25,200.05) ;
%Straight Lines [id:da3301724816906815] 
\draw [color={shamrockgreen}  ,draw opacity=0.51 ]   (250.5,49.22) -- (260.5,58.42) ;
%Straight Lines [id:da4101172031407887] 
\draw [color={shamrockgreen}  ,draw opacity=0.51 ]   (250.25,59.17) -- (260.25,68.37) ;
%Straight Lines [id:da6539988047865604] 
\draw [color={shamrockgreen}  ,draw opacity=0.51 ]   (250.25,54.17) -- (260.25,63.37) ;
%Straight Lines [id:da9143403943223125] 
\draw [color={shamrockgreen}  ,draw opacity=0.51 ]   (250.5,64.22) -- (260.5,73.42) ;
%Straight Lines [id:da5234031418874026] 
\draw [color={shamrockgreen}  ,draw opacity=0.51 ]   (250.25,74.42) -- (260.25,83.62) ;
%Straight Lines [id:da05441312585340641] 
\draw [color={shamrockgreen}  ,draw opacity=0.51 ]   (250.25,69.17) -- (260.25,78.37) ;
%Straight Lines [id:da793270738330121] 
\draw [color={shamrockgreen}  ,draw opacity=0.51 ]   (250.5,79.72) -- (260.5,88.92) ;
%Straight Lines [id:da6373458382817443] 
\draw [color={shamrockgreen}  ,draw opacity=0.51 ]   (250.25,89.67) -- (260.25,98.87) ;
%Straight Lines [id:da8304955645612395] 
\draw [color={shamrockgreen}  ,draw opacity=0.51 ]   (250.25,84.42) -- (260.25,93.62) ;
%Straight Lines [id:da32505108626529067] 
\draw [color={shamrockgreen}  ,draw opacity=0.51 ]   (250.5,94.72) -- (260.5,103.92) ;
%Straight Lines [id:da9722853299647627] 
\draw [color={shamrockgreen}  ,draw opacity=0.51 ]   (250,104.92) -- (260,114.12) ;
%Straight Lines [id:da2314428765944161] 
\draw [color={shamrockgreen}  ,draw opacity=0.51 ]   (250.25,99.67) -- (260.25,108.87) ;
%Straight Lines [id:da22544347980635693] 
\draw [color={shamrockgreen}  ,draw opacity=0.51 ]   (250.25,109.97) -- (260.25,119.17) ;
%Straight Lines [id:da42455845240081125] 
\draw [color={shamrockgreen}  ,draw opacity=0.51 ]   (250,119.92) -- (260,129.12) ;
%Straight Lines [id:da8777203213741758] 
\draw [color={shamrockgreen}  ,draw opacity=0.51 ]   (250,114.92) -- (260,124.12) ;
%Straight Lines [id:da7144312522331873] 
\draw [color={shamrockgreen}  ,draw opacity=0.51 ]   (250.25,124.97) -- (260.25,134.17) ;
%Straight Lines [id:da8029436653073065] 
\draw [color={shamrockgreen}  ,draw opacity=0.51 ]   (250.5,135.42) -- (260.5,144.62) ;
%Straight Lines [id:da7607099683166241] 
\draw [color={shamrockgreen}  ,draw opacity=0.51 ]   (250,129.92) -- (260,139.12) ;
%Straight Lines [id:da6407918535310042] 
\draw [color={shamrockgreen}  ,draw opacity=0.51 ]   (250.25,140.47) -- (260.25,149.67) ;
%Straight Lines [id:da855576841946758] 
\draw [color={shamrockgreen}  ,draw opacity=0.51 ]   (250,150.42) -- (260,159.62) ;
%Straight Lines [id:da21818295532843757] 
\draw [color={shamrockgreen}  ,draw opacity=0.51 ]   (250,145.42) -- (260,154.62) ;
%Straight Lines [id:da9264328488916471] 
\draw [color={shamrockgreen}  ,draw opacity=0.51 ]   (250.25,155.47) -- (260.25,164.67) ;
%Straight Lines [id:da06516263371747932] 
\draw [color={shamrockgreen}  ,draw opacity=0.51 ]   (250,165.67) -- (260,174.87) ;
%Straight Lines [id:da43742174771326203] 
\draw [color={shamrockgreen}  ,draw opacity=0.51 ]   (250,160.42) -- (260,169.62) ;
%Straight Lines [id:da1717051661694431] 
\draw [color={shamrockgreen}  ,draw opacity=0.51 ]   (250.25,175.42) -- (260.25,184.62) ;
%Straight Lines [id:da5393249756290295] 
\draw [color={shamrockgreen}  ,draw opacity=0.51 ]   (250.25,170.42) -- (260.25,179.62) ;
%Straight Lines [id:da07947465299695389] 
\draw [color={shamrockgreen}  ,draw opacity=0.51 ]   (250.5,180.47) -- (260.5,189.67) ;
%Straight Lines [id:da8125826484982019] 
\draw [color={shamrockgreen}  ,draw opacity=0.51 ]   (250.25,190.42) -- (260.25,199.62) ;
%Straight Lines [id:da8234523639853091] 
\draw [color={shamrockgreen}  ,draw opacity=0.51 ]   (250.25,185.42) -- (260.25,194.62) ;
%Straight Lines [id:da9928544745254733] 
\draw [color={shamrockgreen}  ,draw opacity=0.51 ]   (250.63,195.92) -- (255.25,199.67) ;
%Straight Lines [id:da34684091255796146] 
\draw [color={shamrockgreen}  ,draw opacity=0.51 ]   (256.38,49.67) -- (260.25,53.17) ;
%Rounded Rect [id:dp7711346612930092] 
\draw  [color={rgb, 255:red, 0; green, 0; blue, 0 }  ,draw opacity=1 ] (355.13,67.94) .. controls (355.13,66.27) and (356.48,64.92) .. (358.14,64.92) -- (392.36,64.92) .. controls (394.02,64.92) and (395.38,66.27) .. (395.38,67.94) -- (395.38,121.65) .. controls (395.38,123.32) and (394.02,124.67) .. (392.36,124.67) -- (358.14,124.67) .. controls (356.48,124.67) and (355.13,123.32) .. (355.13,121.65) -- cycle ;
%Straight Lines [id:da6597945938329319] 
\draw [color={denim}  ,draw opacity=0.5 ]   (360.25,70.17) -- (370.25,80.17) ;
%Straight Lines [id:da3877192804374967] 
\draw [color={denim}  ,draw opacity=0.5 ]   (365,69.92) -- (370.38,75.42) ;
%Straight Lines [id:da16944299674860352] 
\draw [color={shamrockgreen}  ,draw opacity=0.5 ]   (360.25,90.17) -- (370.25,100.17) ;
%Straight Lines [id:da675172077945206] 
\draw [color={shamrockgreen}  ,draw opacity=0.5 ]   (365.38,90.42) -- (370.13,95.42) ;
%Straight Lines [id:da5945729164180005] 
\draw [color={shamrockgreen}  ,draw opacity=0.5 ]   (360.75,95.17) -- (365.13,99.92) ;
%Straight Lines [id:da12902308530885542] 
\draw [color={denim}  ,draw opacity=0.5 ]   (360,74.92) -- (365.38,80.42) ;
%Straight Lines [id:da937318961496602] 
\draw [color={shamrockgreen}  ,draw opacity=0.51 ]   (414.5,119.92) -- (420.75,125.62) ;
%Straight Lines [id:da36920446924297945] 
\draw [color={denim}  ,draw opacity=0.5 ]   (415.25,49.47) -- (421,54.67) ;
%Straight Lines [id:da9115926437925137] 
\draw [color={denim}  ,draw opacity=0.5 ]   (406,39.72) -- (410.25,44.67) ;
%Straight Lines [id:da9912790405180716] 
\draw [color={denim}  ,draw opacity=0.5 ]   (409.75,39.47) -- (410.75,40.67) ;
%Straight Lines [id:da19418825449983568] 
\draw [color={denim}  ,draw opacity=0.5 ]   (420.25,49.72) -- (421.25,50.67) ;
%Straight Lines [id:da6235064658910423] 
\draw [color={byzantium}  ,draw opacity=0.5 ]   (360.75,109.67) -- (370.75,119.67) ;
%Straight Lines [id:da3064384313086357] 
\draw [color={byzantium}  ,draw opacity=0.5 ]   (366.13,109.42) -- (370.88,114.42) ;
%Straight Lines [id:da5899852844405518] 
\draw [color={byzantium}  ,draw opacity=0.5 ]   (360.63,114.67) -- (365.38,119.67) ;
%Straight Lines [id:da48830242663637047] 
\draw [color={byzantium}  ,draw opacity=0.5 ]   (264.5,209.67) -- (263.33,208.31) -- (260.25,204.72) ;
%Straight Lines [id:da6957122497594981] 
\draw [color={rgb, 255:red, 208; green, 2; blue, 27 }  ,draw opacity=1 ] [dash pattern={on 0.84pt off 2.51pt}]  (245.83,35.28) -- (260.5,50.22) ;
%Straight Lines [id:da5335494613055176] 
\draw [color={rgb, 255:red, 208; green, 2; blue, 27 }  ,draw opacity=1 ] [dash pattern={on 0.84pt off 2.51pt}]  (410.55,119.06) -- (429.35,119.06) ;
%Straight Lines [id:da5355870367796329] 
\draw [color={denim}  ,draw opacity=1 ]   (260.5,30.22) -- (410.08,30.21) ;
%Straight Lines [id:da6330098675214677] 
\draw [color={denim}  ,draw opacity=1 ]   (429.5,119.06) -- (430.21,49.91) ;
%Shape: Arc [id:dp23791862650430784] 
\draw  [draw opacity=0] (246.44,35.8) .. controls (250.06,32.32) and (255.01,30.18) .. (260.46,30.18) .. controls (260.57,30.18) and (260.67,30.18) .. (260.78,30.18) -- (260.46,49.91) -- cycle ; \draw  [color={denim}  ,draw opacity=1 ] (246.44,35.8) .. controls (250.06,32.32) and (255.01,30.18) .. (260.46,30.18) .. controls (260.57,30.18) and (260.67,30.18) .. (260.78,30.18) ;  
%Shape: Arc [id:dp08001739280116027] 
\draw  [draw opacity=0] (409.53,30.23) .. controls (409.79,30.22) and (410.05,30.21) .. (410.31,30.21) .. controls (421.25,30.21) and (430.13,39.13) .. (430.13,50.14) .. controls (430.13,50.15) and (430.12,50.16) .. (430.12,50.17) -- (410.31,50.14) -- cycle ; \draw  [color={denim}  ,draw opacity=1 ] (409.53,30.23) .. controls (409.79,30.22) and (410.05,30.21) .. (410.31,30.21) .. controls (421.25,30.21) and (430.13,39.13) .. (430.13,50.14) .. controls (430.13,50.15) and (430.12,50.16) .. (430.12,50.17) ;  
% Text Node
\draw (176,71) node [anchor=north west][inner sep=0.75pt]  [color={rgb, 255:red, 0; green, 0; blue, 0 }  ,opacity=1 ]  {$\partial \Omega $};
% Text Node
\draw (127.5,118) node [anchor=north west][inner sep=0.75pt]  [font=\LARGE]  {$\Omega $};
% Text Node
\draw (377,69.5) node [anchor=north west][inner sep=0.75pt]  [color={rgb, 255:red, 0; green, 0; blue, 0 }  ,opacity=1 ]  {$\Gamma _{I}$};
% Text Node
\draw (377,90) node [anchor=north west][inner sep=0.75pt]  [color={rgb, 255:red, 0; green, 0; blue, 0 }  ,opacity=1 ]  {$\Gamma _{D}$};
% Text Node
\draw (377,110) node [anchor=north west][inner sep=0.75pt]  [color={rgb, 255:red, 0; green, 0; blue, 0 }  ,opacity=1 ]  {$\Gamma _{N}$};
% Text Node
\draw (327.5,118) node [anchor=north west][inner sep=0.75pt]  [font=\LARGE]  {$\Omega $};

\end{tikzpicture}\vspace{-1mm}
         \caption{A thermally conducting body with fully \textit{(left)} and partly \textit{(right)}~\mbox{insulated}~\mbox{boundary}.}\vspace{-2.5mm}
         \label{fig:domain}
     \end{figure}
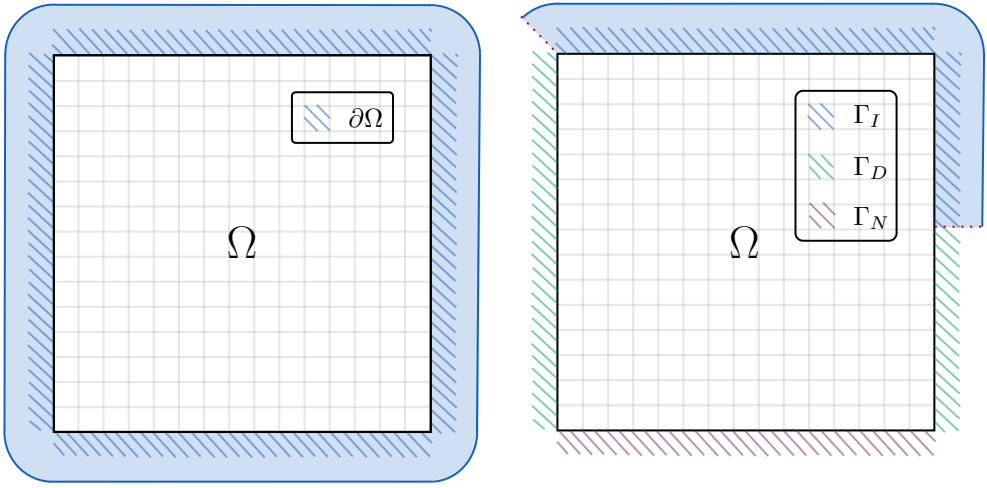
	
	\subsection{Classical function spaces}

    \hspace{5mm}Let $\omega\subseteq \mathbb{R}^d$, $d\in \mathbb{N}$, be a (Lebesgue) measurable set. Then, for (Lebesgue) measurable functions or vector fields $v,w\colon \omega\to \mathbb{R}^{\ell}$, $\ell\in\{1,d\}$, we employ~the~inner~product~${(v,w)_{\omega}\coloneqq \int_{\omega}{v\odot w\,\mathrm{d}x}}$, 
	whenever the right-hand side is well-defined, where $\odot\colon \mathbb{R}^{\ell}\times \mathbb{R}^{\ell}\to \mathbb{R}$ either denotes~scalar~multiplication or the Euclidean inner product. If $\vert \omega\vert\coloneqq \int_{\omega}{1\,\mathrm{d}x}\in  (1,+\infty)$, the average~of~an~\mbox{integrable} function or vector field $v\colon \omega\to \mathbb{R}^{\ell}$, $\ell\in\{1,d\}$,~is~defined~by~${\langle v\rangle_\omega\coloneqq \smash{\frac{1}{\vert \omega\vert}\int_{\omega}{v\,\mathrm{d}x}}}$. For $p\in [1,+\infty]$, we employ the notation $\|\cdot\|_{p,\omega}\coloneqq (\int_\omega{\vert \cdot\vert^p\,\mathrm{d}x})^{\smash{\frac{1}{p}}}$~if~$p\in [1,+\infty)$ and $\|\cdot\|_{\infty,\omega}\coloneqq 
    \textup{ess\,sup}_{x\in \omega}{\vert (\cdot)(x)\vert}$~else. Moreover, in the particular case $p=2$, we employ the abbreviated notation $\|\cdot\|_{\omega}\coloneqq \|\cdot\|_{2,\omega}$. 
    
    We employ the same notation in the case that $\omega$ is replaced by a (relatively) open boundary part $\gamma\subseteq \partial\Omega$, in which case the Lebesgue measure $\mathrm{d}x$ is replaced by the~surface~\mbox{measure}~$\mathrm{d}s$ 
    (\textit{e.g.}, we employ the notation $\vert \gamma\vert\coloneqq \int_{\gamma}{1\,\mathrm{d}s}$). 
	
	For $m\in \mathbb{N}$ and an open set $\omega\subseteq \mathbb{R}^d$, $d\in \mathbb{N}$,~we~define 
	\begin{align*} 
		\smash{H^m(\omega)\coloneqq \big\{v\in L^2(\omega)\mid \mathrm{D}^{\boldsymbol{\alpha}} v\in L^2(\omega)\textup{ for all }\boldsymbol{\alpha}\in (\mathbb{N}_0)^d\text{ with }\vert \boldsymbol{\alpha}\vert \leq m\big\}\,,}
	\end{align*}
	where $\mathrm{D}^{\boldsymbol{\alpha}}\coloneqq \smash{\frac{\partial^{\vert \boldsymbol{\alpha}\vert }}{\partial x_1^{\alpha_1}\cdot\ldots\cdot \partial x_d^{\alpha_d}}}$ and $\vert \boldsymbol{\alpha}\vert \coloneqq\sum_{i=1}^d{\alpha_i}$ 
	for each multi-index 
    ${\boldsymbol{\alpha}\hspace{-0.1em}\coloneqq \hspace{-0.1em}(\alpha_1,\ldots,\alpha_d)^\top\in (\mathbb{N}_0)^d}$, and the \emph{Sobolev~semi-norm}
	\begin{align*} 
		\vert\cdot\vert _{m,\omega}\coloneqq\Bigg(\sum_{\boldsymbol{\alpha}\in \mathbb{N}^d\,:\,\vert\boldsymbol{\alpha}\vert= m }{\|\mathrm{D}^{\boldsymbol{\alpha}}(\cdot)\|_{\omega}^2}\Bigg)^{\smash{\frac{1}{2}}}\,.
	\end{align*}  
	For $s \in (0,\infty)\setminus \mathbb{N}$ and an open set $\omega\subseteq \mathbb{R}^d$, $d\in \mathbb{N}$, the  \emph{Sobolev--Slobodeckij~semi-norm}~is~defined~by
	\begin{align*}
		\vert \cdot \vert_{s,\omega}\coloneqq \Bigg(\sum_{\boldsymbol{\alpha}\in \mathbb{N}^d\,:\,\vert \boldsymbol{\alpha}\vert= \lfloor s\rfloor}{\int_{\omega}{\int_{\omega}{\frac{\vert(\mathrm{D}^{\boldsymbol{\alpha}} (\cdot))(x)-(\mathrm{D}^{\boldsymbol{\alpha}}(\cdot))(y)\vert^2}{\vert x-y\vert^{2(s-\lfloor s\rfloor)+d}}\,\mathrm{d}x}\,\mathrm{d}y}}\Bigg)^{\smash{\frac{1}{2}}}\,.
	\end{align*} 
	Then, for $s \in (0,\infty)\setminus \mathbb{N}$ 
    and an open set $\omega\subseteq \mathbb{R}^n$, $n\in \mathbb{N}$,  the \emph{Sobolev--Slobodeckij space}~is~defined~by
	\begin{align*}
	      \smash{ H^s(\omega)\coloneqq \big\{v\in H^{ \lfloor s\rfloor}(\omega)\mid \vert v\vert _{s,\omega}<\infty\big\}\,.}
	\end{align*}\newpage
 
    The assumption $\smash{\Gamma_I\neq\emptyset}$ ensures the validity of a 
    \emph{Friedrich inequality} (\textit{cf}.\ \cite[Ex.\ II.5.13]{Galdi}), which states that there exists a constant $c_{\mathrm{F}}>0$ such that for every ${v\in H^1(\Omega)}$,~there~holds
        \begin{align}\label{lem:poin_cont}
            \|v\|_{\Omega}\leq \smash{c_{\mathrm{F}}\,\big\{\|\nabla v\|_{\Omega}+\vert \langle v\rangle_{\Gamma_I}\vert\big\}}\,.\\[-6mm]\notag
        \end{align}

	\subsubsection{Integration-by-parts formula and trace spaces}\vspace{-1mm}\enlargethispage{5mm}
	\hspace*{5mm}We define the space
	\begin{align*} 
		\smash{H(\textup{div};\Omega)\coloneqq \big\{y\in (L^2(\Omega))^d\mid \textup{div}\,y\in L^2(\Omega)\big\}\,.}
	\end{align*}
    %In addition, for $f\in L^2(\Omega)$, we define
    %\begin{align*}  
    %    H(\textup{div}\!=\!f;\Omega)\coloneqq \big\{y\in H(\textup{div};\Omega)\mid \textup{div}\,y=f\text{ a.e.\ in }\Omega\big\}\,.
	%\end{align*}
	Denote by $\textup{tr}(\cdot)\colon \hspace*{-0.15em}H^1(\Omega)\hspace*{-0.15em}\to\hspace*{-0.15em} \smash{H^{\smash{\frac{1}{2}}}(\partial\Omega)}$ the trace operator  and by ${\textup{tr}((\cdot)\cdot n)\colon \hspace*{-0.15em} H(\textup{div};\Omega)\hspace*{-0.15em}\to\hspace*{-0.15em} \smash{H^{-\smash{\frac{1}{2}}}(\partial\Omega)}}$~the normal trace operator. %, where $n\colon \partial\Omega\to \mathbb{S}^{d-1}$ denotes the outward unit normal vector~field~to~$\partial\Omega$.
	Then,  for every $v\hspace*{-0.1em}\in\hspace*{-0.1em} H^1(\Omega)$~and~${y\hspace*{-0.1em}\in\hspace*{-0.1em} H(\textup{div};\Omega)}$, there holds %the integration-by-parts formula 
    (\textit{cf}.~\mbox{\cite[Sec.~4.3]{EG21I}}) 
	\begin{align}\label{eq:pi_cont}
		(\nabla v,y)_{\Omega}+(v,\textup{div}\,y )_{\Omega}=\langle \textup{tr}(y)\cdot n,\textup{tr}(v)\rangle_{\partial\Omega}\,,
	\end{align}
	 where, for every $\widehat{y}\in \smash{H^{-\smash{\frac{1}{2}}}(\gamma)}$, $\widehat{v}\in \smash{H^{\smash{\frac{1}{2}}}(\gamma)}$, and $\gamma\in \{\Gamma_I,\Gamma_D,\Gamma_N,\partial\Omega\}$, we  abbreviate 
	\begin{align}\label{eq:abbreviation}
		\smash{\langle \widehat{y},\textup{tr}(\widehat{v})\rangle_{\gamma}\coloneqq \langle \widehat{y},\textup{tr}(\widehat{v})\rangle_{H^{\smash{\frac{1}{2}}}(\gamma)}\,.}
	\end{align}
	In \eqref{eq:abbreviation}, for  $\gamma\subseteq \partial \Omega$ and $s >0$, the  space $H^s(\gamma)$~is~defined~as~the~range~of~the~\mbox{restricted}~trace operator~$\textup{tr}(\cdot)|_{\gamma}$ defined on $H^{s+\smash{\frac{1}{2}}}(\Omega)$
	 endowed with %~the~image~norm, \textit{i.e.}, 
     ${\|\cdot\|_{s,\gamma}\hspace{-0.15em}\coloneqq\hspace{-0.15em} \inf\{\|v\|_{s +\smash{\frac{1}{2}},\Omega}\mid v\hspace{-0.15em}\in\hspace{-0.15em} H^{s+\smash{\frac{1}{2}}}(\Omega) : \textup{tr}(v)|_{\gamma}\hspace{-0.15em}=\hspace{-0.15em}(\cdot)\}}$, and $H^{-s}(\gamma)\coloneqq (H^s(\gamma))^*$  as the associated (topological) dual space. 
	
	Moreover, for every $\mathrm{X}\in \{I,D,N\}$, we employ the notation
	\begin{align*} 
		 	\smash{H^1_{\mathrm{X}}(\Omega)\coloneqq  \big\{v\in 	H^1(\Omega) \mid \textup{tr}(v)=0\textup{ a.e.\ on }\Gamma_{\mathrm{X}}\big\}\,.}  
	\end{align*}

    If $y\hspace{-0.15em}\in\hspace{-0.15em} H(\textup{div};\Omega)$ is such that there exists a constant $c\hspace{-0.15em}>\hspace{-0.15em}0$ such that 
    for every $\smash{v\hspace{-0.15em}\in\hspace{-0.15em}  H^1_D(\Omega)\hspace{-0.1em}\cap\hspace{-0.1em} H^1_N(\Omega)}$, there holds\vspace{-1mm}
    \begin{align*}
        \vert \langle \textup{tr}(y)\cdot n,\textup{tr}(v)\rangle_{\partial\Omega}\vert \leq c\,\|\textup{tr}(v)\|_{1,\Gamma_I}\,,
    \end{align*}
    then, by the Hahn--Banach theorem, there exists an  extension $\smash{\overline{\textup{tr}(y)\cdot n}}\in L^\infty(\Gamma_I)\cong (L^1(\Gamma_I))^*$, \textit{i.e.},  for every $v\in H^1_D(\Omega)\cap H_N^1(\Omega)$, we have that
    \begin{align*}
        (\smash{\overline{\textup{tr}(y)\cdot n}}, \textup{tr}(v))_{\Gamma_I}=\langle \textup{tr}(y)\cdot n,\textup{tr}(v)\rangle_{\partial\Omega}\,.
    \end{align*}
    In light of the previous argument, we introduce the space\enlargethispage{2.5mm}
    \begin{align*}
        \smash{\overline{H}_I(\textup{div};\Omega)\coloneqq \big\{y\in H(\textup{div};\Omega)\mid \exists \,\smash{\overline{\textup{tr}(y)\cdot n}}\in L^\infty(\Gamma_I)\big\}\,,}
    \end{align*}
    which turns out to be the natural energy space of an associated (Fenchel) dual problem to \eqref{eq:primal_intro}.
    This is primarily a consequence of the following lemma.
 
    \begin{lemma}\label{lem:normal_trace}
        Let $y\in H(\textup{div};\Omega)$ and $g\in H^{-\frac{1}{2}}(\Gamma_N)$ be  such that there exists a constant $c>0$ such that for every $v\in  H^1_D(\Omega)$, there holds
        \begin{align}\label{lem:normal_trace.1}
            \vert \langle \textup{tr}(y)\cdot n,\textup{tr}(v)\rangle_{\partial\Omega}-\langle g,\textup{tr}(v)\rangle_{\Gamma_N}\vert \leq c\,\|\textup{tr}(v)\|_{1,\Gamma_I}\,.
        \end{align} 
        \if0
        the following statements apply:
        \begin{itemize}[noitemsep,topsep=2pt,leftmargin=!,labelwidth=\widthof{(ii)}]
            \item[(i)]\hypertarget{lem:normal_trace.i}{} There exists a constant $c>0$ such that $\vert \langle \textup{tr}(y)\cdot n,\textup{tr}(v)\rangle_{\partial\Omega}-\langle g,\textup{tr}(v)\rangle_{\Gamma_N}\vert \leq c\,\|\textup{tr}(v)\|_{1,\Gamma_I\cup\Gamma_N}$ for all ${v\in  H^1_D(\Omega)}$;
            \item[(ii)]\hypertarget{lem:normal_trace.ii}{}  For every $v\in H^1_I(\Omega)\cap H^1_D(\Omega)$, there holds $\langle \textup{tr}(y)\cdot n,\textup{tr}(v)\rangle_{\partial\Omega}-\langle g,\textup{tr}(v)\rangle_{\Gamma_N}=0$. \todo{why is the Neumann part zero?}
        \end{itemize}
        \fi
        %\begin{itemize}[noitemsep,topsep=2pt,leftmargin=!,labelwidth=\widthof{(ii)}]
        %    \item[(i)] There exists a constant $c>0$ such that for every $v\in H^1_D(\Omega)$, it holds that
        %    \begin{align*}
        %         \vert \langle \textup{tr}(y)\cdot n,\textup{tr}(v)\rangle_{\partial\Omega}\vert \leq c\,\|\textup{tr}(v)\|_{1,\Gamma_I\cup \Gamma_N}\,;
        %    \end{align*}
        %    \item[(ii)]  For every $v\in H^1_I(\Omega)\cap H^1_D(\Omega)$, it holds that
        %    \begin{align*}
        %        \langle \textup{tr}(y)\cdot n,\textup{tr}(v)\rangle_{\partial\Omega}=0\,.
        %    \end{align*}
        %\end{itemize}
      %  If $y\in H(\textup{div};\Omega)$ is such that 
   % for every $v\in H^1_D(\Omega)$, it holds that
   % \begin{align*}
  %      \vert \langle \textup{tr}(y)\cdot n,\textup{tr}(v)\rangle_{\partial\Omega}\vert \leq c\,\|\textup{tr}(v)\|_{1,\Gamma_I\cup \Gamma_N}\,.
 %   \end{align*}
    Then, we have that $y\in \overline{H}_I(\textup{div};\Omega)$ and for every $v\in H^1_D(\Omega)$, there holds
    \begin{align}\label{lem:normal_trace.2}
         (\smash{\overline{\textup{tr}(y)\cdot n}}, \textup{tr}(v))_{\Gamma_I}=\langle \textup{tr}(y)\cdot n,\textup{tr}(v)\rangle_{\partial\Omega}-\langle g,\textup{tr}(v)\rangle_{\Gamma_N}\,.
    \end{align} 
    \end{lemma}

    \begin{proof}
    By the Hahn--Banach theorem, there exists some $E\in \smash{L^\infty(\Gamma_I\cup\Gamma_N)}\cong (L^1(\Gamma_I\cup\Gamma_N))^* $ such that  for every $v\in H^1_D(\Omega)$, we have that
    \begin{align}\label{lem:normal_trace.3}
        (E, \textup{tr}(v))_{\Gamma_I\cup\Gamma_N}=\langle \textup{tr}(y)\cdot n,\textup{tr}(v)\rangle_{\partial\Omega}-\langle g,\textup{tr}(v)\rangle_{\Gamma_N}\,.
    \end{align}
    Moreover,  for every $v\in H^1_I(\Omega)\cap H^1_D(\Omega)$, from \eqref{lem:normal_trace.1}, it follows that 
    \begin{align}\label{lem:normal_trace.4}
        \langle \textup{tr}(y)\cdot n,\textup{tr}(v)\rangle_{\partial\Omega}-\langle g,\textup{tr}(v)\rangle_{\Gamma_N}=0\, ,
    \end{align}
        which, due to \eqref{lem:normal_trace.3} and the density of  $(\textup{tr}(\cdot)|_{\Gamma_N})(H^1_I(\Omega)\cap H^1_D(\Omega))$ in $L^1(\Gamma_N)$, implies~that~${E=0}$~a.e. on~$\Gamma_N$, so that from \eqref{lem:normal_trace.3}, it follows that $y\hspace{-0.15em}\in\hspace{-0.15em} \overline{H}_I(\textup{div};\Omega)$ with \eqref{lem:normal_trace.2}, where ${\overline{\textup{tr}(y)\cdot n}\hspace{-0.15em}=\hspace{-0.15em}E|_{\Gamma_I}\hspace{-0.15em}\in\hspace{-0.15em} L^\infty(\Gamma_I)}$.

    \if0
       Due to (\hyperlink{lem:normal_trace.i}{i}), by the Hahn--Banach theorem, there exists an extension~$\smash{\overline{\textup{tr}(y)\cdot n}}\in \smash{L^\infty(\Gamma_I\cup\Gamma_N)}$ $\cong (L^1(\Gamma_I\cup\Gamma_N))^* $, \textit{i.e.},  for every $v\in H^1_D(\Omega)$, we have that
    \begin{align}\label{lem:normal_trace.1}
        (\smash{\overline{\textup{tr}(y)\cdot n}}, \textup{tr}(v))_{\Gamma_I\cup \Gamma_N}=\langle \textup{tr}(y)\cdot n,\textup{tr}(v)\rangle_{\partial\Omega}\,.
    \end{align}
    In particular, we have that $y\in \overline{H}_I(\textup{div};\Omega)$.
        %In addition, 
        Due to (\hyperlink{lem:normal_trace.ii}{ii}), for every $v\in H^1_I(\Omega)\cap H^1_D(\Omega)$,~there~holds
        \begin{align*}
            (\smash{\overline{\textup{tr}(y)\cdot n}}, \textup{tr}(v))_{\Gamma_N}&=(\smash{\overline{\textup{tr}(y)\cdot n}}, \textup{tr}(v))_{\Gamma_I\cup \Gamma_N}
            \\&=\langle \textup{tr}(y)\cdot n,\textup{tr}(v)\rangle_{\partial\Omega}=\langle g,\textup{tr}(v)\rangle_{\Gamma_N}\,,
        \end{align*}
        which, due to  $(\textup{tr}(\cdot)|_{\Gamma_N})(H^1_I(\Omega)\cap H^1_D(\Omega))=H^{\frac{1}{2}}(\Gamma_N)$, implies that
        %\begin{align}\label{lem:normal_trace.2}
        %    \smash{\overline{\textup{tr}(y)\cdot n}}=0\textup{ a.e.\ on }\Gamma_N\,.
        %\end{align}
        $\smash{\overline{\textup{tr}(y)\cdot n}}=0$~a.e.~on~$\Gamma_N$ and, consequently, together with \eqref{lem:normal_trace.1}, the assertion. 
        \fi
    \end{proof}
   % \fi 

  In the following, we will in most cases refrain from writing $\textup{tr}(\cdot)$ or $\textup{tr}((\cdot)\cdot n)$.\pagebreak

	\subsection{Triangulations and standard finite element spaces}\vspace{-0.5mm}\enlargethispage{6mm}
	
	\hspace{5mm}In what follows, we denote by $\{\mathcal{T}_h\}_{h>0}$ a family of shape-regular triangulations~of $\Omega$ (\textit{cf}.\  \cite{EG21I}). Here, the parameter
	$h>0$ refers to the \textit{averaged mesh-size}, \textit{i.e.},~we~define~${h 
	\coloneqq (\vert \Omega\vert/\textup{card}(\mathcal{N}_h))^{\frac{1}{d}}}
	$, where $\mathcal{N}_h$  is the set of vertices of $\mathcal{T}_h$. We define  the following sets of sides of $\mathcal{T}_h$:
	\begin{align*}
	 	\mathcal{S}_h&\coloneqq \mathcal{S}_h^{i}\cup \mathcal{S}_h^{\partial}\,,\\
	 	\mathcal{S}_h^{i}&\coloneqq \{T\cap T'\mid T,T'\in\mathcal{T}_h\,,\text{dim}_{\mathscr{H}}(T\cap T')=d-1\}\,,\\
	 	\mathcal{S}_h^{\partial}&\coloneqq\{T\cap \partial\Omega\mid T\in \mathcal{T}_h\,,\text{dim}_{\mathscr{H}}(T\cap \partial\Omega)=d-1\}\,,\\
	 	\mathcal{S}_h^{\mathrm{X}}&\coloneqq\{S\in \mathcal{S}_h^{\partial}\mid \textup{int}(S)\subseteq \Gamma_\mathrm{X}\}\text{ for } \mathrm{X}\in \{I,D,N\}\,,
	 \end{align*}
	 where the Hausdorff dimension is defined by $\text{dim}_{\mathscr{H}}(\omega)\coloneqq\inf\{d'\geq 0\mid \mathscr{H}^{d'}(\omega)=0\}$~for~all~${\omega\subseteq \mathbb{R}^d}$.
	 We also assume that  $\{\mathcal{T}_h\}_{h>0}$ and $\Gamma_I$, $\Gamma_D$,~and~$\Gamma_N$~are~chosen in such a way that  ${\mathcal{S}_h^{\partial}=\mathcal{S}_h^I\dot{\cup}\mathcal{S}_h^D\dot{\cup} \mathcal{S}_h^N}$.%, \textit{e.g.}, in the case $d=2$,  $\overline{\Gamma}_D$, $\overline{\Gamma}_C$, and $\overline{\Gamma}_N$~touch~only~in~vertices.

	For $n\in \mathbb{N}_0$ and $T\in \mathcal{T}_h$, let $\mathbb{P}^n(T)$ denote the set of polynomials of maximal~degree~$n$~on~$T$. Then, for $n\in \mathbb{N}_0$, the \textit{space of  element-wise polynomial functions (of order $n$)} is defined by
	\begin{align*}
		\smash{\mathcal{L}^n(\mathcal{T}_h)\coloneqq \big\{v_h\in L^\infty(\Omega)\mid v_h|_T\in\mathbb{P}^k(T)\text{ for all }T\in \mathcal{T}_h\big\}\,.}
	\end{align*} 
	For $\ell \hspace{-0.15em}\in \hspace{-0.15em}\{1,d\}$, the (local) $L^2$-projection $\Pi_h\colon  \hspace{-0.15em}(L^1(\Omega))^{\ell}\hspace{-0.15em} \to\hspace{-0.15em} (\mathcal{L}^0(\mathcal{T}_h))^{\ell} $ onto element-wise~\mbox{constant}~functions or vector fields, respectively, for every 
	$v\in  (L^1(\Omega))^{\ell} $ is defined by $\Pi_h v|_T\coloneqq\langle v\rangle_T$~for~all~${T\in \mathcal{T}_h}$.  
	
	For $m\in \mathbb{N}_0$ and $S\in \mathcal{S}_h$, let $\mathbb{P}^m(S)$ denote the set of polynomials of maximal degree~$m$~on~$S$. Then, for $m\hspace*{-0.175em}\in\hspace*{-0.175em} \mathbb{N}_0$ and $\widehat{\mathcal{S}}_h\hspace*{-0.175em}\in\hspace*{-0.175em} \{\mathcal{S}_h,\mathcal{S}^{i}_h,\mathcal{S}^{\partial}_h,\mathcal{S}^D_h,\mathcal{S}^I_h,\mathcal{S}^N_h\}$, the \textit{space of side-wise~polynomial functions (of order $m$)}  is defined by
	\begin{align*}
	\smash{\mathcal{L}^m(\widehat{\mathcal{S}}_h)\coloneqq  \big\{v_h\in L^\infty(\cup\widehat{\mathcal{S}}_h)\mid v_h|_S\in\mathbb{P}^m(S)\text{ for all }S\in \widehat{\mathcal{S}}_h\big\}\,.}
	\end{align*} 
	For $\ell  \hspace*{-0.15em}\in  \hspace*{-0.15em}\{1,d\}$, the (local) $L^2$-projection $\pi_h\colon  \hspace*{-0.15em}(L^1(\cup\mathcal{S}_h))^{\ell} \hspace*{-0.15em}\to \hspace*{-0.15em} (\mathcal{L}^0(\mathcal{S}_h))^{\ell}$ onto side-wise constant functions or vector  fields, respectively,  for every 
	$v\hspace*{-0.15em}\in  \hspace*{-0.15em}(L^1(\cup\mathcal{S}_h))^{\ell} $ is defined by ${\pi_h v|_S\hspace*{-0.15em}\coloneqq \hspace*{-0.15em}\langle v\rangle_S}$~for~all~${S\hspace*{-0.15em}\in \hspace*{-0.15em}\mathcal{S}_h}$.

	\subsubsection{Crouzeix--Raviart element}\vspace{-0.5mm}
	
	\hspace{5mm}The \textit{Crouzeix--Raviart space} (\textit{cf}.\ \cite{CR73}) is defined as  
	\begin{align}
		\mathcal{S}^{1,cr}(\mathcal{T}_h)\coloneqq \big\{v_h\in \mathcal{L}^1(\mathcal{T}_h)\mid \pi_h\jump{v_h}=0\text{ a.e.\ on }\cup \mathcal{S}_h^{i}\big\}\,,\label{def:CR}
	\end{align}
    where, for  every $v_h\in \mathcal{L}^1(\mathcal{T}_h)$, the \textit{jump (across $\mathcal{S}_h$)} $\jump{v_h}\in \mathcal{L}^1(\mathcal{S}_h)$,~is~defined~by~${\jump{v_h}|_S\coloneqq \jump{v_h}_S}$ for all $S\in\mathcal{S}_h$, where for every  $S\in\mathcal{S}_h$, the \textit{jump (across $S$)} $\jump{v_h}_S\in \mathbb{P}^1(S)$ is defined by 
	\begin{align*}
		\jump{v_h}_S\coloneqq\begin{cases}
			v_h|_{T_+}-v_h|_{T_-}&\text{ if }S\in \mathcal{S}_h^{i}\,,\text{ where }T_+, T_-\in \mathcal{T}_h\text{ satisfy }\partial T_+\cap\partial T_-=S\,,\\
			v_h|_T&\text{ if }S\in\mathcal{S}_h^{\partial}\,,\text{ where }T\in \mathcal{T}_h\text{ satisfies }S\subseteq \partial T\,.
		\end{cases}
	\end{align*}
	Denote by 	$\varphi_S \in \smash{\mathcal{S}^{1,cr}(\mathcal{T}_h)}$, $S \in \mathcal{S}_h$, satisfying 
	$\langle\varphi_S\rangle_{S'} = \delta_{S,S'}$~for~all~$S,S' \in \mathcal{S}_h$,~a~\mbox{basis}~of~$\smash{\mathcal{S}^{1,cr}(\mathcal{T}_h)}$.
	Then, the canonical interpolation operator $\smash{\Pi_h^{cr}\colon H^1(\Omega)\to \smash{\mathcal{S}^{1,\textit{cr}}(\mathcal{T}_h)}}$  (\textit{cf}.\ \cite[Secs.\ 36.2.1, 36.2.2]{EG21II}), for every $v\in H^1(\Omega)$ defined by
	\begin{align}
		\Pi_h^{cr}v\coloneqq \sum_{S\in \mathcal{S}_h}{\langle v\rangle_S\,\varphi_S}\,,\label{CR-interpolant}
	\end{align}
	preserves averages of gradients and moments (on sides), \textit{i.e.}, for every $v\in H^1(\Omega)$, there holds
	\begin{alignat}{2}
		\nabla_h\Pi_h^{cr}v&=\Pi_h\nabla v&& \quad\text{ a.e.\ in }\Omega\,,\label{eq:grad_preservation}\\
		\pi_h\Pi_h^{cr}v&=\pi_h  v&&\quad \text{ a.e.\ on }\cup\mathcal{S}_h\, ,\label{eq:trace_preservation}
	\end{alignat}
	where $\nabla_h\colon \mathcal{L}^1(\mathcal{T}_h)\to (\mathcal{L}^0(\mathcal{T}_h))^d$ is  defined by $(\nabla_hv_h)|_T\coloneqq \nabla(v_h|_T)$ for all $v_h\in \mathcal{L}^1(\mathcal{T}_h)$~and~${T\in \mathcal{T}_h}$.
    
    %For every $s \in [0,1]$, there  exists a constant $c>0$ (\emph{cf}.\ \cite[Lem.\ 36.1]{EG21II}), independent of $h>0$, 
 %such that for every 
	%$v\in H^{1+s}(\Omega)$ and $T\in \mathcal{T}_h$,~it~holds~that
	%\begin{align}\label{eq:CR-Interpolant-Rate}
	%	\|v-\Pi_h^{cr} v\|_T+h_T\, \|\nabla v-\nabla \Pi_h^{cr} v\|_T\leq c\,h_T^{1+s}\,\vert v\vert_{1+s,T}\,.
	%\end{align}  

    The assumption $\Gamma_I\neq\emptyset$ also ensures the validity of a
    \emph{discrete Friedrich inequality}.\vspace{-1mm}

    \begin{lemma}\label{lem:poin_discrete}
        There exists a constant $c_{\mathrm{F}}^{cr}>0$ such that for every $v_h\in \mathcal{S}^{1,cr}(\mathcal{T}_h)$, there holds
        \begin{align*}
            \|v_h\|_{\Omega}\leq c_{\mathrm{F}}^{cr}\,\smash{\big\{\|\nabla_h v_h\|_{\Omega}+\vert \langle \pi_h v_h\rangle_{\Gamma_I}\vert\big\}}\,. 
        \end{align*}
    \end{lemma}

    \begin{proof}
        Let $I_h^{p1}\colon\hspace{-0.15em} \mathcal{S}^{1,cr}(\mathcal{T}_h)\hspace{-0.15em}\to \hspace{-0.15em}\mathcal{S}^{1,cr}(\mathcal{T}_h)\cap H^1(\Omega)$ be an $H^1$-enriching operator (\textit{e.g.}, the~\mbox{node-averag-} ing\hspace{-0.15mm} quasi-interpolation \hspace{-0.15mm}operator \hspace{-0.15mm}$\Pi_h^{av}$, \hspace{-0.15mm}\textit{cf}.\ \hspace{-0.15mm}\cite[Sec.\ \hspace{-0.15mm}22.2]{EG21I}) \hspace{-0.15mm}such \hspace{-0.1mm}that \hspace{-0.15mm}for \hspace{-0.1mm}every \hspace{-0.15mm}$v_h\hspace{-0.175em}\in\hspace{-0.175em} \mathcal{S}^{1,cr}(\mathcal{T}_h)$,~\hspace{-0.15mm}there~\hspace{-0.15mm}holds
        \begin{align}\label{lem:poin_discrete.1}
            \|v_h-I_h^{p1}v_h\|_{\Omega}+h\,\|\nabla I_h^{p1}v_h\|_{\Omega}\leq c^{p1}h\,\|\nabla_h v_h\|_{\Omega}\,,
        \end{align}
        where $c^{p1}>0$ is independent of $h>0$. Using \eqref{lem:poin_discrete.1} and  the Friedrich~inequality~\eqref{lem:poin_cont},~we~find~that
        \begin{align*}
            \|v_h\|_{\Omega}& \leq \|I_h^{p1}v_h\|_{\Omega}+\|v_h-I_h^{p1}v_h\|_{\Omega}
            \\&\leq 
            (c_{\mathrm{F}}+c^{p1}h)\,\|\nabla I_h^{p1}v_h\|_{\Omega}+ c_{\mathrm{F}}\,\vert \langle I_h^{p1} v_h\rangle_{\Gamma_I}\vert
            \\&\leq 
             (c_{\mathrm{F}}+c^{p1}h)\,c^{p1}\|\nabla_h v_h\|_{\Omega}+c_{\mathrm{F}}\,\vert \langle v_h\rangle_{\Gamma_I}\vert
             +c_{\mathrm{F}}\,\vert \langle v_h-I_h^{p1} v_h\rangle_{\Gamma_I}\vert\,.
        \end{align*}
        Eventually, the claimed discrete Friedrich inequality 
        follows from $\vert \langle v_h-I_h^{p1} v_h\rangle_{\Gamma_I}\vert%\leq c\,\|v_h-I_h^{p1} v_h\|_{1,\Gamma_I}
            %\\
            %&\leq c\,\|\nabla_h v_h\|_{1,\Omega}
            %\\&
            \leq c_1\,h^{\frac{1}{2}}\|\nabla_h v_h\|_{\Omega}$ (\textit{cf}.\ \cite[Rem.\ 12.17]{EG21I}),
        %\begin{align*}
        %    \vert \langle v_h-I_h^{p1} v_h\rangle_{\Gamma_I}\vert%\leq c\,\|v_h-I_h^{p1} v_h\|_{1,\Gamma_I}
        %    %\\
        %    %&\leq c\,\|\nabla_h v_h\|_{1,\Omega}
        %    %\\&
        %    \leq c_1\,\|\nabla_h v_h\|_{\Omega}\,,
        %\end{align*}
        where $c_1>0$ is independent of $h>0$, 
        and $\langle v_h\rangle_{\Gamma_I}=\langle \pi_hv_h\rangle_{\Gamma_I} $.%, we conclude the validity of the claimed discrete Friedrich inequality.
    \end{proof}

    If $\Gamma_D\neq \emptyset$, in the discrete Friedrich inequality (\textit{cf}.\ Lemma \ref{lem:poin_discrete}) the boundary integral on the right-hand side can be omitted,  
    when restricted~to~the~space
    \begin{align*}
    \mathcal{S}^{1,cr}_D(\mathcal{T}_h)\coloneqq\big\{v_h\in \mathcal{S}^{1,cr}(\mathcal{T}_h)\mid \pi_hv_h=0\text{ a.e.\ on }\Gamma_D\big\}\,.
    \end{align*}
	
	\subsubsection{Raviart--Thomas element}\vspace{-0.5mm}
	
	\hspace{5mm}The \textit{(lowest order) Raviart--Thomas space} (\textit{cf}.\ \cite{RT75}) is defined as 
	\begin{align}
		\mathcal{R}T^0(\mathcal{T}_h)\coloneqq \bigg\{y_h\in(\mathcal{L}^1(\mathcal{T}_h))^d\;\bigg|\; \begin{aligned}
			y_h|_T\cdot n_T&=\textup{const}\text{ on }\partial T\text{ for all }T\in \mathcal{T}_h\,,\\[-1mm] 
			\jump{y_h\cdot n}_S&=0\text{ on }S\text{ for all }S\in \mathcal{S}_h^{i}
		\end{aligned}\,\bigg\}\,, \label{def:RT}  
	\end{align}  
    where, for  every $y_h\in (\mathcal{L}^1(\mathcal{T}_h))^d$ and $S\in\mathcal{S}_h$, the \emph{normal jump (across} $S$)~is~defined~by 
	\begin{align*}
		\jump{y_h\cdot n}_S\coloneqq\begin{cases}
			y_h|_{T_+}\!\cdot n_{T_+}+y_h|_{T_-}\!\cdot n_{T_-}&\text{ if }S\in \mathcal{S}_h^{i}\,,\text{ where }T_+, T_-\in \mathcal{T}_h\text{ satisfy }\partial T_+\cap\partial T_-=S\,,\\
			y_h|_T\cdot n&\text{ if }S\in\mathcal{S}_h^{\partial}\,,\text{ where }T\in \mathcal{T}_h\text{ satisfies }S\subseteq \partial T\,,
		\end{cases}
	\end{align*}
	where, for every $T\in \mathcal{T}_h$, $\smash{n_T\colon\partial T\to \mathbb{S}^{d-1}}$ denotes the outward unit normal vector field to $ T$.
    %Note that $\mathcal{R}T^0(\mathcal{T}_h)\subseteq H(\textup{div};\Omega)$.
    %In addition, for $f_h\in \mathcal{L}^0(\mathcal{T}_h)$, we define
    %\begin{align*}
    %    \mathcal{R}T^0_{\textup{div}=f_h}(\mathcal{T}_h)\coloneqq \big\{y_h\in \mathcal{R}T^0(\mathcal{T}_h)\mid \textup{div}\,y_h=f_h\text{ a.e.\ in }\Omega\big\}\,.
    %\end{align*}
    %Then, %for $f_h\in \mathcal{L}^0(\mathcal{T}_h)$,  we have that $\mathcal{R}T^0_{\textup{div}=f_h}(\mathcal{T}_h)\subseteq H(\textup{div}\!=\!f_h;\Omega)$. 
	Denote by $\psi_S\hspace{-0.15em}\in\hspace{-0.15em} \mathcal{R}T^0(\mathcal{T}_h)$, $S\hspace{-0.15em}\in\hspace{-0.15em} \mathcal{S}_h$, satisfying  $\psi_S|_{S'}\cdot n_{S'}\hspace{-0.15em}=\hspace{-0.15em}\delta_{S,S'}$ on $S'$ for all ${S'\hspace{-0.15em}\in\hspace{-0.15em} \mathcal{S}_h}$,~a~basis~of~$\mathcal{R}T^0(\mathcal{T}_h)$,  where~$n_S$ is the unit normal vector on $S$ pointing from $T_-$ to $T_+$~if~${T_+,T_-\hspace{-0.15em}\in\hspace{-0.15em} \mathcal{T}_h}$~with~${S\hspace{-0.15em}=\hspace{-0.15em}\partial T_+\cap \partial T_-}$. Then,  
	\hspace{-0.1mm}the \hspace{-0.1mm}canonical \hspace{-0.1mm}interpolation \hspace{-0.1mm}operator \hspace{-0.1mm}${\Pi_h^{rt}\colon \hspace{-0.175em}V_{p,q}(\Omega)\hspace{-0.175em}\coloneqq \hspace{-0.175em}\{y\hspace{-0.175em}\in\hspace{-0.175em} (L^p(\Omega))^d\hspace{-0.15em}\mid\hspace{-0.15em}  \textup{div}\,y\hspace{-0.175em}\in \hspace{-0.175em}L^q(\Omega)\}\hspace{-0.175em}\to \hspace{-0.175em}\smash{\mathcal{R}T^{0}(\mathcal{T}_h)}}$ (\textit{cf}.~\mbox{\cite[Sec.~16.1]{EG21I}}), where $p>2$ and $q>\frac{2d}{d+2}$, for every $y\in V_{p,q}(\Omega)$~defined~by
	\begin{align}
		\Pi_h^{rt} y\coloneqq \sum_{S\in \mathcal{S}_h}{\langle y\cdot n_S\rangle_S\,\psi_S}\,,\label{RT-interpolant}
	\end{align}
	preserves averages of divergences and normal traces (on sides), \textit{i.e.}, for every $y\hspace{-0.175em}\in\hspace{-0.175em} V_{p,q}(\Omega)$,~there~holds 
	\begin{alignat}{2}
		\textup{div}\,\Pi_h^{rt}y&=\Pi_h\textup{div}\,y&&\quad \text{ a.e.\ in }\Omega\,,\label{eq:div_preservation}\\
		\Pi_h^{rt}y\cdot n&=\pi_h(y\cdot n)&&\quad \text{ a.e.\ on }\cup\mathcal{S}_h\,.\label{eq:normal_trace_preservation}
	\end{alignat}
    In definition \eqref{RT-interpolant}, the local averages $(\langle y\cdot n_S\rangle_S)_{S\in \mathcal{S}_h}$  are defined via local lifting~as~in~\mbox{\cite[(12.12)]{EG21I}} and, in \eqref{eq:normal_trace_preservation}, the function $\pi_h(y\cdot n)\in  \mathcal{L}^0(\mathcal{S}_h)$ is defined by $\pi_h(y\cdot n)|_S=\langle y\cdot n_S\rangle_S$~for~all~${S\in \mathcal{S}_h}$.
    From \hspace{-0.15mm}the \hspace{-0.15mm}structure-preserving \hspace{-0.15mm}properties \hspace{-0.15mm}\eqref{eq:div_preservation},\eqref{eq:normal_trace_preservation} \hspace{-0.15mm}of \hspace{-0.15mm}the \hspace{-0.15mm}canonical \hspace{-0.15mm}interpolation~\hspace{-0.15mm}\mbox{operator}~\hspace{-0.15mm}\eqref{RT-interpolant}, it readily follows the surjectivity of the divergence operator  from\enlargethispage{5mm}
    \begin{align*}
        \mathcal{R}T^0_N(\mathcal{T}_h)\coloneqq \smash{\big\{y_h\in \mathcal{R}T^0_N(\mathcal{T}_h)\mid y_h\cdot n=0\text{ a.e.\ in }\Gamma_N\big\}}\,,
    \end{align*}
    into $\mathcal{L}^0(\mathcal{T}_h)$  if $\Gamma_N\neq\partial\Omega$ and into $\mathcal{L}^0(\mathcal{T}_h)/\mathbb{R}$ else.\medskip
    
%\textit{i.e.}, %due to $\Gamma_N\neq \partial\Omega$ (as $\Gamma_I\neq \emptyset)$, 
%we have that\enlargethispage{7.5mm}
%\begin{align*}
%    \textup{div}(\mathcal{R}T^0_N(\mathcal{T}_h))
%    =\begin{cases}
%        \mathcal{L}^0(\mathcal{T}_h)\text{ if }\Gamma_N=\partial\Omega\,,\\
%        \mathcal{L}^0(\mathcal{T}_h)/\mathbb{R}\text{ else}\,.
%{cases}
%\end{align*}
    %For every $s\hspace*{-0.15em} \in \hspace*{-0.15em}(\frac{1}{2},1]$, there exists a constant $c\hspace*{-0.15em}>\hspace*{-0.15em}0$ (\emph{cf}.\ \cite[Thms.\ 16.4, 16.6]{EG21I}),~\mbox{independent}~of~${h\hspace*{-0.15em}>\hspace*{-0.15em}0}$, such that for every $y\in (H^s(\Omega))^d\cap H(\textup{div};\Omega)$ and $T\in \mathcal{T}_h$, it holds that
	%\begin{align}
	%	\|y-\Pi_h^{rt} y\|_T\leq c\,h_T^{s}\,\vert y\vert_{s,T}\,.\label{eq:RT-Interpolant-Rate}
	%\end{align} 

	%\subsubsection{Discrete integration-by-parts formula}\vspace{-0.5mm}

    The Crouzeix--Raviart element (\textit{cf}.\ \eqref{def:CR})   the Raviart--Thomas element (\textit{cf}.~\eqref{def:RT}) are deeply connected, in particular, through a  \textit{discrete integration-by-parts formula}, which states that
	 for every $v_h\in \mathcal{S}^{1,cr}(\mathcal{T}_h)$ and $y_h\in  \mathcal{R}T^0(\mathcal{T}_h)$,~there~holds
	\begin{align}
		(\nabla_hv_h,\Pi_h y_h)_{\Omega}+(\Pi_h v_h,\,\textup{div}\,y_h)_{\Omega}=(\pi_h v_h,y_h\cdot n)_{\partial\Omega}\,.\label{eq:pi}
	\end{align} 
    \newpage 
	
	\section{A (Fenchel) duality framework for an optimal insulation problem}\label{sec:continuous} 

    \hspace{5mm}In this section, we discuss a generalization of an optimal insulation problem originally proposed by \textsc{Buttazzo} (\textit{cf}.\ \cite{Buttazzo1988}) to bounded polyhedral Lipschitz domains and the possible presence of non-trivial Dirichlet and Neumann boundary data. For a detailed derivation,~we~refer~the~reader~to~\cite{AKKInsulationModel}.\medskip
    
	$\bullet$ \textit{Primal \hspace{-0.1mm}problem.} \hspace{-0.25mm}Given \hspace{-0.1mm}an \hspace{-0.1mm}\textit{amount \hspace{-0.1mm}of \hspace{-0.1mm}insulating \hspace{-0.1mm}material} \hspace{-0.1mm}$m\hspace{-0.175em}>\hspace{-0.175em}0$, \hspace{-0.1mm}a \hspace{-0.1mm}\textit{heat \hspace{-0.1mm}source~\hspace{-0.1mm}\mbox{density}}~\hspace{-0.1mm}${f\hspace{-0.175em}\in\hspace{-0.175em} L^2(\Omega)}$, a \textit{heat flux} $g\in H^{-\smash{\frac{1}{2}}}(\Gamma_N)$, and a \textit{Dirichlet boundary temperature distribution} $u_D\hspace{-0.1em}\in\hspace{-0.1em} H^{\smash{\frac{1}{2}}}(\Gamma_D)$~such that there exists a trace lift $\widehat{u}_D\hspace{-0.1em}\in \hspace{-0.1em}H^1(\Omega)$ (\textit{i.e.}, $\widehat{u}_D\hspace{-0.1em}=\hspace{-0.1em}u_D$ a.e.\ on $\Gamma_I$), the \textit{primal problem}~is~defined~as
    the  \hspace{-0.1mm}minimization \hspace{-0.1mm}of \hspace{-0.1mm}the \hspace{-0.1mm}\textit{primal \hspace{-0.1mm}energy \hspace{-0.1mm}functional} \hspace{-0.1mm}${I\colon  H^1(\Omega)\to\mathbb{R}\cup\{+\infty\}}$,~for~\mbox{every}~${v\in  H^1(\Omega)}$ \mbox{defined} by 
	\begin{align} \label{eq:primal}
		\begin{aligned} 
			I(v)&\coloneqq  \tfrac{1}{2}\| \nabla v\|_{\Omega}^2+\tfrac{1}{2m}\|v\|_{1,\Gamma_I}^2-(f,v)_{\Omega}-\langle g,v\rangle_{\Gamma_N}+\smash{I_{\{u_D\}}^{\Gamma_D}}(v) \,,
		\end{aligned}
	\end{align} 
    where $\smash{I_{\{u_D\}}^{\Gamma_D}}\colon H^{\frac{1}{2}}(\partial\Omega)\to \mathbb{R}\cup\{+\infty\}$, for every $\widehat{v}\in H^{\frac{1}{2}}(\partial\Omega)$, is defined by 
	\begin{align*}
		\smash{I_{\{u_D\}}^{\Gamma_D}}(\widehat{v})
		&\coloneqq 
		\begin{cases}
			0&\text{ if }\widehat{v}=u_D\text{ a.e.\ on }\Gamma_D\,,\\[-0.5mm]
			+\infty &\text{ else}\,.
		\end{cases}
	\end{align*}
    Then, the effective domain of the primal energy functional \eqref{eq:primal}  is given via 
    \begin{align*}
        K\coloneqq \textup{dom}(I)=\widehat{u}_D+H_D^1(\Omega)\,.
    \end{align*}
    Since the functional \eqref{eq:primal} is proper,  convex, weakly coercive, and lower semi-continuous,
the direct method in the calculus of variations yields the existence~of~a~minimizer~$u\in K$,~called~\textit{primal solution}. Here, the weak coercivity is a consequence of the Friedrich inequality~\eqref{lem:poin_cont}.~More~precisely, for every $v\in H^1(\Omega)$, one uses that
\begin{align*}
   \| \nabla v\|_{\Omega}^2+\tfrac{1}{m}\|v\|_{1,\Gamma_I}^2
    &\ge \min\big\{1,\tfrac{\vert \Gamma_I\vert}{m}\big\}\big\{\| \nabla v\|_{\Omega}^2+\vert\langle v\rangle_{\Gamma_I}\vert^2\big\}
    \\&\ge \tfrac{1}{2c_{\mathrm{F}}^2}\min\big\{1,\tfrac{\vert \Gamma_I\vert}{m}\big\}\|v\|_{\Omega}^2
    \,.
\end{align*}
In what follows, we always employ the notation $u\in K$ for primal solutions. In this connection,
%We reserve the notation $u\in K$ for primal solutions. 
note that, if $ \Gamma_D\neq \emptyset$ or $\Omega$ is connected, analogously to \cite[Sec.\ 5]{Buttazzo2017},  the~\mbox{functional} \eqref{eq:primal} is strictly convex and, consequently, the primal solution $u\in K$~is~uniquely~determined.\medskip\enlargethispage{4mm}

%$\bullet$ \emph{Primal variational inequality.} The primal solution $u\in K$ equivalently is a solution of the following variational inequality: 
%	for every $v\in K$, it holds that
%	\begin{align}
%	\tfrac{1}{2m}\|u\|_{1,\Gamma_I}^2+	(\nabla u,\nabla u-\nabla v)_{\Omega}\leq \tfrac{1}{2m}\|v\|_{1,\Gamma_I}^2+(f,u-v)_{\Omega}+\langle g,u-v\rangle_{\Gamma_N}\,.\label{eq:variational_ineq}
%	\end{align}

	$\bullet$  \textit{Dual problem.} A \textit{(Fenchel) dual problem} (in the sense of \cite[Rem.\ 4.2, p.\ 60/61]{ET99}) to the minimization of \eqref{eq:primal} 
    is given via the maximization of the \textit{dual energy functional} $D\colon \overline{H}_I(\textup{div};\Omega)\to \mathbb{R}\cup\{-\infty\}$, for every $y\in \overline{H}_I(\textup{div};\Omega)$ defined by
	\begin{align} \label{eq:dual} 
			D(y)\coloneqq   \left\{
        \begin{aligned} &-\tfrac{1}{2}\|y\|_{\Omega}^2-\tfrac{m}{2}\|\overline{y\cdot n}\|_{\infty,\Gamma_I}^2\\
            &+\langle y\cdot n,\widehat{u}_D\rangle_{\partial\Omega}-(\overline{y\cdot n},\widehat{u}_D)_{\Gamma_I}-\langle g,\widehat{u}_D\rangle_{\Gamma_N}\\&-I_{\{-f\}}^{\Omega}(\textup{div}\,y)-I_{\{g\}}^{\Gamma_N}(y\cdot n)  \,,
		\end{aligned}\right.
	\end{align} 
    where 
	$	\smash{I_{\{-f\}}^{\Omega}}\colon  L^2(\Omega)\to \mathbb{R}\cup \{+\infty\}$, for every $\widehat{v}\in  L^2(\Omega)$, is defined by
	\begin{align*}
		\smash{I_{\{-f\}}^{\Omega}}(\widehat{v}) \coloneqq
		\begin{cases}
			0&\text{ if }\widehat{v}=-f\text{ a.e.\ in }\Omega\,,\\[-0.5mm]
			+\infty& \text{ else}\,,
		\end{cases}
	\end{align*}
	and $	\smash{I_{\{g\}}^{\Gamma_N}}\colon  H^{-\smash{\frac{1}{2}}}(\partial\Omega)\to \mathbb{R}\cup \{+\infty\}$, for every $\widehat{v}\in H^{-\smash{\frac{1}{2}}}(\partial\Omega)$, is defined by
	\begin{align*}
		\smash{I_{\{g\}}^{\Gamma_N}}(\widehat{v})&\coloneqq \begin{cases}
			0&\text{ if }\langle \widehat{v},v\rangle_{\partial\Omega}=\langle g, v\rangle_{\Gamma_N}\text{ for all }v\in H^1_I(\Omega)\cap H^1_D(\Omega)\,,\\[-0.5mm]
			+\infty &\text{ else}\,.
		\end{cases}
	\end{align*}
    Then, the effective domain of the negative of the dual energy functional \eqref{eq:dual} is given via
    \begin{align*}
        K^*\coloneqq \textup{dom}(-D)%=\textup{dom}(\smash{I_{\{-f\}}^{\Omega}}\circ \textup{div})\cap \textup{dom}(\smash{I_{\{g\}}^{\Gamma_N}}\circ ((\cdot)\cdot n) )
        =
        \left\{y\in \overline{H}_I(\textup{div};\Omega)\;\bigg|\; \begin{aligned}
            \textup{div}\,y&=-f&&\text{a.e.\ in }\Omega\,,\\[-0.5mm]
            \langle y\cdot n ,v\rangle_{\partial\Omega}&=\langle g,v\rangle_{\Gamma_N}&&\text{for all }v\in H_I^1(\Omega)\cap H_D^1(\Omega)
        \end{aligned}\right\}\,.
    \end{align*}

  %$\bullet$ \emph{Dual variational inequality.} The dual solution $z\in K^*$ equivalently is a solution of the following variational inequality: 
	%for every $y\in K^*$, it holds that
	%\begin{align}
	%\tfrac{m}{2}\|\overline{z\cdot n}\|_{1,\Gamma_I}^2+	(z,z-y)_{\Omega}\geq \tfrac{m}{2}\|\overline{y\cdot n}\|_{\infty,\Gamma_I}^2-\langle z\cdot n-y\cdot n ,\widehat{u}_D\rangle_{\partial\Omega}+(\overline{z\cdot n}-\overline{y\cdot n} ,\widehat{u}_D)_{\Gamma_I}\,.\label{eq:variational_ineq}
	%\end{align}

    \newpage
    The following theorem proves that
    the maximization of \eqref{eq:dual} is  the 
    (Fenchel)~dual~\mbox{problem} (in the sense of \cite[Rem.\ 4.2, p.\ 60/61]{ET99}) to the minimization of \eqref{eq:primal}. In addition, it establishes the existence of a unique dual solution as well as the validity of a strong duality relation and convex optimality relations.\enlargethispage{7.5mm}
	
	\begin{theorem}[strong \hspace*{-0.1mm}duality \hspace*{-0.1mm}and \hspace*{-0.1mm}convex \hspace*{-0.1mm}optimality \hspace*{-0.1mm}relations]\label{thm:duality} \hspace*{-0.1mm}The \hspace*{-0.1mm}following \hspace*{-0.1mm}statements~\hspace*{-0.1mm}\mbox{apply}:
		\begin{itemize}[noitemsep,topsep=2pt,leftmargin=!,labelwidth=\widthof{(ii)}]
			\item[(i)]\hypertarget{thm:duality.i}{}  A (Fenchel) 
            dual problem to the minimization of \eqref{eq:primal} is given via the \mbox{maximization}~of~\eqref{eq:dual};  
			\item[(ii)]\hypertarget{thm:duality.ii}{}  There exists a unique maximizer $z\hspace{-0.1em}\in \hspace{-0.1em}\overline{H}_I(\textup{div};\Omega)$ of \eqref{eq:dual} satisfying the \emph{admissibility~\mbox{conditions}}
			\begin{alignat}{2}\label{eq:admissibility.1}
				\textup{div}\,z&=-f&&\quad\text{ a.e.\ in }\Omega\,,\\\label{eq:admissibility.2}
                \langle z\cdot n,v\rangle_{\partial\Omega}-(\overline{z\cdot n},v)_{\Gamma_I}&=\langle g,v\rangle_{\Gamma_N}&&\quad \text{ for all }v\in \smash{H^1_D(\Omega)}\,.
			\end{alignat}
            %\textit{i.e.}, $z\in H(\textup{div}\!=\!-f,\Omega)$.
			In addition, there
			holds a \emph{strong duality relation}, \textit{i.e.}, we have that
			\begin{align}
				I(u) = D(z)\,;\label{eq:strong_duality}
			\end{align}
			\item[(iii)]\hypertarget{thm:duality.iii}{} There hold  \emph{convex optimality relations}, \textit{i.e.}, we have that
			\begin{align}
				z&=\nabla u\quad\text{ a.e.\ in }\Omega\,,\label{eq:optimality.1}\\
                -(\overline{z\cdot n},u)_{\Gamma_I}&=\smash{\tfrac{m}{2}\|\overline{z\cdot n}\|_{\infty,\Gamma_I}^2+\tfrac{1}{2m}\|u\|_{1,\Gamma_I}^2}\,.\label{eq:optimality.2}
			\end{align}
		\end{itemize}
	\end{theorem}

    \begin{remark}[equivalent condition to \eqref{eq:optimality.2}]
        Note that,~by~the~standard equality condition in the Fenchel--Young inequality (\textit{cf}.\ \cite[Prop.\ 5.1,~p.~21]{ET99}) and the chain rule for the subdifferential (\textit{cf}.\ \cite[Thm.\ 4.19]{Clason2020}), the convex optimality relation \eqref{eq:optimality.2} is equivalent to\vspace{-1mm}
        \begin{align}\label{eq:optimality.3}
        -\overline{z\cdot n} \in \tfrac{1}{m}(\partial\vert\cdot\vert)(u)\|u\|_{1,\Gamma_I}\quad \text{ a.e.\ on }\Gamma_I\,.
        \end{align}
    \end{remark}
	
	\begin{proof}[Proof (of Theorem \ref{thm:duality}).]
		\textit{ad (\hyperlink{thm:duality.i}{i}).} 
		To begin with, we need to bring the primal energy functional~\eqref{eq:primal} into the form of a primal energy functional in  the sense of Fenchel (\textit{cf}.~\cite[Rem.\ 4.2, p.\ 60/61]{ET99}),~\textit{i.e.}, 
        \begin{align*}
			I(v)= G(\nabla v)+F(v)\,,
		\end{align*}
        where $G\colon (L^2(\Omega))^d\to \mathbb{R}\cup\{+\infty\}$ and $F\colon  H^1(\Omega)\to \mathbb{R}\cup\{+\infty\}$ should be proper, convex, and lower semi-continuous functionals. To this end, let us introduce the functionals $G\colon (L^2(\Omega))^d\to \mathbb{R}$ and ${F\colon  H^1(\Omega)\to \mathbb{R}\cup\{+\infty\}}$, for every $y\in  (L^2(\Omega))^d$ and $v\in H^1(\Omega)$, respectively,~defined~by
		\begin{align*}
			G(y)&\coloneqq \tfrac{1}{2}\|y\|_{\Omega}^2\,,\\
			F(v)&\coloneqq -(f,v)_{\Omega}-\langle g,v\rangle_{\Gamma_N}+\tfrac{1}{2m}\|v\|_{1,\Gamma_I}^2+\smash{I_{\{u_D\}}^{\Gamma_D}}(v)\,.
		\end{align*}
		Then, \hspace{-0.1mm}according \hspace{-0.1mm}to \hspace{-0.1mm}\cite[Rem.\ 4.2, p.\ 60/61]{ET99}, \hspace{-0.1mm}the \hspace{-0.1mm}(Fenchel) \hspace{-0.1mm}dual \hspace{-0.1mm}problem \hspace{-0.1mm}to \hspace{-0.1mm}the \hspace{-0.1mm}minimization~\hspace{-0.1mm}of~\hspace{-0.1mm}\eqref{eq:primal} is given via the maximization of $D\colon  (L^2(\Omega))^d\to \mathbb{R}\cup \{-\infty\}$, for every $y\in  (L^2(\Omega))^d$~defined~by 
		\begin{align}\label{prop:duality.1}
			D(y)\coloneqq -G^*(y)- F^*(-\nabla^*y)\,,
		\end{align}
		where we denote by $\nabla^*\colon  (L^2(\Omega))^d\to (H^1(\Omega))^*$ the adjoint operator to %the gradient operator 
        $\nabla \colon H^1(\Omega)\to (L^2(\Omega))^d$. 
		
		$\bullet$ First, resorting to \cite[Prop.\ 4.2, p.\ 19]{ET99}, for every $y\in (L^2(\Omega))^d$,  we find that
		\begin{align}\label{prop:duality.2}
			G^*(y)= \tfrac{1}{2}\|y\|_{\Omega}^2\,.
		\end{align} 
		
		$\bullet$  Second, using the integration-by-parts formula \eqref{eq:pi_cont}, for every $y\in (L^2(\Omega))^d$,
        it~turns~out~that
        \begin{align}\label{prop:duality.3}
			\begin{aligned}
				&F^*(-\nabla^*y)
				=\sup_{v\in H^1(\Omega)}{\left\{
                \begin{aligned}
                &-(y,\nabla v)_{\Omega}+(f,v)_{\Omega}+\langle g,v\rangle_{\Gamma_N}\\&-\tfrac{1}{2m}\|v\|_{1,\Gamma_I}^2-\smash{I_{\{u_D\}}^{\Gamma_D}}(v)
                \end{aligned}\right\}}\\&\quad
                =\sup_{v\in H^1_D(\Omega)}{\left\{\begin{aligned}&-(y,\nabla (v+\widehat{u}_D))_{\Omega}+(f,v+\widehat{u}_D)_{\Omega}+\langle g,v+\widehat{u}_D\rangle_{\Gamma_N}\\&-\tfrac{1}{2m}\|v+\widehat{u}_D\|_{1,\Gamma_I}^2 \end{aligned}\right\}} 
				\\&\quad=\begin{cases}
                \left.\begin{aligned}
				    &I_{\{-f\}}^{\Omega}(\textup{div}\,y)
                +I_{\{g\}}^{\Gamma_N}(y\cdot n)-\langle y\cdot n,\widehat{u}_D\rangle_{\partial\Omega}+\langle g,\widehat{u}_D\rangle_{\Gamma_N}\\&
    \quad+\sup_{v\in H^1_D(\Omega)}{\big\{\langle g,v\rangle_{\Gamma_N}-\langle y\cdot n,v\rangle_{\partial\Omega}-\tfrac{1}{2m}\|v+\widehat{u}_D\|_{1,\Gamma_I}^2\big\}}
    \end{aligned}\right\}&\text{ if }y\in H(\textup{div};\Omega)\,,\\[-1mm]
    +\infty&\text{ else}\,,
				\end{cases}  
			\end{aligned}\hspace{-7.5mm}
		\end{align}
        where, due to Lemma \ref{lem:normal_trace} as well as the density of  $(\textrm{tr}(\cdot)|_{\Gamma_I})(K)$ in $L^1(\Gamma_I)$, for every $y\in H(\textup{div};\Omega)$, we have that
        \begin{align}\label{prop:duality.4}
        \begin{aligned} 
            &\sup_{v\in H^1_D(\Omega)}{\big\{\langle g,v\rangle_{\Gamma_N}-\langle y\cdot n,v\rangle_{\partial\Omega}-\tfrac{1}{2m}\|v+\widehat{u}_D\|_{1,\Gamma_I}^2\big\}}
            \\&\quad= \sup_{\rho\ge 0}{ \sup_{\substack{v\in K\\ \|v\|_{1,\Gamma_I}=\rho}}{\big\{\langle g,v-\widehat{u}_D\rangle_{\Gamma_N}-\langle y\cdot n,v-\widehat{u}_D\rangle_{\partial\Omega}-\tfrac{1}{2m}\rho^2\big\}}}
            \\[-1mm]&\quad=\begin{cases} \underset{\rho\ge 0}{\sup}{\underset{\substack{v\in L^1(\Gamma_I)\\ \|v\|_{1,\Gamma_I}=\rho}}{\sup}{\big\{(\overline{y\cdot n},\widehat{u}_D-v)_{\Gamma_I}-\tfrac{1}{2m}\rho^2\big\}}}&\text{ if }y\in \overline{H}_I(\textup{div};\Omega)\,,\\[-1mm]
                +\infty&\text{ else}\,,
            \end{cases} 
            \\&\quad=\begin{cases}
                (\overline{y\cdot n},\widehat{u}_D)_{\Gamma_I}+\smash{\underset{\rho\ge 0}{\sup}{\big\{\rho\,\|\overline{y\cdot n}\|_{\infty,\Gamma_I}-\tfrac{1}{2m}\rho^2\big\}}}&\text{ if }y\in \overline{H}_I(\textup{div};\Omega)\,,\\
                +\infty&\text{ else}\,,
            \end{cases} 
            \\&\quad=\begin{cases}
               (\overline{y\cdot n},\widehat{u}_D)_{\Gamma_I}+\tfrac{m}{2}\|\overline{y\cdot n}\|_{\infty,\Gamma_I}^2&\text{ if }y\in \overline{H}_I(\textup{div};\Omega)\,,\\
                +\infty&\text{ else}\,.
            \end{cases} 
            \end{aligned}
        \end{align}
        Then, using \eqref{prop:duality.2} and \eqref{prop:duality.3} together with \eqref{prop:duality.4} in \eqref{prop:duality.1}, for every $y\in (L^2(\Omega))^d$, we arrive at 
        \begin{align}\label{prop:duality.5}
            D(y)=\begin{cases}
                \left.\begin{aligned} 
               &-\tfrac{1}{2}\|y\|_{\Omega}^2-\tfrac{m}{2}\|\overline{y\cdot n}\|_{\infty,\Gamma_I}^2 \\&+\langle y\cdot n,\widehat{u}_D\rangle_{\partial\Omega}-(\overline{y\cdot n},\widehat{u}_D)_{\Gamma_I}-\langle g,\widehat{u}_D\rangle_{\Gamma_N}\\
              &-I_{\{-f\}}^{\Omega}(\textup{div}\,y)-I_{\{g\}}^{\Gamma_N}(y\cdot n) 
                \end{aligned}\right\}&\text{ if } 
                y\in \overline{H}_I(\textup{div};\Omega)\,,\\
                +\infty&\text{ else}\,.
            \end{cases}
        \end{align} 
         Eventually, since $D= -\infty$ in $\smash{(L^2(\Omega))^d\setminus \overline{H}_I(\textup{div};\Omega)}$, it is enough to restrict \eqref{prop:duality.5} to $\smash{\overline{H}_I(\textup{div};\Omega)}$.
		
		\textit{ad (\hyperlink{thm:duality.ii}{ii}).} Since $G\colon (L^2(\Omega))^d\to \mathbb{R}$ and $F\colon H^1(\Omega)\to \mathbb{R}\cup\{+\infty\}$ are proper, convex,~and lower semi-continuous and  since 
		$G\colon  (L^2(\Omega))^d\to \mathbb{R}$  is continuous at
		$\nabla \widehat{u}_D\in \textup{dom}(G)$~with~${\widehat{u}_D\in \textup{dom}(F)}$, %~\textit{i.e.},  we have that
		%\begin{align*}
		%	G(y)\to G(0)\quad(y\to 0\quad\text{ in }(L^2(\Omega))^d)\,,
		%\end{align*}
        resorting to the Fenchel duality theorem (\emph{cf}.\ \cite[Rem.\ 4.2, (4.21), p.\ 61]{ET99}),  we obtain the existence of a maximizer $z\in  (L^2(\Omega))^d$ of  \eqref{prop:duality.1} and that a strong duality relation applies, \textit{i.e.},~we~have~that
		\begin{align}\label{strong}
			I(u)=D(z)\,.
		\end{align}
		Inasmuch as $D= -\infty$ in $\smash{(L^2(\Omega))^d\setminus \overline{H}_I(\textup{div};\Omega)}$, from \eqref{strong}, we infer that $z\in  \smash{\overline{H}_I(\textup{div};\Omega)}$ and that the admissibility conditions \eqref{eq:admissibility.1},\eqref{eq:admissibility.2} are satisfied. Furthermore, since \eqref{eq:dual}~is~strictly~concave,  the maximizer $z\in \smash{\overline{H}_I(\textup{div};\Omega)}$ is uniquely determined.
  
		\textit{ad (\hyperlink{thm:duality.iii}{iii}).} \hspace*{-0.1mm}By \hspace*{-0.1mm}the \hspace*{-0.1mm}standard \hspace*{-0.1mm}(Fenchel) \hspace*{-0.1mm}convex \hspace*{-0.1mm}duality \hspace*{-0.1mm}theory \hspace*{-0.1mm}(\emph{cf}.\ \hspace*{-0.1mm}\cite[Rem.\ \hspace*{-0.1mm}4.2,~\hspace*{-0.1mm}(4.24),~\hspace*{-0.1mm}(4.25),\hspace*{-0.1mm}~p.~\hspace*{-0.1mm}61]{ET99}), there hold the convex optimality relations\vspace{-0.5mm}
		\begin{align}
			-\nabla^*z &\in \partial F(u)\,,\label{prop:duality.6}\\
			z&\in \partial G(\nabla u)\,.\label{prop:duality.7}
		\end{align} 
		The inclusion \eqref{prop:duality.7} is equivalent to the convex optimality condition \eqref{eq:optimality.1}. The~inclusion~\eqref{prop:duality.6}, by the definition of the subdifferential and, then, using the integration-by-parts formula \eqref{eq:pi_cont}, is equivalent to that for every $v\in K$,~there~holds
        \begin{align*}
			%\begin{aligned} 
            \smash{\tfrac{1}{2m}\|v\|_{1,\Gamma_I}^2-\tfrac{1}{2m}\|u\|_{1,\Gamma_I}^2}&\ge (f,v-u)_{\Omega}+\langle g,v-u\rangle_{\Gamma_N}-(z,\nabla v-\nabla u)_{\Omega}\,.%\\
           % &=-(\overline{z\cdot n},v-u)_{\Gamma_I} \,.
            %\end{aligned}\label{prop:duality.8}
		\end{align*}
        %and 
        Then, by admissibility conditions \eqref{eq:admissibility.1},\eqref{eq:admissibility.2}, this is equivalent to that for every $v\in K$,~there~holds\enlargethispage{5mm}
		\begin{align}
			\begin{aligned} \smash{\tfrac{1}{2m}\|v\|_{1,\Gamma_I}^2-\tfrac{1}{2m}\|u\|_{1,\Gamma_I}^2}%&\ge (f,v-u)_{\Omega}+\langle g,v-u\rangle_{\Gamma_N}-(z,\nabla v-\nabla u)_{\Omega}\\
            %&
            \ge -(\overline{z\cdot n},v-u)_{\Gamma_I} \,.
            \end{aligned}\label{prop:duality.8}
		\end{align}
        Eventually, due to the density of 
        $(\textup{tr}(\cdot)|_{\Gamma_I})(K)$ in $L^1(\Gamma_I)$, from \eqref{prop:duality.8}, we infer that 
        \begin{align*}
            -\overline{z\cdot n}\in \smash{\partial (\tfrac{1}{2m}\|\cdot\|_{1,\Gamma_I}^2)(u)}\,,
        \end{align*}
        which, \hspace{-0.1mm}by \hspace{-0.1mm}the \hspace{-0.1mm}standard \hspace{-0.1mm}equality \hspace{-0.1mm}condition \hspace{-0.1mm}in \hspace{-0.1mm}the \hspace{-0.1mm}Fenchel--Young inequality \hspace{-0.1mm}(\textit{cf}.\ \hspace{-0.1mm}\cite[Prop.~\hspace{-0.1mm}5.1,~\hspace{-0.1mm}p.~\hspace{-0.1mm}21]{ET99}), is equivalent to \eqref{eq:optimality.2}.
	\end{proof}\newpage

 	\section{\emph{A posteriori} error analysis}\label{sec:aposteriori} 
    
	\hspace{5mm}In this section, following an  \emph{a posteriori} error analysis framework  based on convex duality arguments
    from \cite{ABKK2024} (see also \cite{BarGudKal24}), we derive an \emph{a posteriori} error identity for arbitrary admissible approximations of  the primal problem \eqref{eq:primal} and the~dual~problem~\eqref{eq:dual}.~To~this~end, we introduce the 
	\emph{primal-dual gap estimator} ${\eta^2_{\textup{gap}}\colon K\times K^*\to [0,+\infty)}$, for every $v\in K$ and $y\in K^*$ defined by 
	\begin{align}\label{eq:primal-dual.1}
		\begin{aligned}
			\eta^2_{\textup{gap}}(v,y)&\coloneqq I(v)-D(y)\,. 
		\end{aligned}
	\end{align}
	The primal-dual gap estimator \eqref{eq:primal-dual.1} measures  the accuracy of admissible approximations~of~the primal problem \eqref{eq:primal} and the dual problem \eqref{eq:dual} at the same time via measuring the respective violation of the strong duality relation \eqref{eq:strong_duality}.
    More precisely, the primal-dual~gap~\mbox{estimator}~\eqref{eq:primal-dual.1}~splits into \hspace{-0.1mm}two \hspace{-0.1mm}contributions \hspace{-0.1mm}that \hspace{-0.1mm}each \hspace{-0.1mm}measure \hspace{-0.1mm}the \hspace{-0.1mm}violation \hspace{-0.1mm}of \hspace{-0.1mm}the \hspace{-0.1mm}convex~\hspace{-0.1mm}optimality~\hspace{-0.1mm}relations~\hspace{-0.1mm}\eqref{eq:optimality.1},\eqref{eq:optimality.2}.\enlargethispage{6mm} %In fact, the convex optimality relations \eqref{eq:optimality.1},\eqref{eq:optimality.2} are equivalent to the strong duality relation \eqref{eq:strong_duality}.
	
	\begin{lemma}[representation of primal-dual gap estimator]\label{lem:primal_dual_gap_estimator}
		For every $v\in K$ and $y\in K^*$, we have that
		\begin{align*} 
			\eta_{\textup{gap}}^2(v,y)\coloneqq \eta_{\textup{gap},A}^2(v,y)+\eta_{\textup{gap},B}^2(v,y)\,,\hspace{2.225cm}\\
			\quad\text{where}\left\{\quad\begin{aligned}\eta_{\textup{gap},A}^2(v,y)&\coloneqq \tfrac{1}{2}\|\nabla v-y\|_{\Omega}^2\,,\\
			\eta_{\textup{gap},B}^2(v,y)&\coloneqq \tfrac{m}{2}\|\overline{y\cdot n}\|_{\infty,\Gamma_I}^2+(\overline{y\cdot n}, v)_{\Gamma_I}+\tfrac{1}{2m}\|v\|_{1,\Gamma_I}^2\,. 
			\end{aligned}\right.
		\end{align*}
	\end{lemma}
	
	\begin{remark}[interpretation of the components of the primal-dual gap estimator]\hphantom{                   }
		\begin{itemize}[noitemsep,topsep=2pt,leftmargin=!,labelwidth=\widthof{(iii)},font=\itshape]
			\item[(i)] The estimator $\eta_{\textup{gap},A}^2$ measures the violation of the convex optimality relation \eqref{eq:optimality.1};
			\item[(ii)] The estimator $\eta_{\textup{gap},B}^2$ measures the violation of the convex optimality relation \eqref{eq:optimality.2}. Moreover,  by the Fenchel--Young inequality (\textit{cf}.\ \cite[Prop.\ 5.1, p.\ 21]{ET99}), for every $v\in H^1(\Omega)$ and $y\in \overline{H}_I(\textup{div},\Omega)$, we have that
        \begin{align*}
            \tfrac{m}{2}\|\overline{y\cdot n}\|_{\infty,\Gamma_I}^2+(\overline{y\cdot n}, v)_{\Gamma_I}+\tfrac{1}{2m}\|v\|_{1,\Gamma_I}^2\ge 0\,.
        \end{align*}
		\end{itemize}
	\end{remark}
	
	\begin{proof}[Proof (of Lemma \ref{lem:primal_dual_gap_estimator}).]\let\qed\relax 
		For every $v\in K$ and $y\in K^*$, using the admissibility condition \eqref{eq:admissibility.1}, the integration-by-parts formula \eqref{eq:pi_cont},  the binomial formula, and the admissibility condition \eqref{eq:admissibility.2} together with $v-\widehat{u}_D\in H^1_D(\Omega)$, 
	       we find that
		\begin{align*}
			I(v)-D(y)&= \tfrac{1}{2}\| \nabla v\|_{\Omega}^2-(f,v)_{\Omega}-\langle g,v\rangle_{\Gamma_N}+\tfrac{1}{2}\| y\|_{\Omega}^2\\&\quad+\tfrac{m}{2}\|\overline{y\cdot n}\|_{\infty,\Gamma_I}^2-\langle
            y\cdot n,\widehat{u}_D\rangle_{\partial\Omega}+(\overline{y\cdot n},\widehat{u}_D)_{\Gamma_I}+\langle g,\widehat{u}_D\rangle_{\Gamma_N}+\tfrac{1}{2m}\|v\|_{1,\Gamma_I}^2\\&
			= \tfrac{1}{2}\| \nabla v\|_{\Omega}^2+(\textup{div}\,y,v)_{\Omega}+\tfrac{1}{2}\| y\|_{\Omega}^2-\langle g,v-\widehat{u}_D\rangle_{\Gamma_N}\\&\quad+\tfrac{m}{2}\|\overline{y\cdot n}\|_{\infty,\Gamma_I}^2-\langle
            y\cdot n,\widehat{u}_D\rangle_{\partial\Omega}+(\overline{y\cdot n},\widehat{u}_D)_{\Gamma_I}+\tfrac{1}{2m}\|v\|_{1,\Gamma_I}^2
			\\&
			= \tfrac{1}{2}\| \nabla v\|_{\Omega}^2-(y,\nabla v)_{\Omega}+\tfrac{1}{2}\| y\|_{\Omega}^2+\langle
            y\cdot n,v-\widehat{u}_D\rangle_{\partial\Omega}-\langle g,v-\widehat{u}_D\rangle_{\Gamma_N}\\&\quad+\tfrac{m}{2}\|\overline{y\cdot n}\|_{\infty,\Gamma_I}^2+(\overline{y\cdot n},\widehat{u}_D)_{\Gamma_I}+\tfrac{1}{2m}\|v\|_{1,\Gamma_I}^2
			\\&
			= \tfrac{1}{2}\| \nabla v-y\|_{\Omega}^2+(\overline{y\cdot n}, v-\widehat{u}_D)_{\Gamma_I}\\&\quad
            +\tfrac{m}{2}\|\overline{y\cdot n}\|_{\infty,\Gamma_I}^2+(\overline{y\cdot n}, \widehat{u}_D)_{\Gamma_I}+\tfrac{1}{2m}\|v\|_{1,\Gamma_I}^2 \,.
            \\&
			= \tfrac{1}{2}\| \nabla v-y\|_{\Omega}^2+\tfrac{m}{2}\|\overline{y\cdot n}\|_{\infty,\Gamma_I}^2+(\overline{y\cdot n}, v)_{\Gamma_I}+\tfrac{1}{2m}\|v\|_{1,\Gamma_I}^2 \,.
            \tag*{$\qedsymbol$}
		\end{align*}
	\end{proof}
	Next, \hspace{-0.1mm}as \hspace{-0.1mm}per \hspace{-0.1mm}\cite{BarGudKal24,ABKK2024}, \hspace{-0.1mm}as \hspace{-0.1mm}\emph{`natural'} \hspace{-0.1mm}error \hspace{-0.1mm}quantities \hspace{-0.1mm}in \hspace{-0.1mm}the \hspace{-0.1mm}primal-dual \hspace{-0.1mm}gap \hspace{-0.1mm}identity \hspace{-0.1mm}(\textit{cf}.~\hspace{-0.1mm}\mbox{Theorem}~\hspace{-0.1mm}\ref{thm:prager_synge_identity}), we employ the 
     \emph{optimal strong convexity measures} for the primal energy functional~\eqref{eq:primal} at a primal solution $u\hspace{-0.15em}\in\hspace{-0.15em} K$, \textit{i.e.}, 
	$\rho_I^2\colon \hspace{-0.15em}K\hspace{-0.15em}\to \hspace{-0.15em}[0,+\infty)$,  and for the negative of the dual~energy~\mbox{functional}~\eqref{eq:dual}~at~the dual solution $z\hspace{-0.15em}\in\hspace{-0.15em} K^*$, \textit{i.e.}, $\rho_{-D}^2\colon \hspace{-0.15em}K^*\hspace{-0.15em}\to\hspace{-0.15em}  [0,+\infty)$,  
    for every $v\hspace{-0.15em}\in\hspace{-0.15em} K$ and $y\hspace{-0.15em}\in \hspace{-0.15em}K^*$,~\mbox{respectively},~\mbox{defined}~by
	\begin{align}\label{def:optimal_primal_error}
		\rho_I^2(v)&\coloneqq I(v)-I(u)\,,\\
        \rho_{-D}^2(y)&\coloneqq- D(y)+D(z)\,.\label{def:optimal_dual_error}
	\end{align}
    As for the primal-dual gap estimator \eqref{eq:primal-dual.1} in Lemma \ref{lem:primal_dual_gap_estimator}, the optimal strong convexity~\mbox{measures} \eqref{def:optimal_primal_error},\eqref{def:optimal_dual_error} \hspace{-0.1mm}split \hspace{-0.1mm}into \hspace{-0.1mm}two \hspace{-0.1mm}contributions \hspace{-0.1mm}that \hspace{-0.1mm}each \hspace{-0.1mm}measure \hspace{-0.1mm}the \hspace{-0.1mm}accuracy \hspace{-0.1mm}of \hspace{-0.1mm}admissible~\hspace{-0.1mm}\mbox{approximations} in terms of the violation of the convex optimality relations \eqref{eq:optimality.1},\eqref{eq:optimality.2}.
    \pagebreak
	
	\begin{lemma}[representations of the optimal strong convexity measures]\label{lem:strong_convexity_measures}
		The following statements apply:
		\begin{itemize}[noitemsep,topsep=2pt,leftmargin=!,labelwidth=\widthof{(ii)}]
			\item[(i)]\hypertarget{lem:strong_convexity_measures.i}{} For every $v\in K$, we have that
			\begin{align*}
				\rho_I^2(v)=\tfrac{1}{2}\|\nabla v-\nabla u\|_{\Omega}^2+\tfrac{m}{2}\|\overline{z\cdot n}\|_{\infty,\Gamma_I}^2+(\overline{z\cdot n}, v)_{\Gamma_I}+\tfrac{1}{2m}\|v\|_{1,\Gamma_I}^2\,.
			\end{align*}
			\item[(ii)]\hypertarget{lem:strong_convexity_measures.ii}{} For every $y\in K^*$, we have that
			\begin{align*}
				\rho_{-D}^2(y)=\tfrac{1}{2}\|y-z\|_{\Omega}^2+\tfrac{m}{2}\|\overline{y\cdot n}\|_{\infty,\Gamma_I}^2+(\overline{y\cdot n}, u)_{\Gamma_I}+\tfrac{1}{2m}\|u\|_{1,\Gamma_I}^2\,.
			\end{align*}
		\end{itemize}
	\end{lemma} 
	
	\begin{proof}\let\qed\relax
		
		\emph{ad (\hyperlink{lem:strong_convexity_measures.i}{i}).} For every $v\in K$, using the admissibility condition \eqref{eq:admissibility.1}, the integration-by-parts formula \eqref{eq:pi_cont}, the convex optimality relation \eqref{eq:optimality.1},
        the admissibility condition \eqref{eq:admissibility.2} together with $v-u\in H^1_D(\Omega)$, 
  the binomial formula, and 
  the convex optimality relation \eqref{eq:optimality.2},
  we find that
		\begin{align*}
			I(v)-I(u)&=\tfrac{1}{2}\|\nabla v\|_{\Omega}^2-\tfrac{1}{2}\|\nabla u\|_{\Omega}^2-(f,v-u)_{\Omega}-\langle g,v-u\rangle_{\Gamma_N}\\&\quad+\tfrac{1}{2m}\|v\|_{1,\Gamma_I}^2-\tfrac{1}{2m}\|u\|_{1,\Gamma_I}^2
    \\&=\tfrac{1}{2}\|\nabla v\|_{\Omega}^2-\tfrac{1}{2}\|\nabla u\|_{\Omega}^2+(\textup{div}\,z,v-u)_{\Omega}\\&\quad+\tfrac{1}{2m}\|v\|_{1,\Gamma_I}^2-\tfrac{1}{2m}\|u\|_{1,\Gamma_I}^2-\langle g,v-u\rangle_{\Gamma_N}\\&
			=\tfrac{1}{2}\|\nabla v\|_{\Omega}^2-\tfrac{1}{2}\|\nabla u\|_{\Omega}^2-(z,\nabla v-\nabla u)_{\Omega}
            \\&\quad+\tfrac{1}{2m}\|v\|_{1,\Gamma_I}^2-\tfrac{1}{2m}\|u\|_{1,\Gamma_I}^2+\langle z\cdot n,v-u\rangle_{\partial\Omega}-\langle g,v-u\rangle_{\Gamma_N}
            \\&
			=\tfrac{1}{2}\|\nabla v\|_{\Omega}^2-\tfrac{1}{2}\|\nabla u\|_{\Omega}^2-(\nabla u,\nabla v-\nabla u)_{\Omega}
            \\&\quad+\tfrac{1}{2m}\|v\|_{1,\Gamma_I}^2-\tfrac{1}{2m}\|u\|_{1,\Gamma_I}^2+(\overline{z\cdot n},v-u)_{\Gamma_I}
			\\&
			=\tfrac{1}{2}\|\nabla v-\nabla u\|_{\Omega}^2+\tfrac{m}{2}\|\overline{z\cdot n}\|_{\infty,\Gamma_I}^2+(\overline{z\cdot n},v)_{\Gamma_I}+\tfrac{1}{2m}\|v\|_{1,\Gamma_I}^2\,.
		\end{align*}
		
		\emph{ad (\hyperlink{lem:strong_convexity_measures.ii}{ii}).} For every $y\in K^*$, using the binomial formula, the convex optimality relation \eqref{eq:optimality.1}, the integration-by-parts  formula \eqref{eq:pi_cont}, the admissibility condition \eqref{eq:admissibility.1}, the admissibility condition \eqref{eq:admissibility.2} together with $u-\widehat{u}_D\in H_D^1(\Omega)$,
        and the convex optimality relation \eqref{eq:optimality.2},~we~find~that
		\begin{align*}
			-D(y)+D(z)&=\tfrac{1}{2}\|y\|_{\Omega}^2-\tfrac{1}{2}\|z\|_{\Omega}^2+\tfrac{m}{2}\|\overline{y\cdot n}\|_{\infty,\Gamma_I}^2-\tfrac{m}{2}\|\overline{z\cdot n}\|_{\infty,\Gamma_I}^2\\&\quad-
            \langle y\cdot n-z\cdot n,\widehat{u}_D\rangle_{\partial\Omega}+(\overline{y\cdot n}-\overline{z\cdot n},\widehat{u}_D)_{\Gamma_I}\\&
			=\tfrac{1}{2}\|y-z\|_{\Omega}^2+(z,y-z)_{\Omega}+\tfrac{m}{2}\|\overline{y\cdot n}\|_{\infty,\Gamma_I}^2-\tfrac{m}{2}\|\overline{z\cdot n}\|_{\infty,\Gamma_I}^2
            \\&\quad-
            \langle y\cdot n-z\cdot n,\widehat{u}_D\rangle_{\partial\Omega}+(\overline{y\cdot n}-\overline{z\cdot n},\widehat{u}_D)_{\Gamma_I}
   \\&
			=\tfrac{1}{2}\|y-z\|_{\Omega}^2+(\nabla u,y-z)_{\Omega}+\tfrac{m}{2}\|\overline{y\cdot n}\|_{\infty,\Gamma_I}^2-\tfrac{m}{2}\|\overline{z\cdot n}\|_{\infty,\Gamma_I}^2\\&\quad-
            \langle y\cdot n-z\cdot n,\widehat{u}_D\rangle_{\partial\Omega}+(\overline{y\cdot n}-\overline{z\cdot n},\widehat{u}_D)_{\Gamma_I}
   \\&
			=\tfrac{1}{2}\|y-z\|_{\Omega}^2-(\textup{div}\,y-\textup{div}\,z,u)_{\Omega}+\langle y\cdot n- z\cdot n,u-\widehat{u}_D\rangle_{\partial\Omega}\\&\quad+\tfrac{m}{2}\|\overline{y\cdot n}\|_{\infty,\Gamma_I}^2-\tfrac{m}{2}\|\overline{z\cdot n}\|_{\infty,\Gamma_I}^2+(\overline{y\cdot n}-\overline{z\cdot n},\widehat{u}_D)_{\Gamma_I}
            \\&
			=\tfrac{1}{2}\|y-z\|_{\Omega}^2+(\overline{y\cdot n}-\overline{z\cdot n},u-\widehat{u}_D)_{\Gamma_I}\\&\quad+\tfrac{m}{2}\|\overline{y\cdot n}\|_{\infty,\Gamma_I}^2-\tfrac{m}{2}\|\overline{z\cdot n}\|_{\infty,\Gamma_I}^2
        +(\overline{y\cdot n}-\overline{z\cdot n},\widehat{u}_D)_{\Gamma_I}
			\\&
			=\tfrac{1}{2}\|y-z\|_{\Omega}^2+\tfrac{m}{2}\|\overline{y\cdot n}\|_{\infty,\Gamma_I}^2+(\overline{y\cdot n},u )_{\Gamma_I}+\tfrac{1}{2m}\|u\|_{1,\Gamma_I}^2 \,.\tag*{$\qedsymbol$}
		\end{align*}
	\end{proof}
	
	Eventually, we establish an \emph{a posteriori} error identity that identifies the \emph{primal-dual~total~error} $\rho_{\textup{tot}}^2\colon K\times K^*\to [0,+\infty)$, for every $v\in K$ and $y\in K^*$ defined by 
	\begin{align}\label{def:primal_dual_total_error}
		\rho_{\textup{tot}}^2(v,y)\coloneqq \rho_I^2(v)+\rho_{-D}^2(y)\,,
	\end{align}
	with the primal-dual gap estimator \eqref{eq:primal-dual.1}. \enlargethispage{5mm}

	\begin{theorem}[primal-dual gap identity]\label{thm:prager_synge_identity}
		For every $v\in K$ and $y\in K^*$, we have that
		\begin{align*}
			\rho_{\textup{tot}}^2(v,y)
			=\eta_{\textup{gap}}^2(v,y)\,.
		\end{align*}
	\end{theorem}
	
	\begin{proof}
		We combine the definitions \eqref{eq:primal-dual.1}--\eqref{def:primal_dual_total_error} using the strong duality relation \eqref{eq:strong_duality}.
	\end{proof}\newpage

    \section{A (Fenchel) duality framework for a discrete optimal insulation problem}\label{sec:discrete}\vspace{-1mm}
	
	\hspace{5mm}In this section, we propose approximations of
    the primal problem \eqref{eq:primal} using the Crouzeix--Raviart \hspace{-0.1mm}element \hspace{-0.1mm}(\textit{cf}.\ \eqref{def:CR})  \hspace{-0.1mm}and \hspace{-0.1mm}the \hspace{-0.1mm}dual \hspace{-0.1mm}problem \hspace{-0.1mm}  \hspace{-0.1mm}\eqref{eq:dual} \hspace{-0.1mm}using \hspace{-0.1mm}the \hspace{-0.1mm}Raviart--Thomas~\hspace{-0.1mm}\mbox{element}~\hspace{-0.1mm}(\textit{cf}.~\hspace{-0.1mm}\eqref{def:RT}).\medskip% By analogy with Section \ref{sec:continuous}, establish a (Fenchel) duality framework.\medskip
	
	\hspace*{-2.5mm}$\bullet$ \emph{Discrete primal problem.} Let $f_h\in  \mathcal{L}^0(\mathcal{T}_h)$, $g_h\in \mathcal{L}^0(\mathcal{S}_h^N)$, and $u_D^h\in \mathcal{L}^0(\mathcal{S}_h^D)$ be approximations of $f\in L^2(\Omega)$, $g\in H^{-\smash{\frac{1}{2}}}(\Gamma_N)$, and $u_D\in H^{\smash{\frac{1}{2}}}(\Gamma_D)$, respectively. Then, the \emph{discrete primal problem} is defined as the minimization of the \textit{discrete primal energy functional} $I_h^{cr}\colon \mathcal{S}^{1,cr}(\mathcal{T}_h)\to \mathbb{R}\cup\{+\infty\}$, for every $v_h\in \mathcal{S}^{1,cr}(\mathcal{T}_h)$ defined by
	\begin{align}
		\label{eq:discrete_primal}
		I_h^{cr}(v_h)\coloneqq 
        \left\{\begin{aligned}&\tfrac{1}{2}\| \nabla_hv_h\|_{\Omega}^2+\tfrac{1}{2m}\| \pi_hv_h\|_{1,\Gamma_I}^2\\&-(f_h,\Pi_hv_h)_{\Omega}-(g_h,\pi_hv_h)_{\Gamma_N}+\smash{I_{\{u_D^h\}}^{\Gamma_D}}(\pi_hv_h)\,,
		\end{aligned}\right.
	\end{align}
    where $\smash{I_{\{u_D^h\}}^{\Gamma_D}}\colon \mathcal{L}^0(\mathcal{S}_h^{\partial})\to \mathbb{R}\cup\{+\infty\}$, for every $\widehat{v}_h\in \mathcal{L}^0(\mathcal{S}_h^{\partial})$, is defined by 
	\begin{align*}
		\smash{I_{\{u_D^h\}}^{\Gamma_D}}(\widehat{v}_h)
		&\coloneqq 
		\begin{cases}
			0&\text{ if }\widehat{v}_h=u_D^h\text{ a.e.\ on }\Gamma_D\,,\\[-0.5mm]
			+\infty &\text{ else}\,.
		\end{cases}
	\end{align*}
    Then, the effective domain of the discrete primal energy functional \eqref{eq:discrete_primal} is given via
    \begin{align*}
       \smash{K_h^{cr}\coloneqq \textup{dom}(I_h^{cr})=\widehat{u}_D^h+\mathcal{S}^{1,cr}_D(\mathcal{T}_h)\,.}
    \end{align*}
	Since the functional \eqref{eq:discrete_primal} is proper, convex, weakly coercive, and lower semi-continuous, the direct method in the calculus of variations yields the existence of a  minimizer $u_h^{cr}\in K_h^{cr}$, called  \emph{discrete primal solution}. Here,   the weak coercivity is a consequence of the discrete~Friedrich~inequality (\textit{cf}.\ Lemma \ref{lem:poin_discrete}). 
    More precisely, for every $v_h\in \mathcal{S}^{1,cr}(\mathcal{T}_h)$, one uses that
    \begin{align*}
    \smash{\| \nabla_h v_h\|_{\Omega}^2+\tfrac{1}{m}\|\pi_hv_h\|_{1,\Gamma_I}^2}
    &\ge \smash{\min\big\{1,\tfrac{\vert \Gamma_I\vert }{m}\big\}\big\{\| \nabla_h v_h\|_{\Omega}^2+\vert\langle \pi_h v_h\rangle_{\Gamma_I}\vert^2\big\}}
    \\&\ge \smash{\tfrac{1}{2(c_{\mathrm{F}}^{cr})^2}}\min\big\{1,\tfrac{\vert \Gamma_I\vert }{m}\big\}\|v_h\|_{\Omega}^2
    \,.
    \end{align*}    
    In what follows, we always employ the notation $\smash{u_h^{cr}\in K_h^{cr}}$ for discrete primal solutions.~Note~that, if $ \Gamma_D\neq \emptyset$ or $\Omega$ is connected,  
the functional \eqref{eq:discrete_primal} is strictly convex and, consequently, the discrete primal solution $\smash{u_h^{cr}\in K_h^{cr}}$ is uniquely determined.\enlargethispage{10mm}\medskip

%$\bullet$ \emph{Discrete primal variational inequality.} A discrete primal solution $u_h^{cr}\in K_h^{cr}$ equivalently is a solution of the following discrete variational inequality: 
%for every $v_h\in K_h^{cr}$, it holds that
%\begin{align}
%\begin{aligned} 
%\tfrac{1}{2m}\|\pi_hu_h^{cr}\|_{1,\Gamma_I}^2+	(\nabla_h u_h^{cr},\nabla_h u_h^{cr}-\nabla_h v_h^{cr})_{\Omega}\leq \tfrac{1}{2m}\|\pi_hv_h\|_{1,\Gamma_I}^2&+(f_h,\Pi_hu_h^{cr}-\Pi_hv_h)_{\Omega}\\&\quad+(g_h,\pi_hu_h^{cr}-\pi_hv_h)_{\Gamma_N}\,.
%\end{aligned}\label{eq:discrete_variational_ineq}
%\end{align}
	
	\hspace*{-2.5mm}$\bullet$ \emph{Discrete dual problem.} A \emph{(Fenchel) dual problem} (in the sense of \cite[Rem.\ 4.2, p.\ 60/61]{ET99}) to the minimization of \eqref{eq:discrete_primal} is given via the maximization of the \textit{discrete dual energy functional} $D_h^{rt}\colon \mathcal{R}T^0(\mathcal{T}_h)\to \mathbb{R}\cup\{-\infty\}$, for every $y_h\in \mathcal{R}T^0(\mathcal{T}_h)$~defined~by 
	\begin{align}\label{eq:discrete_dual}
		D_h^{rt}(y_h)\coloneqq
        \left\{\begin{aligned}&-\tfrac{1}{2}\| \Pi_hy_h\|_{\Omega}^2-\tfrac{m}{2}\|y_h\cdot n\|_{\infty,\Gamma_I}^2+(y_h\cdot n,u_D^h)_{\Gamma_D}\\&-
			\smash{I_{\{-f_h\}}^{\Omega}} (\textup{div}\,y_h)-\smash{I_{\{g_h\}}^{\Gamma_N}} (y_h\cdot n)\,,
		\end{aligned}\right.
	\end{align}  
    where 
	$	\smash{I_{\{-f_h\}}^{\Omega}}\colon  \smash{\mathcal{L}^0(\mathcal{T}_h)}\to \mathbb{R}\cup \{+\infty\}$, for every $\smash{\widehat{v}_h\in  \mathcal{L}^0(\mathcal{T}_h)}$, is defined by
	\begin{align*}
		\smash{I_{\{-f_h\}}^{\Omega}}(\widehat{v}_h) \coloneqq
		\begin{cases}
			0&\text{ if }\widehat{v}_h=-f_h\text{ a.e.\ in }\Omega\,,\\[-0.5mm]
			+\infty& \text{ else}\,,
		\end{cases}
	\end{align*}
	and $	\smash{I_{\{g_h\}}^{\Gamma_N}}\colon  \smash{\mathcal{L}^0(\mathcal{S}_h^{\partial})}\to \mathbb{R}\cup \{+\infty\}$, for every $\smash{\widehat{v}_h\in \mathcal{L}^0(\mathcal{S}_h^{\partial})}$, is defined by
	\begin{align*}
		\smash{I_{\{g_h\}}^{\Gamma_N}}(\widehat{v}_h)&\coloneqq \begin{cases}
			0&\text{ if }\widehat{v}_h=g_h\text{ a.e.\ on }\Gamma_N\,,\\[-0.5mm]
			+\infty &\text{ else}\,.
		\end{cases}
	\end{align*}
    Then, the effective domain of the  negative of  the discrete dual energy functional \eqref{eq:discrete_dual} is~given~via
    \begin{align*}
        K^{rt,*}_h\coloneqq \textup{dom}(-D_h^{rt})
        = \left\{y_h\in \mathcal{R}T^0(\mathcal{T}_h)\;\bigg|\;  
        \begin{aligned}
            \textup{div}\,y_h&=-f_h&&\text{ a.e.\ in }\Omega\,,\\[-0.5mm]
            y_h\cdot n&=g_h&&\text{ a.e.\ on }\Gamma_N
        \end{aligned} \right\}\,.
    \end{align*}

    The following theorem proves that
    the maximization of \eqref{eq:discrete_dual} is truly the 
    (Fenchel)~dual~\mbox{problem} (in the sense of \cite[Rem.\ 4.2, p.\ 60/61]{ET99}) to the minimization of \eqref{eq:discrete_primal}. In addition, it establishes the existence of a unique discrete dual solution as well as the validity of a discrete strong duality relation and discrete convex optimality relations.
    
 	\begin{theorem}[strong duality and convex duality relations]\label{prop:discrete_duality} The following statements apply:
		\begin{itemize}[noitemsep,topsep=2pt,leftmargin=!,labelwidth=\widthof{(ii)}]
			\item[(i)]\hypertarget{prop:discrete_duality.i}{}  The (Fenchel) dual problem to the minimization of \eqref{eq:discrete_primal} is given via the~\mbox{maximization}~of~\eqref{eq:discrete_dual};  
			\item[(ii)]\hypertarget{prop:discrete_duality.ii}{}  There exists a unique maximizer $z_h^{rt}\in  \mathcal{R}T^0(\mathcal{T}_h)$ of \eqref{eq:discrete_dual} satisfying the \textup{discrete~admissibility conditions}
			\begin{alignat}{2}
				\textup{div}\,z_h^{rt}&=-f_h&&\quad\text{ a.e.\ in }\Omega\,,\label{eq:discrete_admissibility.1}\\
                z_h^{rt}\cdot n&=g_h&&\quad\text{ a.e.\ on }\Gamma_N\,.\label{eq:discrete_admissibility.2}
			\end{alignat}  
			In addition, there
			holds a \emph{discrete strong duality relation}, \textit{i.e.}, we have that
			\begin{align}
				I_h^{cr}(u_h^{cr}) = D_h^{rt}(z_h^{rt})\,;\label{eq:discrete_strong_duality}
			\end{align}
			\item[(iii)]\hypertarget{prop:discrete_duality.iii}{} There hold the \emph{discrete convex optimality relations}, \textit{i.e.}, we have that
			\begin{alignat}{2}
				\Pi_h z_h^{rt}&=\nabla_h u_h^{cr}\quad\text{ a.e.\ in }\Omega\label{eq:discrete_optimality.1}\,,\\
				-( z_h^{rt}\cdot n,\pi_hu_h^{cr})_{\Gamma_I}&=\smash{\tfrac{m}{2}\|\smash{z_h^{rt}\cdot n}\|_{\infty,\Gamma_I}^2+\tfrac{1}{2m}\|\pi_hu_h^{cr}\|_{1,\Gamma_I}^2}\label{eq:discrete_optimality.2}\,.
			\end{alignat}
		\end{itemize}
	\end{theorem}

    \begin{remark}[equivalent condition to \eqref{eq:discrete_optimality.2}]
        Note that,~by~the~standard equality condition in the Fenchel--Young inequality (\textit{cf}.\ \cite[Prop.\ 5.1,~p.~21]{ET99}) and the chain rule for the subdifferential (\textit{cf}.\ \cite[Thm.\ 4.19]{Clason2020}), the discrete convex optimality relation \eqref{eq:discrete_optimality.2} is equivalent to%\enlargethispage{5mm}
        \begin{align}\label{eq:discrete_optimality.3}
        -z_h^{rt}\cdot n \in \tfrac{1}{m}(\partial\vert\cdot\vert)(\pi_hu_h^{cr})\|\pi_hu_h^{cr}\|_{1,\Gamma_I}\quad \text{ a.e.\ on }\Gamma_I\,.
        \end{align}
    \end{remark}
	
		\begin{proof}[Proof (of Theorem \ref{prop:discrete_duality}).]
		\emph{ad (\hyperlink{prop:discrete_duality.i}{i}).} To begin with, we need to bring the primal energy functional~\eqref{eq:discrete_primal} into the form of a primal energy functional in  the sense of Fenchel (\textit{cf}.~\cite[Rem.\ 4.2, p.\ 60/61]{ET99}),~\textit{i.e.}, 
        \begin{align*}
			I_h^{cr}(v_h)= G_h(\nabla_h v_h)+F_h(v_h)\,,
		\end{align*}
        where $G_h\colon (\mathcal{L}^0(\mathcal{T}_h))^d\to \mathbb{R}\cup\{+\infty\}$ and $F_h\colon \mathcal{S}^{1,cr}(\mathcal{T}_h)\to \mathbb{R}\cup\{+\infty\}$ should be proper, convex, and lower semi-continuous functionals.
        To this end, let us introduce 
         the functionals~${G_h\colon \hspace{-0.15em}(\mathcal{L}^0(\mathcal{T}_h))^d\hspace{-0.15em}\to \hspace{-0.15em}\mathbb{R}}$ and $F_h\colon \hspace{-0.15em}\mathcal{S}^{1,cr}(\mathcal{T}_h)\hspace{-0.15em}\to\hspace{-0.15em} \mathbb{R}\cup\{+\infty\}$, for every $\overline{y}_h\hspace{-0.15em}\in \hspace{-0.15em} (\mathcal{L}^0(\mathcal{T}_h))^d$ and $v_h\hspace{-0.15em}\in\hspace{-0.15em}\mathcal{S}^{1,cr}(\mathcal{T}_h)$,~respectively,~\mbox{defined}~by
		\begin{align*}
			G_h(\overline{y}_h)&\coloneqq \tfrac{1}{2}\|\overline{y}_h\|_{\Omega}^2\,,\\
			F_h(v_h)&\coloneqq -(f_h,\Pi_h v_h)_{\Omega}-(g_h,\pi_h v_h)_{\Gamma_N}+\tfrac{1}{2m}\|\pi_hv_h\|_{1,\Gamma_I}^2+\smash{I_{\{u_D^h\}}^{\Gamma_D}}(\pi_hv_h)\,.
		\end{align*}
		Then, according to \cite[Rem.\  4.2, p.\ 60]{ET99}, the (Fenchel) dual problem to the~\mbox{minimization}~of~\eqref{eq:discrete_primal} is given via the maximization of $D_h^0\colon (\mathcal{L}^0(\mathcal{T}_h))^d\to \mathbb{R}\cup\{-\infty\}$, 
		for every $\overline{y}_h\in (\mathcal{L}^0(\mathcal{T}_h))^d$~defined~by
		\begin{align}\label{helper_functional}
			\smash{D_h^0(\overline{y}_h)\coloneqq -G^*_h(\overline{y}_h)- F^*_h(-\nabla^*_h\overline{y}_h)\,,}
		\end{align}
		where $\nabla^*_h\colon  (\mathcal{L}^0(\mathcal{T}_h))^d\to (\mathcal{S}^{1,cr}(\mathcal{T}_h))^*$ denotes the adjoint operator to $\nabla_h \colon \mathcal{S}^{1,cr}(\mathcal{T}_h)\to (\mathcal{L}^0(\mathcal{T}_h))^d$. 
		
        $\bullet$ First, resorting to \cite[Prop.\ 4.2, p.\ 19]{ET99}, for every $\overline{y}_h\in (\mathcal{L}^0(\mathcal{T}_h))^d$, we have that
		\begin{align}\label{prop:discrete_duality.1}
			G^*_h(\overline{y}_h)=\tfrac{1}{2}\|\overline{y}_h\|_{\Omega}^2\,.
		\end{align}

        $\bullet$ Second, using a lifting lemma (\textit{cf}.\ \cite[Lem.\ A.1]{BarGudKal24}) and the discrete integration-by-parts~formula \eqref{eq:pi}, for every $\overline{y}_h\in (\mathcal{L}^0(\mathcal{T}_h))^d$,
        it turns out that  
		\begin{align} \label{prop:discrete_duality.3}
        \begin{aligned} 
				&F_h^*(-\nabla_h^*\overline{y}_h)
				=\sup_{v_h\in \mathcal{S}^{1,cr}(\mathcal{T}_h)}{\left\{
                \begin{aligned} 
                &-(\overline{y}_h,\nabla_h v_h)_{\Omega}+(f_h,\Pi_hv_h)_{\Omega}+(g_h,\pi_hv_h)_{\Gamma_N}\\[-0.5mm]&-\tfrac{1}{2m}\|\pi_hv_h\|_{1,\Gamma_I}^2
                -\smash{I_{\{u_D^h\}}^{\Gamma_D}}(\pi_hv_h)
                \end{aligned}
                \right\}}
                \\&\quad=\sup_{v_h\in \mathcal{S}^{1,cr}_D(\mathcal{T}_h)}{\left\{
                \begin{aligned} 
                &-(\overline{y}_h,\nabla_h (v_h+\widehat{u}_D^h))_{\Omega}+(f_h,\Pi_h(v_h+\widehat{u}_D^h))_{\Omega}+(g_h,\pi_h(v_h+\widehat{u}_D^h))_{\Gamma_N}\\[-0.5mm]&-\tfrac{1}{2m}\|\pi_h(v_h+\widehat{u}_D^h)\|_{1,\Gamma_I}^2 
                \end{aligned}
                \right\}}
				\\&\quad=
                \begin{cases}
                    \left.\begin{aligned} &I_{\{-f_h\}}^{\Omega}(\textup{div}\,y_h)+I_{\{g_h\}}^{\Gamma_N}(y_h\cdot n)-(y_h\cdot n,\widehat{u}_D^h)_{\Gamma_I\cup \Gamma_D}\\&+\underset{v_h\in \mathcal{S}^{1,cr}_D(\mathcal{T}_h)}{\sup}{\!\big\{-(y_h\cdot n,\pi_hv_h)_{\Gamma_I}-\tfrac{1}{2m}\|\pi_h(v_h+\widehat{u}_D^h)\|_{1,\Gamma_I}^2\big\}}
                    \end{aligned}\right\}&
                    \left\{\begin{aligned}
                        &\textup{ if }\overline{y}_h=\Pi_hy_h\\
                        &\textup{ for }y_h\in \mathcal{R}T^0(\mathcal{T}_h)\,,
                    \end{aligned}\right.\\
                    +\infty&\text{ else}\,,
                \end{cases} 
                \end{aligned}\hspace{-7.5mm}
		\end{align}
        where, due to $(\pi_h(\cdot)|_{\Gamma_I})(K_h^{cr})=\mathcal{L}^0(\mathcal{S}_h^I)$, for every $y_h\in \mathcal{R}T^0(\mathcal{T}_h)$, we have that
        \begin{align}\label{prop:discrete_duality.4}
        \begin{aligned} 
            &\sup_{v_h\in \mathcal{S}^{1,cr}_D(\mathcal{T}_h)}{\big\{-(y_h\cdot n,\pi_hv_h)_{\Gamma_I}-\tfrac{1}{2m}\|\pi_h(v_h+\widehat{u}_D^h)\|_{1,\Gamma_I}^2\big\}} 
            \\&=\sup_{v_h\in K_h^{cr}}{\big\{(y_h\cdot n,\widehat{u}_D^h-\pi_hv_h)_{\Gamma_I}-\tfrac{1}{2m}\|\pi_hv_h\|_{1,\Gamma_I}^2\big\}} 
            \\&=(y_h\cdot n,\widehat{u}_D^h)_{\Gamma_I}+\sup_{\rho\ge 0}{ \sup_{\substack{\overline{v}_h\in \mathcal{L}^0(\mathcal{S}_h^I)\\ \|\overline{v}_h\|_{1,\Gamma_I}=\rho}}{\big\{-(y_h\cdot n,\overline{v}_h)_{\Gamma_I}-\tfrac{1}{2m}\rho^2\big\}}}
            \\&=(y_h\cdot n,\widehat{u}_D^h)_{\Gamma_I}+\sup_{\rho\ge 0}{\big\{\rho\,\|y_h\cdot n\|_{\infty,\Gamma_I}-\tfrac{1}{2m}\rho^2\big\}}
            \\&=(y_h\cdot n,\widehat{u}_D^h)_{\Gamma_I}+\tfrac{m}{2}\|y_h\cdot n\|_{\infty,\Gamma_I}^2\,.
        \end{aligned}
        \end{align}
		Using \eqref{prop:discrete_duality.1} and \eqref{prop:discrete_duality.3} together with \eqref{prop:discrete_duality.4} in \eqref{helper_functional}, for every $\overline{y}_h\in (\mathcal{L}^0(\mathcal{T}_h))^d$, we arrive at
        \begin{align*}
            D_h^0(\overline{y}_h)=
            \begin{cases}
                \left.\begin{aligned}
                    &-\tfrac{1}{2}\|\Pi_h y_h\|_{\Omega}^2-\tfrac{m}{2}\|y_h\cdot n\|_{\infty,\Gamma_I}^2 +(y_h\cdot n,u_D^h)_{\Gamma_D}\\
                    &-I_{\{-f_h\}}^{\Omega}(\textup{div}\,y_h)-I_{\{g_h\}}^{\Gamma_N}(y_n\cdot n)
                \end{aligned}\right\}&
                 \left\{\begin{aligned}
                        &\textup{ if }\overline{y}_h=\Pi_hy_h\\
                        &\textup{ for }y_h\in \mathcal{R}T^0(\mathcal{T}_h)\,,
                    \end{aligned}\right.\\
                    -\infty&\text{ else}\,.
            \end{cases}
        \end{align*}
		Since $D_h^0= -\infty$ in $(\mathcal{L}^0(\mathcal{T}_h))^d\setminus \Pi_h(\mathcal{R}T^0(\mathcal{T}_h))$, we  restrict~\eqref{eq:discrete_dual}~to~$ \Pi_h(\mathcal{R}T^0(\mathcal{T}_h))$. More precisely, 
        we define $D_h^{rt}\colon \mathcal{R}T^0(\mathcal{T}_h)\to \mathbb{R}\cup\{+\infty\}$, for every $y_h\in\mathcal{R}T^0(\mathcal{T}_h)$, by $D_h^{rt}(y_h)\coloneqq D_h^0(\Pi_h y_h)$.
		
		\emph{ad (\hyperlink{prop:discrete_duality.ii}{ii}).}  Since $G_h\colon \hspace{-0.1em} (\mathcal{L}^0(\mathcal{T}_h))^d\hspace{-0.1em}\to\hspace{-0.1em} \mathbb{R}$ and $F_h\colon  \hspace{-0.1em}\mathcal{S}^{1,cr}(\mathcal{T}_h)\hspace{-0.1em}\to\hspace{-0.1em} \mathbb{R}\cup\{+\infty\}$ are proper,~convex,~and~lower semi-continuous \hspace{-0.15mm}and \hspace{-0.15mm}since \hspace{-0.15mm}$G_h\colon \hspace{-0.175em}(\mathcal{L}^0(\mathcal{T}_h))^d\hspace{-0.175em}\to\hspace{-0.175em}\mathbb{R}$ \hspace{-0.15mm}is \hspace{-0.15mm}continuous \hspace{-0.15mm}at \hspace{-0.15mm}${\nabla_h\widehat{u}_D^h\hspace{-0.175em}\in\hspace{-0.175em} \textup{dom}(G_h)}$ \hspace{-0.15mm}with \hspace{-0.15mm}${\widehat{u}_D^h\hspace{-0.175em}\in\hspace{-0.175em}  \textup{dom}(F_h)}$, %\textit{i.e.},
		%\begin{align*}
		%	 G_h(\overline{y}_h)\to G_h(0)\quad\big(\overline{y}_h\to 0\quad\text{ in }(\mathcal{L}^0(\mathcal{T}_h))^d\big)\,,
		%\end{align*}
		the Fenchel duality  theorem (\emph{cf}.\  \cite[Rem.\ 4.2, (4.21), p.\ 61]{ET99}) yields the existence of a maximizer   $\overline{z}_h^0\in  (\mathcal{L}^0(\mathcal{T}_h))^d$ of  \eqref{helper_functional} and that a discrete strong duality relation applies, \textit{i.e.},\vspace*{-0.25mm}
		\begin{align*}
			\smash{I_h^{cr}(u_h^{cr})=D_h^0(\overline{z}_h^0)\,.}
		\end{align*}
		Since $D_h^0= -\infty$ in $(\mathcal{L}^0(\mathcal{T}_h))^d\setminus \Pi_h(\mathcal{R}T^0(\mathcal{T}_h))$, there exists  $z_h^{rt}\in \mathcal{R}T^0(\mathcal{T}_h)$ satisfying the discrete admissibility conditions \eqref{eq:discrete_admissibility.1},\eqref{eq:discrete_admissibility.2} such that $\overline{z}_h^0=\Pi_h z_h^{rt}$~a.e.\ in $\Omega$.
		In particular, we have that  $D_h^0(\overline{z}_h^0)=D_h^{rt}(z_h^{rt})$, so that $z_h^{rt}\in \mathcal{R}T^0(\mathcal{T}_h)$ is a maximizer of \eqref{eq:discrete_dual} and the discrete strong duality relation \eqref{eq:discrete_strong_duality} applies.  By the~strict~\mbox{convexity}~of~${G_h^*\colon \hspace{-0.1em}(\mathcal{L}^0(\mathcal{T}_h))^d\hspace{-0.1em}\to\hspace{-0.1em}\mathbb{R}}$~and~the~\mbox{divergence}~\mbox{constraint} \eqref{eq:discrete_admissibility.1}, the maximizer $z_h^{rt}\in \mathcal{R}T^0(\mathcal{T}_h)$ is uniquely determined.
		
		\emph{ad (\hyperlink{prop:discrete_duality.iii}{iii}).} \hspace*{-0.15mm}By \hspace*{-0.15mm}the \hspace*{-0.15mm}standard \hspace*{-0.15mm}(Fenchel) \hspace*{-0.15mm}convex \hspace*{-0.15mm}duality \hspace*{-0.15mm}theory \hspace*{-0.15mm}(\emph{cf}.\ \hspace*{-0.15mm}\cite[Rem.\ \hspace*{-0.15mm}4.2, \hspace*{-0.15mm}(4.24),~\hspace*{-0.15mm}(4.25),~\hspace*{-0.15mm}p.~\hspace*{-0.15mm}61]{ET99}), there hold the convex optimality relations\vspace*{-0.25mm}\enlargethispage{5mm}
		\begin{align}
			-\nabla_h^*\Pi_h z_h^{rt}&\in \partial F_h(u_h^{cr})\,,\label{prop:discrete_duality.5}\\
			\Pi_h z_h^{rt}&\in \partial G_h(\nabla_h u_h^{cr})\,.\label{prop:discrete_duality.6}
		\end{align}
		The inclusion  \eqref{prop:discrete_duality.6} is equivalent to the discrete convex optimality relation \eqref{eq:discrete_optimality.1}. The~inclusion \eqref{prop:discrete_duality.5}, by the definition of the subdifferential and, then, using
		the discrete integration-by-parts formula \eqref{eq:pi}, is equivalent to that for every $v_h\in K_h^{cr}$, it holds that 
        \begin{align*}
			%\begin{aligned}
            \tfrac{1}{2m}\|\pi_hv_h\|_{1,\Gamma_I}^2-\tfrac{1}{2m}\|\pi_hu_h^{cr}\|_{1,\Gamma_I}^2&\ge (f_h,\Pi_hv_h-\Pi_h u_h)_{\Omega}+(g_h,\pi_hv_h-\pi_h u_h)_{\Gamma_N}\\&\quad-(\Pi_hz_h^{rt},\nabla_h v_h-\nabla_h u_h^{cr})_{\Omega}\,.
            %\\
            %&=-(z_h^{rt}\cdot n,\pi_hv_h-\pi_hu_h^{cr})_{\Gamma_I} \,.
            %\end{aligned}\label{prop:discrete_duality.7}
		\end{align*}
        By the discrete admissibility conditions \eqref{eq:discrete_admissibility.1},\eqref{eq:discrete_admissibility.2}, this is equivalent to that for every $v_h\in K_h^{cr}$, it holds that 
		\begin{align}
			\begin{aligned}\tfrac{1}{2m}\|\pi_hv_h\|_{1,\Gamma_I}^2-\tfrac{1}{2m}\|\pi_hu_h^{cr}\|_{1,\Gamma_I}^2\ge -(z_h^{rt}\cdot n,\pi_hv_h-\pi_hu_h^{cr})_{\Gamma_I} \,.
            \end{aligned}\label{prop:discrete_duality.7}
		\end{align}
        Since $(\pi_h|_{\Gamma_I})(K_h^{cr})=\mathcal{L}^0(\mathcal{S}_h^I)$, from \eqref{prop:discrete_duality.7}, we infer that
        \begin{align*}
            -z_h^{rt}\cdot n\in \partial (\tfrac{1}{2m}\|\cdot\|_{1,\Gamma_I}^2)(\pi_hu_h^{cr})\,,
        \end{align*}
        which, \hspace{-0.1mm}by \hspace{-0.1mm}the \hspace{-0.1mm}standard \hspace{-0.1mm}equality \hspace{-0.1mm}condition \hspace{-0.1mm}in \hspace{-0.1mm}the \hspace{-0.1mm}Fenchel--Young \hspace{-0.1mm}inequality \hspace{-0.1mm}(\textit{cf}.\ \hspace{-0.1mm}\cite[Prop.~\hspace{-0.1mm}5.1,~\hspace{-0.1mm}p.~\hspace{-0.1mm}21]{ET99}), is equivalent to \eqref{eq:discrete_optimality.2}.
	\end{proof}\newpage

    \section{\emph{A priori} error analysis}\label{sec:apriori}
	
	\qquad In this section, resorting to the discrete convex duality relations established in Section \ref{sec:discrete},~we derive an \emph{a priori} error identity for  the discrete primal problem \eqref{eq:discrete_primal} and the discrete~dual~problem \eqref{eq:discrete_dual} \hspace{-0.1mm}at \hspace{-0.1mm}the \hspace{-0.1mm}same \hspace{-0.1mm}time. \hspace{-0.1mm}From \hspace{-0.1mm}this \hspace{-0.1mm}\emph{a priori} \hspace{-0.1mm}error \hspace{-0.1mm}identity, \hspace{-0.1mm}in \hspace{-0.1mm}turn, \hspace{-0.1mm}we \hspace{-0.1mm}extract \hspace{-0.1mm}convergence \hspace{-0.1mm}under \hspace{-0.1mm}minimal \hspace{-0.1mm}regularity \hspace{-0.1mm}assumptions \hspace{-0.1mm}and \hspace{-0.1mm}explicit~\hspace{-0.1mm}error~\hspace{-0.1mm}decay~\hspace{-0.1mm}rates \hspace{-0.1mm}under \hspace{-0.1mm}fractional~\hspace{-0.1mm}\mbox{regularity}~\hspace{-0.1mm}assumptions. 
	To this end, we~proceed~analogously to the continuous setting (\textit{cf}.\ Section \ref{sec:aposteriori})~and~introduce~the \emph{discrete~primal-dual gap estimator} $\eta_{\textup{gap},h}^2\colon\hspace{-0.1em} K_h^{cr}\hspace{-0.1em}\times\hspace{-0.1em} K_h^{\smash{rt,*}}\hspace{-0.1em}\to \hspace{-0.1em}[0,+\infty)$, for every $v_h\in K_h^{cr}$~and~${y_h\hspace{-0.1em}\in\hspace{-0.1em} K_h^{\smash{rt,*}}}$ defined by 
	\begin{align}\label{eq:discrete_primal_dual_gap_estimator}
		\eta_{\textup{gap},h}^2(v_h,y_h)\coloneqq I_h^{cr}(v_h)-D_h^{rt}(y_h)\,.
	\end{align}
	The  discrete primal-dual gap estimator  \eqref{eq:discrete_primal_dual_gap_estimator} measures the accuracy of admissible approximations of the discrete primal problem \eqref{eq:discrete_primal} and the discrete dual problem \eqref{eq:discrete_dual}~at~the~same~time~via measuring the respective violation of the discrete strong duality relation \eqref{eq:discrete_strong_duality}. More precisely, discrete primal-dual gap estimator \eqref{eq:discrete_primal_dual_gap_estimator} 
    splits into two contributions~that~each~measure~the~violation of the discrete convex optimality relations \eqref{eq:discrete_optimality.1},\eqref{eq:discrete_optimality.2}.\enlargethispage{7.5mm}
	
	\begin{lemma}[representation of discrete primal-dual gap estimator]\label{lem:discrete_primal_dual_gap_estimator}
		For every $v_h\in K_h^{cr}$ and $y_h\in K_h^{\smash{rt,*}}$, we have that 
		\begin{align*}
			\eta_{\textup{gap},h}^2(v_h,y_h)\coloneqq  \eta_{A,\textup{gap},h}^2(v_h,y_h)+\eta_{B,\textup{gap},h}^2(v_h,y_h)\,,\hspace*{2.5cm}\\
			\text{where}\quad
            \left\{\quad\begin{aligned}
                \eta_{A,\textup{gap},h}^2(v_h,y_h)&\coloneqq 	\tfrac{1}{2} \|\nabla_h v_h-\Pi_h y_h\|_{\Omega}^2\,,\\
			\eta_{B,\textup{gap},h}^2(v_h,y_h)&\coloneqq 
            \tfrac{m}{2}\|y_h\cdot n\|_{\infty,\Gamma_I}^2+(y_h\cdot n,\pi_h v_h)_{\Gamma_I}+\tfrac{1}{2m}\|\pi_hv_h\|_{1,\Gamma_I}^2\,.
            \end{aligned}\right.
		\end{align*}
	\end{lemma}
	
	\begin{remark}[interpretation of the components of the discrete primal-dual gap estimator]\hphantom{                   }
		\begin{itemize}[noitemsep,topsep=2pt,leftmargin=!,labelwidth=\widthof{(iii)},font=\itshape]
			\item[(i)] The estimator $\eta_{A,\textup{gap},h}^2$ measures the violation of the discrete convex optimality~relation~\eqref{eq:discrete_optimality.1};
			\item[(ii)] The estimator $\eta_{B,\textup{gap},h}^2$ measures the violation of the discrete convex optimality~relation~\eqref{eq:discrete_optimality.2}. Moreover,  by the Fenchel--Young inequality (\textit{cf}.\ \cite[Prop.\ 5.1, p.\ 21]{ET99}), for every $v_h\in \mathcal{S}^{1,cr}(\mathcal{T}_h)$ and  $y_h\in \mathcal{R}T^0(\mathcal{T}_h)$, we have that
            \begin{align*}
                \tfrac{m}{2}\|y_h\cdot n\|_{\infty,\Gamma_I}^2+(y_h\cdot n, \pi_hv_h)_{\Gamma_I}+\tfrac{1}{2m}\|\pi_hv_h\|_{1,\Gamma_I}^2\ge 0\,.
            \end{align*}
		\end{itemize}
	\end{remark}
	
	\begin{proof}[Proof (of Lemma \ref{lem:discrete_primal_dual_gap_estimator})]\let\qed\relax
        For every $v_h\in K_h^{cr}$ and $y_h\in K^{rt,*}_h$, 
		using the admissibility conditions \eqref{eq:discrete_admissibility.1},\eqref{eq:discrete_admissibility.2},  $v_h=u_D^h$ a.e.\ on $\Gamma_D$, the integration-by-parts formula \eqref{eq:pi_cont}, and the~binomial~formula, we find that
		\begin{align*}
			I_h^{cr}(v_h)-D_h^{rt}(y_h)&= \tfrac{1}{2}\| \nabla_hv_h\|_{\Omega}^2-(f_h,\Pi_hv_h)_{\Omega}-(g_h,\pi_hv_h)_{\Gamma_N}+\tfrac{1}{2m}\|\pi_hv_h\|_{1,\Gamma_I}^2\\&\quad+\tfrac{1}{2}\|\Pi_hy_h\|_{\Omega}^2-(y_h\cdot n,v_h)_{\Gamma_D}+\tfrac{m}{2}\|y_h\cdot n\|_{\infty,\Gamma_I}^2\\&
			= \tfrac{1}{2}\| \nabla_hv_h\|_{\Omega}^2+(\textup{div}\,y_h,\Pi_hv_h)_{\Omega}+\tfrac{1}{2}\| \Pi_hy_h\|_{\Omega}^2\\&\quad+\tfrac{m}{2}\|y_h\cdot n\|_{\infty,\Gamma_I}^2-(y_h\cdot n,v_h)_{\Gamma_D\cup \Gamma_N}+\tfrac{1}{2m}\|\pi_hv_h\|_{1,\Gamma_I}^2
			\\&
			= \tfrac{1}{2}\| \nabla_h v_h\|_{\Omega}^2-(\Pi_hy_h,\nabla_h v_h)_{\Omega}+\tfrac{1}{2}\|\Pi_hy_h\|_{\Omega}^2\\&\quad+\tfrac{m}{2}\|y_h\cdot n\|_{\infty,\Gamma_I}^2+(y_h\cdot n, \pi_hv_h)_{\Gamma_I}+\tfrac{1}{2m}\|\pi_hv_h\|_{1,\Gamma_I}^2
			\\&
			= \tfrac{1}{2}\| \nabla_h v_h-\Pi_hy_h\|_{\Omega}^2\\&\quad+\tfrac{m}{2}\|y_h\cdot n\|_{\infty,\Gamma_I}^2+(y_h\cdot n, \pi_hv_h)_{\Gamma_I}+\tfrac{1}{2m}\|\pi_hv_h\|_{1,\Gamma_I}^2 \,.\tag*{$\qedsymbol$}
		\end{align*}
	\end{proof}

    Next, as \emph{`natural'} error quantities in the discrete primal-dual gap identity  (\textit{cf}.\ Theorem~\ref{thm:discrete_prager_synge_identity}), we employ the \emph{optimal strong convexity measures} for the discrete  primal energy functional~\eqref{eq:discrete_primal}~at~a discrete primal solution $u_h^{cr}\in  K_h^{cr}$, \textit{i.e.}, 
	$\smash{\rho_{I_h^{cr}}^2}\colon K_h^{cr}\to [0,+\infty)$, and the discrete~dual~\mbox{energy}~functional \eqref{eq:discrete_dual} at the discrete dual solution $z_h^{rt} \in K_h^{\smash{rt,*}}$, \textit{i.e.}, $\smash{\rho_{\smash{-D_h^{rt}}}^2}\colon  K_h^{\smash{rt,*}}\to [0,+\infty)$, for every $v_h\in K_h^{cr}$ and $y_h\in  K_h^{\smash{rt,*}}$, respectively, defined by
	\begin{align}\label{def:discrete_optimal_primal_error}
		\rho_{\smash{I_h^{cr}}}^2(v_h)&\coloneqq I_h^{cr}(v_h)-I_h^{cr}(u_h^{cr})\,,\\
        \rho_{\smash{-D_h^{rt}}}^2(y_h)&\coloneqq- D_h^{rt}(y_h)+D_h^{rt}(z_h^{rt})\,.
	\end{align}
	
	\begin{lemma}[discrete optimal strong convexity measures]\label{lem:discrete_strong_convexity_measures}
		The following statements apply:
		\begin{itemize}[noitemsep,topsep=2pt,leftmargin=!,labelwidth=\widthof{(ii)}]
			\item[(i)] \hypertarget{lem:discrete_strong_convexity_measures.i}{} For every $v_h\in K_h^{cr}$, we have that 
			\begin{align*}
				\hspace{-2.5mm}\smash{\rho_{I_h^{cr}}^2(v_h)=\tfrac{1}{2}\|\nabla_h v_h-\nabla_h u_h^{cr}\|_{\Omega}^2+\tfrac{m}{2}\|z_h^{rt}\cdot n\|_{\infty,\Gamma_I}^2+(z_h^{rt}\cdot n, \pi_hv_h)_{\Gamma_I}+\tfrac{1}{2m}\|\pi_hv_h\|_{1,\Gamma_I}^2\,;}
			\end{align*}
			\item[(ii)]\hypertarget{lem:discrete_strong_convexity_measures.ii}{}  For every $y_h\in K_h^{\smash{rt,*}}$, we have that 
			\begin{align*}
				\hspace{-2.5mm}\smash{\rho_{-D_h^{rt}}^2(y_h)=\tfrac{1}{2}\|\Pi_h y_h-\Pi_h z_h^{rt}\|_{\Omega}^2+\tfrac{m}{2}\|y_h\cdot n\|_{\infty,\Gamma_I}^2+(y_h\cdot n, \pi_hu_h^{cr})_{\Gamma_I}+\tfrac{1}{2m}\|\pi_hu_h^{cr}\|_{1,\Gamma_I}^2\,.}
			\end{align*}
		\end{itemize}
		\end{lemma}

		\begin{proof}\let\qed\relax
		\emph{ad (\hyperlink{lem:discrete_strong_convexity_measures.i}{i}).} For every $v_h\in K_h^{cr}$, using the discrete admissibility conditions \eqref{eq:discrete_admissibility.1},\eqref{eq:discrete_admissibility.2}, the discrete integration-by-parts~formula \eqref{eq:pi} together with $\pi_hv_h=\pi_h u_h^{cr}$ a.e.\ on $\Gamma_D$, the discrete convex optimality relations \eqref{eq:discrete_optimality.1},\eqref{eq:discrete_optimality.2}, and 
  the binomial formula,  we find that
		\begin{align*}
			I_h^{cr}(v_h)-I_h^{cr}(u_h^{cr})&=\tfrac{1}{2}\|\nabla_hv_h\|_{\Omega}^2-\tfrac{1}{2}\|\nabla_hu_h^{cr}\|_{\Omega}^2-(f_h,\Pi_hv_h-\Pi_hu_h^{cr})_{\Omega}-(g_h,\pi_hv_h-\pi_hu_h^{cr})_{\Gamma_N}\\&\quad+\tfrac{1}{2m}\|\pi_h v_h\|_{1,\Gamma_I}^2-\tfrac{1}{2m}\|\pi_hu_h^{cr}\|_{1,\Gamma_I}^2
    \\&=\tfrac{1}{2}\|\nabla_h v_h\|_{\Omega}^2-\tfrac{1}{2}\|\nabla_h u_h^{cr}\|_{\Omega}^2+(\textup{div}\,z_h^{rt},\Pi_hv_h-\Pi_hu_h^{cr})_{\Omega}\\&\quad+\tfrac{1}{2m}\|\pi_h v_h\|_{1,\Gamma_I}^2-\tfrac{1}{2m}\|\pi_hu_h^{cr}\|_{1,\Gamma_I}^2-(z_h^{rt}\cdot n,\pi_hv_h-\pi_hu_h^{cr})_{\Gamma_N}\\&
			=\tfrac{1}{2}\|\nabla_h v_h\|_{\Omega}^2-\tfrac{1}{2}\|\nabla_h u_h^{cr}\|_{\Omega}^2+(\Pi_hz_h^{rt},\nabla_h v_h-\nabla_h u_h^{cr})_{\Omega}
            \\&\quad+\tfrac{1}{2m}\|\pi_hv_h\|_{1,\Gamma_I}^2-\tfrac{1}{2m}\|\pi_hu_h^{cr}\|_{1,\Gamma_I}^2+(z_h^{rt}\cdot n,\pi_hv_h-\pi_hu_h^{cr})_{\Gamma_I}
            \\&
			=\tfrac{1}{2}\|\nabla_h v_h\|_{\Omega}^2-\tfrac{1}{2}\|\nabla_h u_h^{cr}\|_{\Omega}^2+(\nabla_h u_h^{cr},\nabla_h v_h-\nabla_h u_h^{cr})_{\Omega}
            \\&\quad+\tfrac{m}{2}\|z_h^{rt}\cdot n\|_{\infty,\Gamma_I}^2+(z_h^{rt}\cdot n,\pi_h v_h)_{\Gamma_I}+\tfrac{1}{2m}\|\pi_h v_h\|_{1,\Gamma_I}^2
			\\&
			=\tfrac{1}{2}\|\nabla_h v_h-\nabla_h u_h^{cr}\|_{\Omega}^2\\&\quad+\tfrac{m}{2}\|z_h^{rt}\cdot n\|_{\infty,\Gamma_I}^2+(z_h^{rt}\cdot n,\pi_hv_h)_{\Gamma_I}+\tfrac{1}{2m}\|\pi_hv_h\|_{1,\Gamma_I}^2\,.
		\end{align*}
		
		\emph{ad (\hyperlink{lem:discrete_strong_convexity_measures.ii}{ii}).} For every $y_h\in K^{rt,*}_h$, using that $y_h\cdot n=z_h^{rt}\cdot n$ a.e.\ on $\Gamma_N$, that $\pi_h u_h^{cr}=u_D^h$ a.e.\ on $\Gamma_D$, the discrete  integration-by-parts  formula \eqref{eq:pi} together with the discrete admissibility conditions \eqref{eq:discrete_admissibility.1},\eqref{eq:discrete_admissibility.2}, the discrete convex optimality relation \eqref{eq:discrete_optimality.2}, and the binomial formula, we find that
		\begin{align*}
			\hspace{-1mm}-D_h^{rt}(y_h)+D_h^{rt}(z_h^{rt})&=\tfrac{1}{2}\|\Pi_hy_h\|_{\Omega}^2-\tfrac{1}{2}\|\Pi_hz_h^{rt}\|_{\Omega}^2
            +(z_h^{rt}\cdot n-y_h\cdot n, u_D^h)_{\Gamma_D}
            \\&\quad+\tfrac{m}{2}\|y_h\cdot n\|_{\infty,\Gamma_I}^2-\tfrac{m}{2}\|z_h^{rt}\cdot n\|_{\infty,\Gamma_I}^2
            \\&=\tfrac{1}{2}\|\Pi_hy_h\|_{\Omega}^2-\tfrac{1}{2}\|\Pi_hz_h^{rt}\|_{\Omega}^2
            +(z_h^{rt}\cdot n-y_h\cdot n, \pi_hu_h^{cr})_{\partial\Omega}
            \\&\quad+\tfrac{m}{2}\|y_h\cdot n\|_{\infty,\Gamma_I}^2-\tfrac{m}{2}\|z_h^{rt}\cdot n\|_{\infty,\Gamma_I}^2-(z_h^{rt}\cdot n-y_h\cdot n, \pi_hu_h^{cr})_{\Gamma_I}\\&
			=\tfrac{1}{2}\|\Pi_hy_h\|_{\Omega}^2-\tfrac{1}{2}\|\Pi_hz_h^{rt}\|_{\Omega}^2-(\Pi_hz_h^{rt},\Pi_hy_h-\Pi_hz_h^{rt})_{\Omega}\\&\quad+\tfrac{m}{2}\|y_h\cdot n\|_{\infty,\Gamma_I}^2+(y_h\cdot n, \pi_hu_h^{cr})_{\Gamma_I}+\tfrac{1}{2m}\|\pi_hu_h^{cr}\|_{1,\Gamma_I}^2
   \\&
			=\tfrac{1}{2}\|\Pi_hy_h-\Pi_h z_h^{rt}\|_{\Omega}^2
            \\&\quad+\tfrac{m}{2}\|y_h\cdot n\|_{\infty,\Gamma_I}^2+(y_h\cdot n, \pi_hu_h^{cr})_{\Gamma_I}+\tfrac{1}{2m}\|\pi_hu_h^{cr}\|_{1,\Gamma_I}^2 \,.\tag*{$\qedsymbol$}
		\end{align*}
		\end{proof}

		Eventually, we establish a discrete \textit{a posteriori} error identity that identifies~the~\emph{discrete~primal-dual total error} $\rho_{\textup{tot},h}^2\colon K_h^{cr}\times K_h^{\smash{rt,*}}\to [0,+\infty)$, for every $v_h\in K_h^{cr}$ and $y_h\in K_h^{\smash{rt,*}}$ defined~by\vspace{-0.5mm} 
		\begin{align}\label{eq:discrete_primal_dual_error}
			\smash{\rho_{\textup{tot},h}^2(v_h,y_h)\coloneqq \rho_{I_h^{cr}}^2(v_h)+\rho_{-D_h^{rt}}^2(y_h)\,,}
		\end{align}
		with the discrete primal-dual gap estimator  \eqref{eq:discrete_primal_dual_gap_estimator}.\enlargethispage{5mm}
		
	\begin{theorem}[discrete primal-dual gap identity]\label{thm:discrete_prager_synge_identity}
		For every $v_h\in  K_h^{cr}$ and $y_h\in K_h^{\smash{rt,*}}$,~we~have that\vspace{-1mm} 
		\begin{align*}
			\smash{\rho_{\textup{tot},h}^2(v_h,y_h)=\eta_{\textup{gap},h}^2(v_h,y_h)\,.}
		\end{align*}
	\end{theorem}
	
	\begin{proof}
        We combine the definitions \eqref{eq:discrete_primal_dual_gap_estimator}--\eqref{eq:discrete_primal_dual_error} using the discrete strong duality \eqref{eq:discrete_strong_duality}.
	\end{proof}\newpage
    Inserting the canonical interpolants \eqref{CR-interpolant},\eqref{RT-interpolant} of a primal and the dual solution,~\mbox{respectively}, in the discrete primal-dual gap identity (\textit{cf}.\ Theorem \ref{thm:discrete_prager_synge_identity}), we arrive at an \textit{a priori}~error~\mbox{identity}, 
    which, depending on regularity assumptions, allows us to 
    extract convergence~or~error~\mbox{decay}~rates.\vspace{-0.5mm}

    \begin{theorem}[\textit{a priori} error identity, convergence,  error decay rates] \label{thm:apriori_identity}
		If $f_h\coloneqq \Pi_h f\in \mathcal{L}^0(\mathcal{T}_h)$, $g_h\coloneqq \pi_h g$ $\in \mathcal{L}^0(\mathcal{S}_h^N)$, and $u_D^h\coloneqq \pi_h u_D\in \mathcal{L}^0(\mathcal{S}_h^D)$, then~the~following~\mbox{statements}~apply:
		\begin{itemize}[noitemsep,topsep=2pt,leftmargin=!,labelwidth=\widthof{(ii)}]
			\item[(i)] \hypertarget{thm:apriori_identity.i}{} \emph{\textit{A priori} error identity and convergence:} If $z\in  \smash{(L^p(\Omega))^d}$, where $p>2$, then  $\Pi_h^{rt}z\in  \smash{K_h^{\smash{rt,*}}}$,~and
			\begin{align*}
				\rho_{\textup{tot},h}^2(\Pi_h^{cr}u,\Pi_h^{rt}z)&= 
    \tfrac{1}{2}\|\Pi_h z -\Pi_h \Pi_h^{rt} z\|_{\Omega}^2+(\pi_h(\overline{z\cdot n})-\overline{z\cdot n},u-\pi_hu)_{\Gamma_I}
    \\&\quad+
   \tfrac{m}{2} \big\{\| \pi_h(\overline{z\cdot n})\|_{\infty,\Gamma_I}^2-\|  \overline{z\cdot n}\|_{\infty,\Gamma_I}^2\big\}+\tfrac{1}{2m}\big\{\| \pi_h u \|_{1,\Gamma_I}^2-\|u \|_{1,\Gamma_I}^2\big\}\,. 
			\end{align*} 
            In particular, there holds\vspace{-0.5mm}
            \begin{align*}
                \rho_{\textup{tot},h}^2(\Pi_h^{cr}u,\Pi_h^{rt}z)\to 0\quad (h\to 0)\,;
            \end{align*} 
			\item[(ii)] \hypertarget{thm:apriori_identity.ii}{} \emph{Error decay rates I:} If $u\in H^{1+\nu}(\Omega)$ (\textup{i.e.}, $z\in (H^{\nu}(\Omega))^d$ due to \eqref{eq:optimality.1}), where $\nu \in (0,1]$,~then\vspace{-0.5mm}
			\begin{align*}
				\rho_{\textup{tot},h}^2(\Pi_h^{cr}u,\Pi_h^{rt}z)\lesssim\,\begin{cases}
                h^{\smash{\min\{2\nu,\frac{1}{2}\}}}&\text{ if }\nu\in (0,\frac{1}{2}]\,,\\
                h^{\smash{\frac{1}{2}+\nu}}&\text{ if }\nu\in (\frac{1}{2},1]\,;
				\end{cases}
			\end{align*}
            \item[(iii)] \hypertarget{thm:apriori_identity.iii}{} \emph{Error decay rates II:} If $u\in  \smash{H^{1+\nu}(\Omega)}$, where $\nu \in (0,1]$, and, in addition, $u\in  \smash{H^\alpha(\Gamma_I)} $ and $z\in (H^\beta(\Gamma_I))^d$, where $\alpha,\beta\in (0,1]$,
			then\vspace{-0.5mm}\enlargethispage{6mm}
            \begin{align*} \rho_{\textup{tot},h}^2(\Pi_h^{cr}u,\Pi_h^{rt}z)\lesssim h^{\smash{\min\{2\nu,\alpha+\beta\}}}\,. 
			\end{align*}
		\end{itemize} 
	\end{theorem}

    \begin{proof}\let\qed\relax
        \emph{ad (\hyperlink{thm:apriori_identity.i}{i}).} First, using \eqref{eq:grad_preservation},\eqref{eq:trace_preservation} and \eqref{eq:div_preservation},\eqref{eq:normal_trace_preservation}, respectively,
		we observe that
		%\begin{align*}
		%	\begin{aligned}  
		%		\nabla_h\Pi_h^{cr}u&=\Pi_h \nabla u&&\quad\text{ a.e.\ in }\Omega\,,\\
        %        \pi_h\Pi_h^{cr}u&= \pi_hu=\pi_hu_D=u_D&&\quad \text{ a.e.\ on }\Gamma_D\,,
		%	\end{aligned}
		%\end{align*}
		%\textit{i.e.}, it holds that 
        $\Pi_h^{cr}u\in K^{cr}_h$ and 
        %Second, using \eqref{eq:div_preservation} and \eqref{eq:normal_trace_preservation}, 
		%we observe that
		%\begin{align*}
		%	\begin{aligned}  
		%		\textup{div}\,\Pi_h^{rt}z&=\Pi_h \textup{div}\,z=-f_h&&\quad\text{ a.e.\ in }\Omega\,,\\
        %        (\Pi_h^{rt}z)\cdot n&= \pi_h(z\cdot n)=g_h&&\quad \text{ a.e.\ on }\Gamma_N\,,
		%	\end{aligned}
		%\end{align*}
		%\textit{i.e.}, it holds that 
        $\Pi_h^{rt}z\in K^{rt,*}_h$. Then, 
		using Theorem \ref{thm:discrete_prager_synge_identity} together with Lemma \ref{lem:discrete_primal_dual_gap_estimator}~and~Lemma~\ref{lem:discrete_strong_convexity_measures}  as well as the convex optimality relation \eqref{eq:optimality.1}, we find that 
          \begin{align}\label{thm:apriori_identity.1}
        \begin{aligned} 
            \rho_{\textup{tot},h}^2(\Pi_h^{cr}u,\Pi_h^{rt}z)&= 
    \tfrac{1}{2}\|\Pi_h z -\Pi_h \Pi_h^{rt} z\|_{\Omega}^2
   \\&\quad+
    \tfrac{m}{2}\|\pi_h(\overline{z\cdot n})\|_{\infty,\Gamma_I}^2+(\pi_h(\overline{z\cdot n}),u)_{\Gamma_I}+\tfrac{1}{2m}\|\pi_h u  \|_{1,\Gamma_I}^2 \,.
    \end{aligned}
        \end{align}
        Using in \eqref{thm:apriori_identity.1} the  convex optimality relation \eqref{eq:optimality.2} and that $ \smash{\pi_h(\overline{z\cdot n})-\overline{z\cdot n}\perp_{L^2} \pi_h u}$,~we~arrive~at
        \begin{align}\label{thm:apriori_identity.2}
            \begin{aligned} 
           \rho_{\textup{tot},h}^2(\Pi_h^{cr}u,\Pi_h^{rt}z) &= 
            \tfrac{1}{2}\|\Pi_h z -\Pi_h \Pi_h^{rt} z\|_{\Omega}^2+(\pi_h(\overline{z\cdot n})-\overline{z\cdot n},u-\pi_hu)_{\Gamma_I}
            \\&\quad+
             \tfrac{m}{2}\big\{\| \pi_h(\overline{z\cdot n}) \|_{\infty,\Gamma_I}^2-\|\overline{z\cdot n}\|_{\infty,\Gamma_I}^2\big\}+\tfrac{1}{2m}\big\{\| \pi_h u  \|_{1,\Gamma_I}^2-\|u\|_{1,\Gamma_I}^2\big\} \,.
            \end{aligned}
        \end{align}

    \textit{ad (\hyperlink{thm:apriori_identity.ii}{ii}).} Let us denote the four terms on the right-hand side of the \textit{a priori} error identity~\eqref{thm:apriori_identity.2}~by $I_i^h$, $i=1,\ldots,4$, respectively. It is left to extract the claimed error decay rates from~these~terms: 

    \textit{ad $I_1^h$.} Using the $L^2$-stability of $\Pi_h$ (with constant 1) and the fractional approximation~properties of $\Pi_h^{rt}$ (\textit{cf}.\ \cite[Thms. 16.4, 16.6]{EG21I}), we obtain $\smash{I_1^h\lesssim h^{2\nu}\,\| u\|_{1+\nu,\Omega}^2}$.
    %\begin{align*}
    %    I_1^h\lesssim c\, h^{2\nu}\,\| u\|_{1+\nu,\Omega}^2\,.
    %\end{align*}

    \textit{ad $I_3^h+I_4^h$.} Using the $L^\infty$- and $L^1$-stability of $\pi_h$ (with constant 1),  we obtain $I_3^h+I_4^h\leq 0$. 
    %\begin{align*}
    %    I_2^h+I_3^h\leq 0\,.
    %\end{align*}

    \textit{ad $I_2^h$.} %Using Hölder's inequality and the approximation properties of $\pi_h$, we obtain
    %\begin{align*}
    %   I_2^h\leq \|\pi_h(\overline{z\cdot n})-\overline{z\cdot n}\|_{2,\Gamma_I}\|u-\pi_h u\|_{2,\Gamma_I}
    %\end{align*}
    We distinguish the cases $\nu \in (0,\smash{\frac{1}{2}}]$ and $\nu \in (\smash{\frac{1}{2}},1]$: 

    $\bullet$ \textit{Case $\nu\hspace{-0.1em} \in\hspace{-0.1em} (0,\smash{\frac{1}{2}}]$.} In this case, by the standard trace theorem, we only have~that~${u|_{\partial\Omega}\hspace{-0.1em}\in\hspace{-0.1em} H^{\frac{1}{2}}(\partial\Omega)}$, so that, using Hölder's inequality, the $L^\infty$-stability  of $\pi_h$ (with constant 1), and the fractional approximation properties of $\pi_h$ (\textit{cf}.\ \cite[Rem.\ 18.17]{EG21I}), we find that
    \begin{align*}
        I_2^h%&\leq \|\pi_h(\overline{z\cdot n})-\overline{z\cdot n}\|_{\infty,\Gamma_I}\|u-\pi_h u\|_{1,\Gamma_I}
        %\\&
        \lesssim \smash{2\|\overline{z\cdot n}\|_{\infty,\Gamma_I}\vert u\vert_{\smash{\frac{1}{2}},\Gamma_I}h^{\smash{\frac{1}{2}}}}\,.
    \end{align*}

    $\bullet$ \textit{Case $\nu \hspace{-0.05em}\in\hspace{-0.05em} (\frac{1}{2},1]$.} In this case, due to $u\hspace{-0.05em}\in\hspace{-0.05em}  \smash{H^{\smash{\frac{3}{2}}}(\Omega)}$ and $\Delta u\hspace{-0.05em}=\hspace{-0.05em}\textup{div}\,z\hspace{-0.05em}\in\hspace{-0.05em}  \smash{L^2(\Omega)}$, by \cite[Cor. 3.7]{BehrndtGesztesyMitrea22},  we have that  $u|_{\partial\Omega}\in H^1(\partial\Omega)$. Moreover, due to $1+\nu >\frac{3}{2}$, by the standard trace theorem and the convex optimality relation \eqref{eq:optimality.1}, we have that
    $z|_{\partial\Omega}\in (H^{\smash{\nu-\frac{1}{2}}}(\partial\Omega))^d$. As a result, using Hölder's inequality and the fractional approximation properties of $\pi_h$ (\textit{cf}.\ \cite[Rem.\ 18.17]{EG21I}), we find that\vspace{-0.5mm} 
    \begin{align*}
        I_2^h%&\leq \|\pi_h(\overline{z\cdot n})-\overline{z\cdot n}\|_{2,\Gamma_I}\|u-\pi_h u\|_{2,\Gamma_I}
        %\\&
        \lesssim \smash{c\,h^{\smash{\nu -\frac{1}{2}}}\vert z\vert_{\smash{\nu -\frac{1}{2}},\Gamma_I}\vert \Gamma_I\vert^{\smash{\frac{1}{2}}}\vert u\vert_{1,\Gamma_I}h}\,.
    \end{align*}

    \textit{ad (\hyperlink{thm:apriori_identity.iii}{iii}).} We proceed as in the proof of (\hyperlink{thm:apriori_identity.ii}{ii}), except for the
    term $\smash{I_2^h}$. For the latter,
    using Hölder's inequality and the fractional approximation properties of $\pi_h$ (\textit{cf}.\ \cite[Rem.\ 18.17]{EG21I}),~we~obtain\vspace{-0.5mm} 
    \begin{align*}
        I_2^h%&\leq \|\pi_h(\overline{z\cdot n})-\overline{z\cdot n}\|_{2,\Gamma_I}\|u-\pi_h u\|_{2,\Gamma_I}
        %\\&
        \lesssim \smash{h^{\beta}\vert z\vert_{\beta,\Gamma_I} h^{\alpha}\vert u\vert_{\alpha,\Gamma_I}}\,.\tag*{$\qedsymbol$}
    \end{align*}
    \end{proof}\newpage

    %\begin{corollary}
    %If $f_h\coloneqq \Pi_h f\in \mathcal{L}^0(\mathcal{T}_h)$ and $u\in H^{1+\nu}(\Omega)$, where $\nu \in (0,1]$, then
    %\begin{align*}
    %    \|\nabla_h u_h^{cr}- \nabla u\|_{\Omega}^2+\|z_h^{rt}- z\|_{\Omega}^2\leq c\,h^{2\nu}\| u\|_{1+\nu,\Omega}\,.
    %\end{align*}
    %\end{corollary}

    \section{A semi-smooth Newton scheme}\vspace{-0.5mm}
    \label{s:implementation}    

    \hspace{5mm}The main challenge in the numerical approximation of the discrete primal problem \eqref{eq:discrete_primal} arises from its both non-local and non-smooth character. Since the degrees of freedom associated with the standard basis $(\psi_S)_{S\in \mathcal{S}_h}$ of $\mathcal{R}T^0(\mathcal{T}_h)$~are~given~via normal traces on mesh sides~(\textit{cf}.~\eqref{def:RT}), for every $y_h\in \mathcal{R}T^0(\mathcal{T}_h)$, we can construct an exact algebraic representation of $y_h \cdot n \in \mathcal{L}^0(\mathcal{S}_h)$. This together with the formula 
    \begin{align}\label{eq:augmentation_trick}
        -\tfrac{m}{2}\|y_h \cdot n\|_{\infty,\Gamma_I}^2=\sup_{\mu_h\in \mathbb{R}}{\big\{-\tfrac{m}{2}\mu_h^2-I_+^{\Gamma_I}(\mu_h-\vert y_h \cdot n\vert)\big\}}\,,
    \end{align}
    valid for all $y_h\in \mathcal{R}T^0(\mathcal{T}_h)$,
     where $\smash{I_+^{\Gamma_I}}\colon \mathcal{L}^0(\mathcal{S}_h^{\partial})\to \mathbb{R}\cup\{+\infty\}$, for every $\widehat{v}_h\in \mathcal{L}^0(\mathcal{S}_h^{\partial})$,~is~defined~by
    \begin{align*}
        I_+^{\Gamma_I}(\widehat{v}_h)\coloneqq \begin{cases}
            0&\text{ if }\widehat{v}_h\ge 0\text{ a.e.\ on }\Gamma_I\,,\\
            +\infty&\text{ else}\,,\\
        \end{cases}
    \end{align*}
    allows to convert the discrete dual problem \eqref{eq:discrete_dual} into an augmented problem that can be treated using a primal-dual active set strategy interpreted as a semi-smooth Newton~scheme~(\mbox{similar}~to~\cite{HintermullerItoKunisch2002}).\vspace{-0.5mm}

    \subsection{A reformulation of the discrete problem}
    \label{ss:discrete_reformulation}\vspace{-0.5mm}

   \hspace{5mm}Using formula \eqref{eq:augmentation_trick}, we reformulate the discrete dual problem \eqref{eq:discrete_dual} as an augmented problem. To \hspace{-0.1mm}this \hspace{-0.1mm}end, \hspace{-0.1mm}introduce \hspace{-0.1mm}the \hspace{-0.1mm}\textit{augmented \hspace{-0.1mm}discrete \hspace{-0.1mm}dual \hspace{-0.1mm}energy \hspace{-0.1mm}functional} \hspace{-0.1mm}${\Phi_h^{rt}\colon\hspace{-0.175em} \mathcal{R}T^0(\mathcal{T}_h)\hspace{-0.2em}\times \hspace{-0.2em}\mathbb{R}\hspace{-0.175em}\to\hspace{-0.175em} \mathbb{R}\hspace{-0.175em}\cup\hspace{-0.175em}\{-\infty\}}$, for every $(y_h,\mu_h)^\top\in \mathcal{R}T^0(\mathcal{T}_h)\times \mathbb{R}$ defined by
   \begin{align}\label{def:augmented_functional}
        \Phi_h^{rt}(y_h,\mu_h) \coloneqq \left\{\begin{aligned}&-\tfrac{1}{2}\| \Pi_h y_h\|_{\Omega}^2-\tfrac{m}{2}\mu_h^2
        %-I_{+}(\mu_h)
        - I_{+}^{\Gamma_I}(\mu_h - |y_h \cdot n|)
        \\& +(y_h\cdot n,u_D^h)_{\Gamma_D}-
			\smash{I_{\{-f_h\}}^{\Omega}} (\textup{div}\,y_h)-\smash{I_{\{g_h\}}^{\Gamma_N}} (y_h\cdot n)\, .
            \end{aligned}\right. 
    \end{align}
    Then, by definition \eqref{def:augmented_functional}, for every $y_h\in \mathcal{R}T^0(\mathcal{T}_h)$, we have that $D_h^{rt}(y_h)=\sup_{\mu_h \in \mathbb{R}}{\{\Phi_{h}^{rt}(y_h,\mu_h)\}}$.
Since the augmented discrete dual energy functional \eqref{def:augmented_functional} is proper, strictly convex, lower semi-continuous, \hspace{-0.1mm}the \hspace{-0.1mm}direct \hspace{-0.1mm}method \hspace{-0.1mm}in \hspace{-0.1mm}the \hspace{-0.1mm}calculus \hspace{-0.1mm}of \hspace{-0.1mm}variations \hspace{-0.1mm}yields \hspace{-0.1mm}the \hspace{-0.1mm}existence~\hspace{-0.1mm}of~\hspace{-0.1mm}a~\hspace{-0.1mm}unique~\hspace{-0.1mm}\mbox{minimizer} $(z_h^{rt},\mu_h)^\top \in \mathcal{R}T^0(\mathcal{T}_h) \times  \mathbb{R}$, where  the first entry in actual fact is the unique discrete dual solution.

The associated KKT system seeks $(z_h^{rt},\overline{u}_h,\mu_h,\lambda_h^+,\lambda_h^-)^\top \in \mathcal{R}T^0(\mathcal{T}_h) \times \mathcal{L}^0(\mathcal{T}_h) \times \mathbb{R} \times (\mathcal{L}^0(\mathcal{S}_h^I))^2$  with $z_h^{rt}\cdot n=g_h$ a.e.\ in $\Gamma_N$
such that for every $(y_h,\overline{v}_h,\eta_h)^\top \in \mathcal{R}T_N^0(\mathcal{T}_h) \times \mathcal{L}^0(\mathcal{T}_h) \times \mathbb{R}$,~there~holds\vspace{-4mm}
\begin{subequations}
\begin{alignat}{2} \label{eq:SS_KKT_a}
(\Pi_h z_h^{rt}, \Pi_h y_h)_\Omega + (\overline{u}_h, \text{div}\,y_h)_\Omega + ( \lambda_h^+ - \lambda_h^-, y_h \cdot n )_{\Gamma_I} &= (u_D^h, y_h \cdot n)_{\Gamma_D}\,,\\ \label{eq:SS_KKT_b}
(\text{div}\, z_h^{rt},\overline{v}_h)_\Omega &= -(f_h,\overline{v}_h)_\Omega\,, \\
m\mu_h\eta_h 
%+ \lambda_h^0)
 +(\lambda_h^+ + \lambda_h^-,\eta_h )_{\Gamma_I} &= 0\,, \label{eq:SS_KKT_c} \\
\mu_h \pm z_h^{rt} \cdot n & \ge 0\quad\text{ a.e.\ in }\Gamma_I\,, \label{eq:SS_KKT_d} \\
%\mu_h - z_h^{rt} \cdot n & \ge 0&&\quad\text{ a.e.\ in }\Gamma_I\,, \label{eq:SS_KKT_e}\\
%\mu_h &\ge 0&&\quad\text{ a.e.\ in }\Gamma_I\,, \label{eq:SS_KKT_f} \\
\lambda_h^{\pm}(\mu_h \pm z_h^{rt} \cdot n) &= 0\quad\text{ a.e.\ in }\Gamma_I\,, \label{eq:SS_KKT_g} \\
%\lambda_h^-(\mu_h - z_h^{rt} \cdot n) &= 0&&\quad\text{ a.e.\ in }\Gamma_I\,, \label{eq:SS_KKT_h}\\
% \lambda_h^0 \mu_h &= 0\quad\text{ a.e.\ in }\Gamma_I\,, \label{eq:SS_KKT_i} \\
\lambda_h^+,\lambda_h^-
%,\lambda_h^0 
& \le 0\quad\text{ a.e.\ in }\Gamma_I\,. \label{eq:SS_KKT_j}
\end{alignat}
\end{subequations}
The strict convexity of the augmented discrete dual energy functional \eqref{def:augmented_functional} guarantees that the KKT conditions \eqref{eq:SS_KKT_a}--\eqref{eq:SS_KKT_j} are not only necessary, but also sufficient optimality conditions.\vspace{-0.5mm}

\subsection{An inverse generalized Marini formula}\vspace{-0.5mm}

\hspace{5mm}Incorporating the additional information provided by the Lagrange multipliers in the
%The second entry of $(z_h^{rt},\overline{u}_h,\mu_h,\lambda_h^+,\lambda_h^-,\lambda_h^0)^\top \in \mathcal{R}T^0(\mathcal{T}_h) \times \mathcal{L}^0(\mathcal{T}_h) \times \mathbb{R} \times \mathcal{L}^0(\mathcal{S}_h^I) \times\mathcal{L}^0(\mathcal{S}_h^I) \times \mathbb{R}$ satisfying the 
KKT conditions \eqref{eq:SS_KKT_a}--\eqref{eq:SS_KKT_j} allows to reconstruct a discrete primal solution from the~discrete~dual~\mbox{solution}.\vspace{-0.5mm}% via an inverse generalized Marini formula:\vspace{-0.5mm}

\begin{lemma}[inverse generalized Marini formula]\label{lem:marini}
Let $(z_h^{rt},\overline{u}_h,\mu_h,\lambda_h^+,\lambda_h^-)^\top \hspace{-0.175em}\in\hspace{-0.175em} \mathcal{R}T^0(\mathcal{T}_h) \times \mathcal{L}^0(\mathcal{T}_h)$  $\times \mathbb{R} \times (\mathcal{L}^0(\mathcal{S}_h^I))^2$ be such that the KKT conditions \eqref{eq:SS_KKT_a}--\eqref{eq:SS_KKT_j}~are~satisfied.~Then,
%we have that $z_h^{rt}=z_h^{rt}$ and 
a discrete primal solution $u_h^{cr} \in \mathcal{S}^{1,cr}(\mathcal{T}_h)$ is available via the \emph{inverse \mbox{generalized}~Marini~\mbox{formula}}\vspace{-0.5mm}
\begin{align*} %\label{eq:inverse_marini}
    u_h^{cr} = \overline{u}_h + \Pi_h z_h^{rt} \cdot (\textup{id}_{\mathbb{R}^d} - \Pi_h \textup{id}_{\mathbb{R}^d})\in \mathcal{L}^1(\mathcal{T}_h)\,.
    \end{align*}
\end{lemma}
\begin{proof}
Let us introduce the function $\widehat{u}_h\coloneqq \overline{u}_h + \smash{\Pi_h z_h^{rt}} \cdot (\textup{id}_{\mathbb{R}^d} - \Pi_h \textup{id}_{\mathbb{R}^d})\in \mathcal{L}^1(\mathcal{T}_h)$,
%\begin{align}\label{eq:marini.1}
%   \widehat{u}_h\coloneqq \overline{u}_h + \Pi_h z_h^{rt} \cdot (\textup{id}_{\mathbb{R}^d} - \Pi_h \textup{id}_{\mathbb{R}^d})\in \mathcal{L}^1(\mathcal{T}_h)\,,
%\end{align}
which~satisfies\vspace{-4mm}
\begin{subequations}\label{eq:marini.2}
\begin{alignat}{2}
    \nabla_h \widehat{u}_h&= \Pi_hz_h^{rt}&&\quad \text{ a.e.\ in }\Omega\,,\label{eq:marini.2.1}\\
    \Pi_h \widehat{u}_h&=\overline{u}_h&&\quad \text{ a.e.\ in }\Omega\,.\label{eq:marini.2.2}
\end{alignat}
\end{subequations}
We establish that $\widehat{u}_h\in \mathcal{S}^{1,cr}(\mathcal{T}_h)$ is a discrete primal solution:\enlargethispage{6.5mm}

\emph{1. Step: ($\widehat{u}_h\hspace{-0.175em}\in\hspace{-0.175em} \mathcal{S}^{1,cr}(\mathcal{T}_h)$).} \hspace{-0.15mm}To \hspace{-0.15mm}begin \hspace{-0.15mm}with, \hspace{-0.15mm}due \hspace{-0.15mm}to \hspace{-0.15mm}\eqref{eq:marini.2.1} \hspace{-0.15mm}and \hspace{-0.15mm}\eqref{eq:discrete_optimality.1}, \hspace{-0.15mm}we \hspace{-0.15mm}have~\hspace{-0.15mm}that~\hspace{-0.15mm}${\widehat{u}_h-u_h^{cr}\hspace{-0.175em}\in \hspace{-0.175em}\mathcal{L}^0(\mathcal{T}_h)}$. Then, 
%Therefore, 
from the discrete integration-by-parts formula \eqref{eq:pi} and \eqref{eq:SS_KKT_a}, for every   $y_h\in \mathcal{R}T^0_0(\mathcal{T}_h)\coloneqq\{\widehat{y}_h\in \mathcal{R}T^0(\mathcal{T}_h)\mid \widehat{y}_h\cdot n= 0\text{ a.e.\ on }\partial\Omega\}$, it follows that
    \begin{align*}
        (\widehat{u}_h-u_h^{cr},\textup{div}\,y_h)_{\Omega}%&=(\Pi_h\widehat{u}_h-\Pi_hu_h^{cr},\textup{div}\,y_h)_{\Omega}
        %\\
        &
        =(\Pi_h\widehat{u}_h,\textup{div}\,y_h)_{\Omega}+(\nabla_hu_h^{cr},\Pi_h y_h)_{\Omega}\\&
        =(\overline{u}_h,\textup{div}\,y_h)_{\Omega}+(\Pi_h z_h^{rt},\Pi_h y_h)_{\Omega} 
        %\\&
        =0\,,
     \end{align*}
    \textit{i.e.}, $\widehat{u}_h-u_h^{cr}\hspace{-0.15em}\perp_{L^2}\hspace{-0.15em}\textup{div}\,(\mathcal{R}T^0_0(\mathcal{T}_h))\hspace{-0.15em}=\hspace{-0.15em}\mathcal{L}^0(\mathcal{T}_h)/\mathbb{R}$, which yields $\widehat{u}_h-u_h^{cr}\hspace{-0.15em}=\hspace{-0.15em}\textup{const}$~and,~thus,~${\widehat{u}_h\hspace{-0.15em}\in\hspace{-0.15em} \mathcal{S}^{1,cr}(\mathcal{T}_h)}$.

\emph{2. Step: ($I_h^{cr}(\widehat{u}_h)=I_h^{cr}(u_h^{cr})$).}
%Since already $\widehat{u}_h\in \mathcal{S}^{1,cr}(\mathcal{T}_h)$ with \eqref{eq:marini.2.2}, 
%It suffices to show that $\widehat{u}_h\in \mathcal{S}^{1,cr}(\mathcal{T}_h)$ satisfies the discrete convex optimality relation \eqref{eq:discrete_optimality.2}, in order to that $\widehat{u}_h\in \mathcal{S}^{1,cr}(\mathcal{T}_h)$ is a discrete primal solution.
In light of $\widehat{u}_h\in \mathcal{S}^{1,cr}(\mathcal{T}_h)$, we can use
the discrete integration-by-parts formula \eqref{eq:pi} together with
\eqref{eq:marini.2.1},\eqref{eq:marini.2.2}  in \eqref{eq:SS_KKT_a}, which yields that
\begin{subequations}\label{eq:marini.3}
\begin{alignat}{2}
    \pi_h \widehat{u}_h&= u_D^h&&\quad \text{ a.e.\ on }\Gamma_D\,,\label{eq:marini.3.1}\\[-0.5mm]
    \pi_h \widehat{u}_h&=\lambda_h^--\lambda_h^+&&\quad \text{ a.e.\ on }\Gamma_I\,.\label{eq:marini.3.2}
\end{alignat}
\end{subequations}
Then, using \eqref{eq:marini.3.2} together with $\smash{(\textrm{tr}(\cdot)\cdot n|_{\Gamma_I})(\mathcal{R}T^0(\mathcal{T}_h))=\mathcal{L}^0(\mathcal{S}_h^{I})}$, we find that
\begin{align}\label{eq:marini.4}
    \| \pi_h \widehat{u}_h \|_{1,\Gamma_I}
    %&= \sup_{\substack{y_h \in \mathcal{R}T_N^0(\mathcal{T}_h) \\ \|y_h \cdot n\|_{\infty, \Gamma_I} = 1}}
    %{(\pi_h u_h^{cr} , y_h \cdot n )_{\Gamma_I}}  \\
    %&
    = \sup_{\substack{y_h \in \mathcal{R}T_N^0(\mathcal{T}_h) \\ \|y_h \cdot n\|_{\infty, \Gamma_I} = 1}}{\big\{( \lambda_h^--\lambda_h^+ , y_h \cdot n )_{\Gamma_I}\big\}}\,.
\end{align}
Moreover, from \eqref{eq:SS_KKT_g}, it follows that
    \begin{align} \label{eq:marini.5} 
       \smash{ (\lambda_h^+ + \lambda_h^-) \mu_h = (\lambda_h^- - \lambda_h^+) z_h^{rt} \cdot n\quad\text{ a.e.\ in }\Gamma_I}\,.
    \end{align}

Next, we differentiate three cases depending on whether the constraints are active or inactive:

\emph{$\bullet$ Case 1:} If $\mu_h + z_h^{rt} \cdot n =\mu_h - z_h^{rt} \cdot n = 0$ on $S$, from \eqref{eq:marini.5}, it follows that $\lambda_h^+=\lambda_h^-=0$ and, thus, $ (\lambda_h^-- \lambda_h^+, y_h \cdot n )_{S} = 0 =  -(\lambda_h^+ + \lambda_h^-, 1)_S$;
%\begin{align*}
%    (\lambda_h^-- \lambda_h^+, z_h \cdot n )_{S} = 0 =  -(\lambda_h^+ + \lambda_h^-, 1)_S\,. 
%\end{align*}

\emph{$\bullet$ Case 2:} If $\mu_h + z_h^{rt} \cdot n,\mu_h - z_h^{rt} \cdot n > 0$ on $S$, from \eqref{eq:SS_KKT_g},\eqref{eq:SS_KKT_j}, it follows that $\lambda_h^+ = \lambda_h^- = 0$ and, thus, $(\lambda_h^-- \lambda_h^+, y_h \cdot n )_{S} = 0 =  -(\lambda_h^+ + \lambda_h^-, 1)_S$;
%\begin{align*}
%    (\lambda_h^-- \lambda_h^+, y_h \cdot n )_{S} = 0 =  -(\lambda_h^+ + \lambda_h^-, 1)_S\,. 
%\end{align*}

\emph{$\bullet$ Case 3:} If $\mu_h \pm z_h^{rt} \cdot n = 0$ on $S$ and $\mu_h \mp z_h^{rt} \cdot n > 0$ on $S$, from \eqref{eq:marini.5}, it follows that   $\lambda_h^{\mp} = 0$ and, thus, $(\lambda_h^-- \lambda_h^+, y_h \cdot n )_{S}  = -(\lambda_h^+ + \lambda_h^-, y_h \cdot n )_S$.
%\begin{align*}
%    (\lambda_h^-- \lambda_h^+, y_h \cdot n )_{S}  = -(\lambda_h^+ + \lambda_h^-, y_h \cdot n )_S\,.
%\end{align*} 

%\emph{$\bullet$ Case 3:} If $\mu_h + z_h^{rt} \cdot n > 0$ on $S$ and $\mu_h - z_h^{rt} \cdot n = 0$ on $S$, from  \eqref{eq:marini.5}, it follows that  $\lambda_h^+ = 0$ and, thus, $(\lambda_h^-- \lambda_h^+, y_h \cdot n )_{S}  = (\lambda_h^+ + \lambda_h^-, y_h \cdot n )_S$.
%\begin{align*}
%    (\lambda_h^-- \lambda_h^+, y_h \cdot n )_{S}  = (\lambda_h^+ + \lambda_h^-, y_h \cdot n )_S\,.
%\end{align*} 

In summary, the supremum in \eqref{eq:marini.4} is attained by any $y_h\in \mathcal{R}T_N^0(\mathcal{T}_h)$~with
\begin{align}\label{eq:marini.6} 
    y_h\cdot n|_S=\begin{cases} 
        \pm 1&\text{ if }\mu_h \pm z_h^{rt} \cdot n = 0\text{ and }\mu_h \mp z_h^{rt} \cdot n > 0\text{ on }S\,,\\[-0.5mm]
       % -1&\text{ if }\mu_h + y_h \cdot n > 0\text{ and }\mu_h - y_h \cdot n = 0\text{ on }S\,,\\
        0&\text{ else}\,
    \end{cases}\quad \text{ for all }S\in \mathcal{S}_h^{I}\,.
\end{align}
Therefore, for some $y_h\in \mathcal{R}T_N^0(\mathcal{T}_h)$ with \eqref{eq:marini.6}, also using \eqref{eq:SS_KKT_j}, we find that
\begin{align}\label{eq:marini.7} 
\smash{\| \pi_h \widehat{u}_h \|_{1,\Gamma_I} 
    =  - ( \lambda_h^+ + \lambda_h^-, 1 )_{\Gamma_I}
    =  \|\lambda_h^+ + \lambda_h^-\|_{1,\Gamma_I}\,,}
\end{align}
%In light of \todo{missing equation number} \eqref{eq:SS_KKT_i}, 
If we test \eqref{eq:SS_KKT_c} with $\eta_h \hspace{-0.1em}=\hspace{-0.1em} 1$, using \eqref{eq:marini.7}, we obtain $\mu_h \hspace{-0.15em}=\hspace{-0.15em} -\tfrac{1}{m} ( \lambda_h^+ + \lambda_h^-, 1 )_{\Gamma_I}\hspace{-0.15em}=\hspace{-0.15em}
        \tfrac{1}{m}\|\lambda_h^++\lambda_h^-\|_{1,\Gamma_I}$~and, thus, $\mu_h^2\hspace{-0.15em} =\hspace{-0.15em} -\tfrac{1}{m} ( \lambda_h^+ + \lambda_h^-, \mu_h)_{\Gamma_I}\hspace{-0.15em}=\hspace{-0.15em}
        \smash{\tfrac{1}{m^2}}\|\lambda_h^++\lambda_h^-\|_{1,\Gamma_I}^2$,
   % \begin{align}\label{eq:marini.8} 
   %     \mu_h = -\tfrac{1}{m} ( \lambda_h^+ + \lambda_h^-, 1 )_{\Gamma_I}=
   %     \tfrac{1}{m}\|\lambda_h^++\lambda_h^-\|_{1,\Gamma_I}\,,
   % \end{align}
    which~together~with~\eqref{eq:marini.5}~and~\eqref{eq:marini.3.2}~\mbox{implies}~that 
    \begin{align}\label{eq:marini.9} 
       \smash{-m\mu_h^2=( \lambda_h^+ + \lambda_h^-, \mu_h )_{\Gamma_I}=(\lambda_h^--\lambda_h^+,z_h^{rt} \cdot n)_{\Gamma_I}
        =(\pi_h\widehat{u}_h,z_h^{rt} \cdot n)_{\Gamma_I}}\,.
    \end{align}
Moreover, we have that\vspace{-0.5mm}
\begin{align}\label{eq:marini.10} 
    \smash{\mu_h=\|z_h^{rt} \cdot n \|_{\infty,\Gamma_I}}\,,
\end{align}
%and, thus,
%\begin{align*}
%    \tfrac{1}{m}\| \pi_h u_h^{cr} \|_{1,\Gamma_I}^2=m\|z_h^{rt}\cdot n\|_{\infty,\Gamma_I}^2\,.
%\end{align*}
since, from \eqref{eq:SS_KKT_d}, we infer that $\|z_h^{rt} \cdot n \|_{\infty,\Gamma_I} \hspace{-0.1em}\le\hspace{-0.1em} \mu_h$ and, if $\|z_h^{rt} \cdot n \|_{\infty,\Gamma_I}\hspace{-0.1em} <\hspace{-0.1em} \mu_h$, then %we have that
${\|z_h^{rt} \cdot n \|_{\infty,\Gamma_I} \hspace{-0.1em}< \hspace{-0.1em}\mu_h'}$ for some $\mu_h'> 0$, so that $\Phi_h^{rt}(z_h^{rt},\mu_h') > \Phi_h^{rt}(z_h^{rt},\mu_h)$, contradicting the maximality of $(z_h^{rt},\mu_h)^\top$. % $\in \mathcal{R}T^0(\mathcal{T}_h)\times \mathbb{R}$. 
Eventually, combining \eqref{eq:marini.7}--\eqref{eq:marini.10}, we conclude that
\begin{align*}
         \smash{-( z_h^{rt}\cdot n,\pi_h\widehat{u}_h)_{\Gamma_I}=\smash{\tfrac{m}{2}\|\smash{z_h^{rt}\cdot n}\|_{\infty,\Gamma_I}^2+\tfrac{1}{2m}\|\pi_h\widehat{u}_h\|_{1,\Gamma_I}^2}\,,}
\end{align*}
which together with \eqref{eq:marini.2.1} and  \eqref{eq:discrete_strong_duality} implies that %the discrete strong duality relation \eqref{eq:discrete_strong_duality} implies that 
$I_h^{cr}(\widehat{u}_h)=D_h^{rt}(z_h^{rt})=I_h^{cr}(u_h^{cr})$. %, \textit{i.e.}, $\widehat{u}_h\in \mathcal{S}^{1,cr}(\mathcal{T}_h)$ is a discrete primal solution.
\end{proof}

\subsection{A semi-smooth Newton method}\label{ss:SS_Newton}\vspace{-1mm}\enlargethispage{12.5mm}
\hspace{5mm}We approximate 
the  KKT conditions \eqref{eq:SS_KKT_a}--\eqref{eq:SS_KKT_j} by means of a primal-dual active set strategy interpreted as a semi-smooth Newton method (\textit{cf}.\ %\cite[Subsec. 5.3.1]{Bartels15} or 
\cite{HintermullerItoKunisch2002}), which we briefly outline~here.~To~this~end, we shift the KKT conditions \eqref{eq:SS_KKT_a}--\eqref{eq:SS_KKT_j} by $z_h^{g}\in \mathcal{R}T_N^0(\mathcal{T}_h)$ with $z_h^{g}\cdot n =g_h$~a.e.~$\Gamma_N$ and $z_h^{g}\cdot n =0$ a.e.\ $\partial\Omega\setminus\Gamma_N$ and seek $z_h^0\coloneqq z_h^{rt}-z_h^{g}\in \mathcal{R}T^0(\mathcal{T}_h)$.%\todo{AK: thanks for your corrections Keegan. I think this page is fine now :)}

%\todo{AK: This page needs to be cleaned up. In particular, we should check whether all matrices defined are also used.}
%\todo{KK: I believe we can remove any reference to matrices with $X = N$, $X = \partial$}
We define $N_h^{rt} \coloneqq \text{card}(\mathcal{S}_h \setminus \mathcal{S}_h^N)$, $N_h^{rt,0} \coloneqq \text{dim} \,(\Pi_h \mathcal{R}T_N^0(\mathcal{T}_h))$, %$N_h^{cr} \coloneqq \text{dim} \,( \mathcal{S}^{1,cr}(\mathcal{T}_h))$, 
$N_h^0 \hspace{-0.1em}\coloneqq\hspace{-0.1em} \text{dim}\, (\mathcal{L}^0(\mathcal{T}_h))$, and $N_h^{\mathrm{X}} \hspace{-0.15em}\coloneqq \hspace{-0.15em} \text{card} (\mathcal{S}_h^{\mathrm{X}})$, $\mathrm{X}\hspace{-0.15em}\in\hspace{-0.15em} \{I,D\}$,
%$\mathrm{X}\hspace{-0.15em}\in\hspace{-0.15em} \{\partial,I,D,N\}$ (\textit{i.e.}, ${N_h^\partial \hspace{-0.15em}=\hspace{-0.15em} N_h^I\hspace{-0.15em} + \hspace{-0.15em}N_h^D\hspace{-0.15em} + \hspace{-0.15em}N_h^N}$). 
%For notational~\mbox{convenience},~
introduce the index sets $\mathcal{I}_h^I \hspace{-0.12em}\coloneqq \hspace{-0.12em}\{1, \dots, N_h^I \}$ and ${\mathcal{I}_h^D \hspace{-0.12em}\coloneqq\hspace{-0.12em} N_h^I\hspace{-0.12em}+\hspace{-0.12em}\{1, \dots, N_h^D \}}$, and
%, and ${\mathcal{I}_h^N\hspace{-0.12em} \coloneqq \hspace{-0.12em} N_h^I\hspace{-0.12em}+\hspace{-0.12em} N_h^D\hspace{-0.12em}+\hspace{-0.12em}\{1, \dots, N_N \}}$. %, and $\mathcal{I}_\partial \coloneqq \mathcal{I}_I \cup \mathcal{I}_D \cup \mathcal{I}_N$. 
 fix orderings of the mesh elements $\{T_i\}_{i=1,\dots,\smash{N_h^0}}$ and mesh sides 
%$\{S_i\}_{i=1,\dots,\smash{N_h^\partial}}$ 
$\{S_i\}_{i=1,\dots,\smash{N_h^I + N_h^D}}$~such~that %$\mathcal{S}_h \cong \mathcal{S}_h^I \times \mathcal{S}_h^D \times \mathcal{S}_h^N \times \mathcal{S}_h^i$, and 
    \begin{align*}
     %   \text{span}\,(\{ \chi_{T_i} \, | \, i = 1,\dots,N_0 \}) &= \mathcal{L}^0(\mathcal{T}_h)\,, \\
     %   \text{span}(\{ \varphi_{S_i}\, | \, i = 1,\dots,N_{cr} \}) &= \mathcal{S}^{1,cr}(\mathcal{T}_h)\,, \\
     %   \text{span}\,(\{ \chi_{S_i} \, | \, i = 1,\dots, N_0 \}) &= \mathcal{L}^0(\mathcal{S}_h)\,, \\
        \text{span}\,(\{ \chi_{S_i} \, | \, i \in \mathcal{I}_{\mathrm{X}} \}) &= \mathcal{L}^0(\mathcal{S}_h^{\mathrm{X}})\quad\text{ for }\mathrm{X}\in 
        %\{I,D,N\}\,.
        \{I,D\}\,.
    \end{align*}
    %
    %Next, 
    %for  $\mathrm{X}\in \{\partial,I,D,N\}$
    For  $\mathrm{X}\hspace{-0.15em}\in \hspace{-0.15em}\{I,D\}$,~we~introduce the matrix representation of the normal trace operator 
    $\smash{\mathrm{T}_h^X\hspace{-0.15em}\in\hspace{-0.15em}  \mathbb{R}^{N_h^{\mathrm{X}}\times  N_{rt}}}$,  for every
    $i\hspace{-0.15em}\in\hspace{-0.15em}  \mathcal{I}_h^{\mathrm{X}}$, $j\hspace{-0.15em} \in\hspace{-0.15em} \{1, \dots, N_h^{rt}\}$ defined by  
    ${(\mathrm{T}_h^{\mathrm{X}})_{i,j} \hspace{-0.15em}\coloneqq\hspace{-0.15em} \smash{\frac{1}{|S_i|(d-1)}} \delta_{i,j}}$\footnote{The inclusion of the factor $(d-1)$ is basis-dependent and required for our implementation in NGSolve~(\textit{cf}.~\cite{ngsolve}).}.~For~${\mathcal{A}_h\hspace{-0.15em}\subseteq\hspace{-0.15em} \smash{\mathcal{I}_h^{I}}}$, we~introduce~the indicator matrix $\indicator_{\smash{\mathcal{A}_h}}\in \mathbb{R}^{N_h^{I}\times N_h^{I}}$, for every $i,j\in \{1,\ldots,N_h^I\}$~defined~by  ${(\indicator_{\smash{\mathcal{A}_h}})_{i,j} \coloneqq 1}$~if~$i=j \in \mathcal{A}_h$ and $(\indicator_{\smash{\mathcal{A}_h}})_{i,j} \coloneqq 0$ else.
    %We denote $I_\partial \coloneqq \text{diag}(I_I,I_D,I_N)$, $\mathrm{M}_\partial \coloneqq \text{diag}(\mathrm{M}_I,\mathrm{M}_D,\mathrm{M}_N)$, and $\mathrm{T}_\partial \coloneqq \text{diag}(\mathrm{T}_h^I,\mathrm{T}_D,\mathrm{T}_N)$. 
    Then, if we introduce the matrix representations of the~bilinear~forms 
    \begin{align*}
 \mathrm{A}_h &\coloneqq ((\Pi_h\psi_{S_i}, \Pi_h \psi_{S_j})_\Omega)_{i,j = 1,\dots, N_h^{rt,0}}\in \mathbb{R}^{\smash{N_h^{rt,0} \times N_h^{rt,0}}}\,, \\[-0.25mm]
   \mathrm{B}_h &\coloneqq ((\nabla \cdot \psi_{S_i}, \chi_{T_j})_{\Omega})_{i=1,\dots,N_h^{rt}, j = 1,\dots,N_h^{0}}\in \mathbb{R}^{\smash{N_h^{rt} \times N_h^0}}\,,\\
   \mathrm{M}_h^I &\coloneqq ((\chi_{S_i}, \chi_{S_j})_{\Gamma_I})_{i,j = 1,\dots,N_h^{I}}\in \mathbb{R}^{\smash{N_h^{I} \times N_h^{I}}}\,, \\[-0.25mm]
    \widetilde{\mathrm{M}}_h^I &\coloneqq ((1, \chi_{S_j})_{\Gamma_I})_{j = 1,\dots,N_h^{I}}\in \mathbb{R}^{\smash{1 \times N_h^{I}}} 
\end{align*}
%\todo{KK: $\widetilde{\mathrm{M}}_h^I$ (which was missing before, my bad) comes from the ``matrix" representation of the terms $\langle \lambda_h^\pm, \eta_h \rangle_{\Gamma_I}$. Technically a vector because $\eta_h \in \mathbb{R}$, but seems appropriate to leave here with rest of matrices. }
as well as the vector representations of the data
\begin{align*}
\mathrm{F}_h^{g} &\coloneqq ((f_h+\textup{div}\,z_h^g,\chi_{T_i})_\Omega)_{i=1,\dots,N_h^0} \in \mathbb{R}^{\smash{N_h^0}}\,, \\
\mathrm{Z}_h^{g} &\coloneqq ((\Pi_h z_h^{g},\Pi_h\chi_{S_i})_\Omega)_{i=1,\dots,N_h^{rt,0}} \in \mathbb{R}^{\smash{N_h^{rt,0}}}\,,\\[-0.25mm]
\mathrm{U}_D^h &\coloneqq ((\smash{u_D^h},\psi_{S_i} \cdot n)_{\Gamma_D})_{i=1,\dots,N_h^D} \in \mathbb{R}^{\smash{N_h^D}}\,,
\end{align*}
the shifted KKT conditions \eqref{eq:SS_KKT_a}--\eqref{eq:SS_KKT_j} in algebraic form equivalently seek $(\mathrm{Z}_h,\overline{\mathrm{U}}_h, \mu_h, \Lambda_h^+, \Lambda_h^-)^\top$  $\in \mathbb{R}^{\smash{N_h^{rt,0}}} \times \mathbb{R}^{\smash{N_h^0}} \times \mathbb{R} \times (\mathbb{R}^{\smash{N_h^I}})^2$ such that 
\begin{subequations}\label{eq:SS_KKT_algebraic}
\begin{align} \label{eq:SS_KKT_algebraic_a}
\mathrm{A}_h\mathrm{Z}_h + \mathrm{B}_h^{\top}\overline{\mathrm{U}}_h + (\mathrm{T}_h^I)^{\top}\mathrm{M}_h^I(\Lambda_h^{+} - \Lambda_h^{-}) &= (\mathrm{U}_D^h)^{\top} \mathrm{T}_h^D-\mathrm{A}\mathrm{Z}_h^g\,, \\
\mathrm{B}_h\mathrm{Z}_h &= -\mathrm{F}_h^{g}\,,\label{eq:SS_KKT_algebraic_b} \\
m\mu_h 
%+ \Lambda_{0} 
%+ \smash{\indicator_{N_h^I}^{\top}}\mathrm{M}_h^I
+ \smash{\widetilde{\mathrm{M}}_h^I}(\Lambda_h^{+} + \Lambda_h^{-}) &= 0\,,\label{eq:SS_KKT_algebraic_c} \\
\mu_h\indicator_{\smash{N_h^I}} \pm \mathrm{T}_h^I \mathrm{Z}_h & \ge 0\,. \label{eq:SS_KKT_algebraic_d}
\end{align} 
\end{subequations}
We approximate the shifted KKT conditions \eqref{eq:SS_KKT_a}--\eqref{eq:SS_KKT_j} in algebraic form \eqref{eq:SS_KKT_algebraic_a}--\eqref{eq:SS_KKT_algebraic_a}~using the following primal-dual active set strategy interpreted as a semi-smooth~Newton~scheme~(\textit{cf}.~\cite{HintermullerItoKunisch2002}):\vspace{-0.5mm}

\begin{algorithm}[Semi-smooth Newton method]\label{alg:semismooth_Newton}
Choose parameters $\alpha,\varepsilon_{\mathrm{STOP}} > 0$. Moreover, let $(\mathrm{Z}^{0}_h, \smash{\overline{\mathrm{U}}}_h^{0},  \mu_h^{0},  (\Lambda_h^+)^{0}, (\Lambda_h^-)^{0})^\top\hspace{-0.1em}\in\hspace{-0.1em} \mathbb{R}^{\smash{N_h^{rt,0}}} \times \mathbb{R}^{\smash{N_h^0}} \times \mathbb{R} \times (\mathbb{R}^{\smash{N_h^I}})^2$ and set $k\hspace{-0.1em}=\hspace{-0.1em}0$. Then, for~every~${k \hspace{-0.1em}\in\hspace{-0.1em} \mathbb{N}_0}$:
\begin{enumerate}[noitemsep,topsep=2pt,leftmargin=!,labelwidth=\widthof{(iii)}]
\item[(i)]\hypertarget{step.i}{} Define the most recent active sets
\begin{align*}
  \smash{\mathscr{A}_h^{\pm,k} \coloneqq\cbr{ i \in \cbr{1, \dots, N_h^{I}} \mid  ((\Lambda_h^\pm)^k + \alpha (\mu_h^k \mathrm{e}_i \pm \mathrm{T}_h^I \mathrm{Z}_h^k))\cdot \mathrm{e}_i < 0 }}\,;
\end{align*}
\item[(ii)]\hypertarget{step.ii}{} Abbreviating $\mathrm{T}_{\smash{\mathscr{A}_h^{\smash{\pm}}}}^k\coloneqq \indicator_{\smash{\mathscr{A}_h^{\pm,k}}} \mathrm{T}_h^I,\mathrm{T}_{\smash{(\mathscr{A}_h^{\smash{\pm}})^c}}^k\coloneqq \indicator_{\smash{(\mathscr{A}_h^{\pm,k})^c}} \mathrm{T}_h^I\in \mathbb{R}^{N_h^I\times N_{rt}}$, 
compute the next iterate $(\mathrm{Z}_h^{k+1}, \smash{\overline{\mathrm{U}}}_h^{k+1},  \mu_h^{k+1},  (\Lambda_h^+)^{k+1}, (\Lambda_h^-)^{k+1})^\top \in \mathbb{R}^{N_{rt}} \times \mathbb{R}^{N_0} \times \mathbb{R} \times (\mathbb{R}^{N_h^I})^2$ such that 
%\todo{AK: I replaced $(\mathrm{M}_h^I)^\top$ by $\mathrm{M}_h^I$ as it should be symmetric. Or am I wrong?}
%\todo{KK: Good catch - yes $M_h^I$ is just the mass matrix on $\mathcal{L}_h^0(\mathcal{S}^I_h)$}
\begin{align*}
\hspace{-1mm}\sbr{\begin{matrix}
\mathrm{A}_h & \mathrm{B}_h^\top & 0 &
(\mathrm{T}_h^I)^\top \mathrm{M}_h^I & -(\mathrm{T}_h^I)^\top \mathrm{M}_h^{I} \\
\mathrm{B}_h & 0 & 0 & 0 & 0 \\
0 & 0 & m & \widetilde{\mathrm{M}}_h^I & \widetilde{\mathrm{M}}_h^I \\
-\alpha \mathrm{T}_{\mathscr{A}_h^{\smash{+}}}^{k} & 0 & -\alpha \indicator_{\smash{\mathscr{A}_h^{+,k}}} & \indicator_{\smash{(\mathscr{A}_h^{+,k})^c}} & 0
\\
\alpha \mathrm{T}_{\mathscr{A}_h^{\smash{-}}}^k & 0 & -\alpha \indicator_{\smash{\mathscr{A}_h^{-,k}}} & 0 & \indicator_{\smash{(\mathscr{A}_h^{-,k})^c}}
\end{matrix}}\hspace{-1mm}
\sbr{\begin{matrix}
    \mathrm{Z}_h^{k+1} \\  \smash{\overline{\mathrm{U}}}_h^{k+1} \\  \mu_h^{k+1} \\  (\Lambda_h^+)^{k+1} \\ (\Lambda_h^-)^{k+1}
\end{matrix}}
=
\sbr{\begin{matrix}
   (\mathrm{U}_D^h)^{\top} \mathrm{T}_h^D -\mathrm{A}_h\mathrm{Z}_h^g \\ -\mathrm{F}_h^g \\ 0 \\ 0\\ 0
\end{matrix}}\,;
\end{align*} 
%\todo{KK: I've removed the $\delta$'s from the solution vector, since this is the linear system for the updated variables, i.e. $\mathrm{Z}_h^{k+1} = \mathrm{Z}_h^{k} + \delta\mathrm{Z}_h^{k} $, not just the increment}
\item[(iii)]\hypertarget{step.iii}{}  Stop if $\smash{|\mathrm{Z}_k^{k+1} - \mathrm{Z}_k^{k}|} \le \varepsilon_{\mathrm{STOP}}$; otherwise, set increase $k \to k+1$ and return to step (\hyperlink{step.i}{i}).
\end{enumerate}
\end{algorithm}\newpage
    \section{Numerical experiments}\label{sec:experiments}
    
    \hspace{5mm}In this section, we conduct a series of numerical experiments to review theoretical findings of the previous sections.  
    All numerical experiments were performed in the open source finite element library \texttt{NETGEN}/\texttt{NGSolve} (version v6.2.2406, \textit{cf}.\ \cite{netgen}/\cite{ngsolve}). All graphics were generated either using the \texttt{Matplotlib} library (version 3.5.1, \textit{cf}.\ \cite{Hunter07}) or   
    the~\texttt{ParaView}~\mbox{engine}~(\mbox{version}~\mbox{5.12.0-RC2},~\textit{cf}.~\cite{ParaView}).\enlargethispage{7.5mm}

\subsection{Numerical experiment concerning the \emph{a priori} error analysis}
    \label{ss:experiment_1}
\hspace{5mm}In this experiment, we consider a smooth manufactured solution to test the rates of convergence. For simplicity, we set $\Gamma_I = \partial \Omega$ in this experiment.
For $r>0$, set $\Omega_r\coloneqq B_r^2(0) \coloneqq \{x \in \mathbb{R}^2 \, | \, |x| < r\}$, and
consider the annular region $\Omega = \Omega_1 \setminus \Omega_{\smash{\frac12}}$. 
Moreover, we set $f\coloneqq -\smash{\frac{1}{|\cdot|^2}}\in C^\infty(\overline{\Omega})$, so that a primal solution and the dual solution, respectively, are given via
\begin{subequations}\label{expl1:solutions} 
\begin{align}
u &\coloneqq C_1 + C_2 \ln|\cdot| + \tfrac{1}{2} (\ln|\cdot|)^2\in C^\infty(\overline{\Omega})\,, \label{expl1:solutions.1} \\
z &\coloneqq \tfrac{1}{|\cdot|}(C_2 + \tfrac{1}{2} \ln|\cdot|^2) \mathrm{id}_{\smash{\mathbb{R}^2}}\in (C^\infty(\overline{\Omega}))^2\,, \label{expl1:solutions.2} 
\end{align}
\end{subequations}
where $C_1 = \frac{\ln 2 \ln 8}{54} - \frac{\ln 64}{27\pi}$ and $C_2 = \tfrac{2\ln 2}{3}$, so that 
\begin{align}\label{expl1:true_energy} 
    I(u) = -\tfrac{(\ln 2)^2 (2 + \pi \ln 2)}{9}  \approx -0.2230149\,.
\end{align} 
%
\if0
A short calculation reveals that the optimal choices of $C_1$ and $C_2$ minimizing functional \eqref{eq:primal} are
%
\[
C_1 = \tfrac{\ln 2 \ln 8}{54} - \tfrac{\ln 64}{27\pi}, \quad C_2 = \tfrac{2\ln 2}{3}, \quad
I(u) = -\tfrac{(\ln 2)^2 (2 + \pi \ln 2)}{9}  \approx -0.2230149.
\]
\fi
%

We generate a series  of triangulations $\mathcal{T}_{h_k}$, $k=0,\dots,5$, with $h_k \approx\frac{1}{2} h_{k-1}$~for~all~$k=1,\dots,5$ and $\Omega_{h_k} \coloneqq \text{int}(\cup \mathcal{T}_{h_k})\subseteq \Omega$ for all $k=0,\dots,5$. For this series of triangulations $\mathcal{T}_{h_k}$,~$k=0,\dots,5$, we compute the discrete dual solution $z_{h_k}^{rt}\in \mathcal{R}T^0(\mathcal{T}_{h_k})$
%solve the discrete dual problem \eqref{eq:discrete_dual} 
using the primal-dual active set strategy interpreted as a semi-smooth Newton scheme (\textit{cf}.\ Algorithm \ref{alg:semismooth_Newton}) and, subsequently, %and initial iterates $y_0 = 0$, ...
%Subsequently, we compute 
a discrete primal solution $u_{h_k}^{cr}\in \mathcal{S}^{1,cr}(\mathcal{T}_{h_k})$ using the inverse generalized~Marini~formula~(\textit{cf}.~Lemma~\ref{lem:marini}).

Due to the regularity of the primal solution \eqref{expl1:solutions.1} and the dual solution \eqref{expl1:solutions.2},~\mbox{Theorem}~\ref{thm:apriori_identity}(\hyperlink{thm:apriori_identity.iii}{iii}) suggests an error decay of order $\mathcal{O}(h_k^2)= \mathcal{O}(N_k^{-1})$, where $N_k\coloneqq\texttt{ndof}(\mathcal{R}T^0(\mathcal{T}_{h_k}))+\texttt{ndof}(\mathcal{L}^0(\mathcal{T}_{h_k}))$, ${k\in \mathbb{N}}$, for the discrete primal-dual total errors (\textit{cf}.\ \eqref{eq:discrete_primal_dual_error}). 
    In Figure \ref{a_priori_rates}(\textit{left}), we~report~the~expected  optimal error decay of order 
	$\mathcal{O}(h_k^2)=  \mathcal{O}(N_k^{-1})$, $k=1,\ldots, 5$, and that the \textit{a priori} error identity in Theorem~\ref{thm:apriori_identity}(\hyperlink{thm:apriori_identity.i}{i}) is  satisfied. In Figure \ref{a_priori_rates}(\textit{right}),
    we observe that the primal energies of the node-averaged discrete primal solutions $I(\overline{u}_{h_k})$, $k=0,\ldots,5$, where $\overline{u}_{h_k}^{cr}\coloneqq \Pi_{h_k}^{av} u_{h_k}^{cr}\in \mathcal{S}^{1,cr}(\mathcal{T}_{h_k})\cap H^1(\Omega)$ and $\Pi_{h_k}^{av}\colon \mathcal{S}^{1,cr}(\mathcal{T}_{h_k})\to \mathcal{S}^{1,cr}(\mathcal{T}_{h_k})\cap H^1(\Omega)$~is~the~node-averaging interpolation operator (\textit{cf}.\ \cite[Sec.\ 22.2]{EG21I}), and the dual energies of the discrete dual solutions $D(z_{h_k}^{rt})$, $k=0,\ldots,5$,~converge~to the true primal/dual energy functional value \eqref{expl1:true_energy}.

\if0
To estimate the rates of convergence, we compute the experimental order of convergence~(EOC), \textit{i.e.},
\[
\mathrm{EOC}_i (e_i) \coloneqq \frac{\log(e_i) - \log(e_{i-1})}{\log(h_i) - \log(h_{i-1})}\,, \quad i=1,\dots,5\,,
\]
where, for every $i=1,\dots,5$, we denote by $e_i$ a generic error quantity.
\fi

\begin{figure}[H]
\centering 
\includegraphics[width=\linewidth]{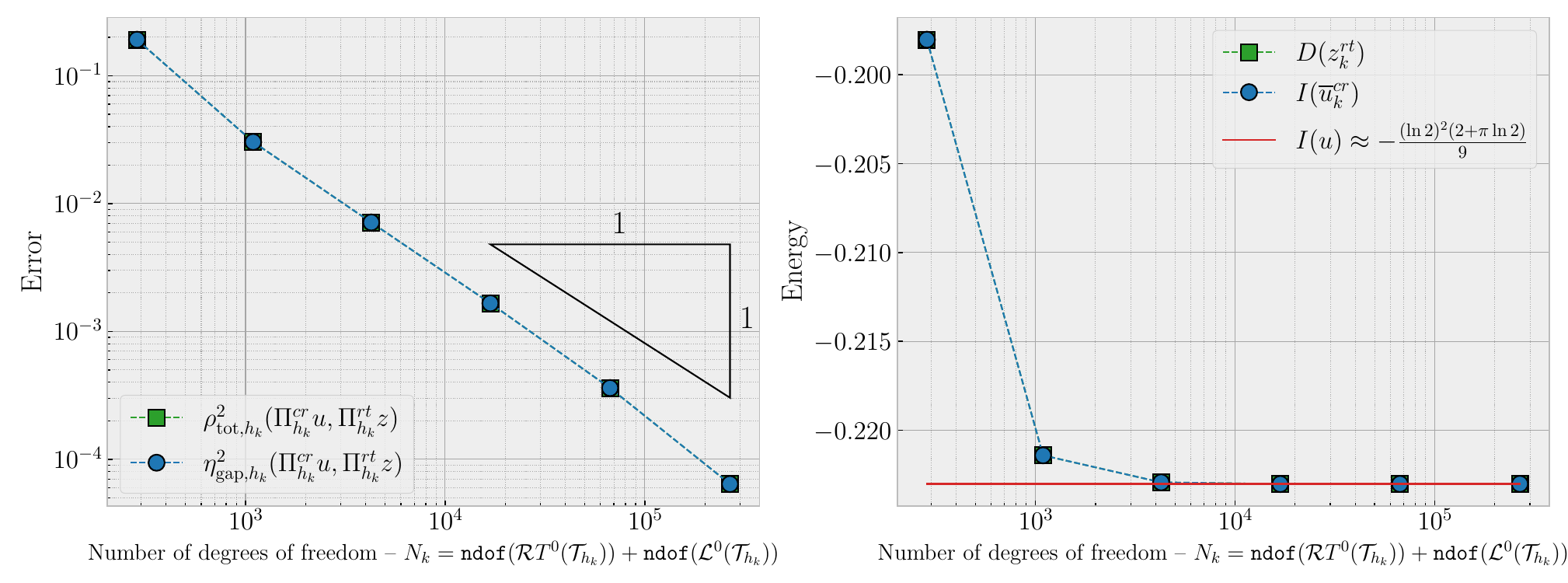}
\caption{\textit{left:} logarithmic plots of $\rho_{\textup{tot},h_k}^2(\Pi_{h_k}^{cr}u,\Pi_{h_k}^{rt}z)=\eta_{\textup{gap},h_k}^2(\Pi_{h_k}^{cr}u,\Pi_{h_k}^{rt}z)$, $k=0,\ldots,5$. We report the expected  optimal error decay of order 
	$\mathcal{O}(h_k^2)=  \mathcal{O}(N_k)$, $k=1,\ldots, 5$; \textit{right:} logarithmic plots of $I(\overline{u}_{h_k}^{cr})$, $k=0,\ldots,5$, and $D(z_{h_k}^{rt})$, $k=0,\ldots,5$, where $\overline{u}_{h_k}^{cr}\coloneqq \Pi_{h_k}^{av} u_{h_k}^{cr}\in \mathcal{S}^{1,cr}(\mathcal{T}_{h_k})\cap H^1(\Omega)$, $k=0,\ldots,5$. We report convergence to the true primal/dual energy functional value \eqref{expl1:true_energy}.}
\label{a_priori_rates}
\end{figure}

\newpage
\subsection{Numerical experiment concerning a posteriori error analysis} \vspace{-0.5mm}

\hspace{5mm}In this experiment, we review the theoretical findings of \Cref{sec:aposteriori}.\vspace{-1mm}\enlargethispage{5.5mm}

\subsubsection{Adaptive algorithm}\label{subsec:adaptive_algorithm}\vspace{-0.5mm}

\hspace{5mm}Even though the problem is non-local, in this subsection, we  propose an adaptive algorithm. 
It is based on the local mesh-refinement indicators 
$\eta_{\mathrm{gap},A,T}^2,\eta_{\mathrm{gap},B,S}^2  \colon K\hspace{-0.1em}\times\hspace{-0.1em} K^*\hspace{-0.1em}\to\hspace{-0.1em} \mathbb{R}_{\ge 0}$,~$T \hspace{-0.1em}\in\hspace{-0.1em} \mathcal{T}_h$,~$S\hspace{-0.1em}\in \hspace{-0.1em}\mathcal{S}_h^{I}$, 
for every $\smash{(v,y)^\top\in K\times K^*}$ defined by
\begin{subequations}\label{subsec:adaptive_algorithm.1}
        \begin{align}\label{subsec:adaptive_algorithm.1.1}
      \eta_{\mathrm{gap},A,T}^2(v,y) &\coloneqq \|\nabla v - y\|_T^2 \\
      \eta_{\mathrm{gap},B,S}^2(v,y) &\coloneqq  \tfrac{m}{2} \|\overline{y \cdot n}\|_{\infty,S}^2 + (\overline{y \cdot n}, v )_{S} +\tfrac{1}{2m} \|v\|_{1,S}^2\,.\label{subsec:adaptive_algorithm.1.2}
    \end{align} 
    \end{subequations}
        Then, for every $(v,y)^\top\in K\times K^*$, we have that
        \begin{subequations}\label{subsec:adaptive_algorithm.2}
        \begin{align}\label{subsec:adaptive_algorithm.2.1}
            \eta_{\mathrm{gap},A}^2(v,y)&=\sum_{T\in \mathcal{T}_h}{\eta_{\mathrm{gap},A,T}^2(v,y)}\,,\\
            \eta_{\mathrm{gap},B}^2(v,y)&\ge \sum_{S\in \mathcal{S}_h^{I}}{\eta_{\mathrm{gap},B,S}^2(v,y)}\,,\label{subsec:adaptive_algorithm.2.2}
        \end{align} 
    \end{subequations}
        where we used the embedding $\ell^1(\mathbb{N})\hookrightarrow \ell^2(\mathbb{N})$ with embedding constant $1$ in \eqref{subsec:adaptive_algorithm.2.2}.~Since~even~for element-wise \hspace{-0.175mm}affine \hspace{-0.175mm}functions, \hspace{-0.175mm}it \hspace{-0.175mm}is \hspace{-0.175mm}non-trivial \hspace{-0.175mm}to \hspace{-0.175mm}evaluate \hspace{-0.175mm}the \hspace{-0.175mm}local \hspace{-0.175mm}refinement \hspace{-0.175mm}indicator~\hspace{-0.175mm}\eqref{subsec:adaptive_algorithm.1.2}~\hspace{-0.175mm}\mbox{exactly}, we introduce the local mesh-refinement indicators 
    $\widehat{\eta}_{\mathrm{gap},B,S}^2  \colon K\cap \mathcal{L}^1(\mathcal{T}_h)\times K^*\cap \mathcal{R}T^0(\mathcal{T}_h)\to \mathbb{R}_{\ge 0}$, $S\in \mathcal{S}_h^{I}$, 
    for every $(v_h,y_h)^\top\in  K\cap \mathcal{L}^1(\mathcal{T}_h)\times K^*\cap \mathcal{R}T^0(\mathcal{T}_h)$ defined by 
    \begin{align}\label{subsec:adaptive_algorithm.3}
\widetilde{\eta}_{\mathrm{gap},B,S}^2(v_h,y_h) &\coloneqq  \tfrac{m}{2} \vert y_h \cdot n|_{S} \vert^2 + y_h \cdot n|_S \vert S\vert\langle v_h\rangle_{S} +\tfrac{1}{2m}\vert S\vert^2\vert \langle v_h\rangle_S\vert^2\,,
    \end{align} 
    which can be evaluated exactly and satisfy $\widetilde{\eta}_{\mathrm{gap},B,S}^2(v_h,y_h)\leq \eta_{\mathrm{gap},B,S}^2(v_h,y_h)$. %for every $(v_h,y_h)^\top\in  K\cap \mathcal{L}^1(\mathcal{T}_h)\times K^*\cap \mathcal{R}T^0(\mathcal{T}_h)$,~$S\in \mathcal{S}_h^{I}$~satisfy 
    %\begin{align*}
    %    \widetilde{\eta}_{\mathrm{gap},B,S}^2(v_h,y_h)\leq \eta_{\mathrm{gap},B,S}^2(v_h,y_h)\,.
    %\end{align*}
    Eventually, on the basis of %the local refinement indicators 
    \eqref{subsec:adaptive_algorithm.3}, we introduce the global estimator $\widetilde{\eta}_{\mathrm{gap},B}^2  \colon K\cap \mathcal{L}^1(\mathcal{T}_h)\times K^*\cap \mathcal{R}T^0(\mathcal{T}_h)\to \mathbb{R}_{\ge 0}$, 
    for every $(v_h,y_h)^\top\in  K\cap \mathcal{L}^1(\mathcal{T}_h)\times K^*\cap \mathcal{R}T^0(\mathcal{T}_h)$ defined by  
    \begin{align}\label{subsec:adaptive_algorithm.4}
        \widetilde{\eta}_{\mathrm{gap},B}^2(v_h,y_h)\coloneqq \sum_{S\in \mathcal{S}_h^{I}}{\widetilde{\eta}_{\mathrm{gap},B,S}^2(v,y)}\,.
    \end{align}

     The numerical experiments are based on the following \emph{adaptive algorithm}:
	 
	 \begin{algorithm}[AFEM]\label{alg:afem}
	 	Let $\varepsilon_{\textup{STOP}}>0$, $\theta_{\mathcal{T}},\theta_{\mathcal{S}}\in (0,1)$, and  $\mathcal{T}_0$ an initial  triangulation of $\Omega$. Then, for every $k\in \mathbb{N}_0$:
	 	\begin{description}[noitemsep,topsep=0.0pt,labelwidth=\widthof{\textit{('Estimate')}},leftmargin=!,font=\normalfont\itshape]
	 		\item[(`Solve')]\hypertarget{Solve}{}
            Compute $z_{h_k}^{rt}\in  K_{h_k}^{rt,*}$ using Algorithm \ref{alg:semismooth_Newton} and, then, $u_{h_k}^{cr}\in  K_{h_k}^{cr}$ using Lemma~\ref{lem:marini}. 
	 		 Post-process $u_{h_k}^{cr}\in  K_{h_k}^{cr}$ and $z_{h_k}^{rt}\in   K_{h_k}^{rt,*}$  to obtain admissible ${\overline{u}_{h_k}^{cr}\in  K}$and $\overline{z}_{h_k}^{rt}\in  K^*$;
	 		
	 		\item[(`Estimate')]\hypertarget{Estimate}{}  Compute \hspace{-0.15mm}$\smash{\{\eta^2_{\textup{gap},A,T}(\overline{u}_{h_k}^{cr},\overline{z}_{h_k}^{rt})\}_{T\in \mathcal{T}_{h_k}}}$ \hspace{-0.15mm}and \hspace{-0.15mm}${\{\widetilde{\eta}^2_{\textup{gap},B,S}(\overline{u}_{h_k}^{cr},\overline{z}_{h_k}^{rt})\}_{\smash{S\in \mathcal{S}_{h_k}^{I}}}}\!\!\!$.\ \hspace{-0.15mm}If \hspace{-0.15mm}$
	 		\smash{\eta^2_{\textup{gap},A}}(\overline{u}_{h_k}^{cr},\overline{z}_{h_k}^{rt})$ $+\smash{\widetilde{\eta}^2_{\textup{gap},B}}(\overline{u}_{h_k}^{cr},\overline{z}_{h_k}^{rt})\leq \varepsilon_{\textup{STOP}}$, then \textup{STOP}; otherwise, continue with step (\hyperlink{Mark}{`Mark'});
	 		\item[(`Mark')]\hypertarget{Mark}{} Choose minimal (in terms of cardinality) subsets $\tau_{h_k}\subseteq\mathcal{T}_{h_k}$ and $\sigma^I_{h_k} \subseteq \mathcal{S}^I_{h_k}$~such~that 
            \begin{align*}
                \sum_{T \in \tau_{h_k}} \eta_{\mathrm{gap},T}^2( \bar u_k^{cr},\overline{z}_k^{rt}) &\ge \theta_{\mathcal{T}} \sum_{T \in \mathcal{T}_{h_k}} \eta_{\mathrm{gap},T}^2( \bar u_k^{cr},\overline{z}_k^{rt})\,,\\ 
                \sum_{S \in \sigma_{h_k}} \eta_{\mathrm{gap},S}^2( \bar u_k^{cr},\overline{z}_k^{rt}) &\ge \theta_{\mathcal{S}} \sum_{S \in \mathcal{S}_{h_k}^I} \eta_{\mathrm{gap},S}^2( \bar u_k^{cr},\overline{z}_k^{rt})\,.
            \end{align*} 
	 		\item[(`Refine')]\hypertarget{Refine}{} Perform a conforming refinement of $\mathcal{T}_{h_k}$ to obtain $\mathcal{T}_{h_{k+1}}$ such that each $T\in \tau_{h_k}$ or $T\in \mathcal{T}_{h_k}$ with $S \subseteq \partial T$ for some $S\in \sigma_{h_k}^{I}$
            is \emph{`refined'} in $\mathcal{T}_{h_{k+1}}$.  
	 		Increase $k\mapsto k+1$ and proceed with step (\hyperlink{Solve}{'Solve'}).
	 	\end{description}
	 \end{algorithm} 

\begin{remark}[comments on Algorithm \ref{alg:afem}]
    \begin{itemize}[noitemsep,topsep=2pt,leftmargin=!,labelwidth=\widthof{(iii)}]
        \item[(i)] In step (\hyperlink{Solve}{'Solve'}), 
        if $\Gamma_D=\emptyset$, we can employ $\overline{u}_k^{cr} = \Pi^{av}_{h_k} u_k^{cr} $, and if $f=f_h\in \mathcal{L}^0(\mathcal{T}_h)$ and $g=g_h\in \mathcal{L}^0(\mathcal{S}_h^N)$, we can employ $\overline{z}_k^{rt} = z_k^{rt} \in K^*$;
        \item[(ii)] In \hspace{-0.1mm}step \hspace{-0.1mm}(\hyperlink{Mark}{'Mark'}), \hspace{-0.1mm}the \hspace{-0.1mm}minimal \hspace{-0.1mm}subsets  \hspace{-0.1mm}$\tau_{h_k}\hspace{-0.175em}\subseteq\hspace{-0.175em}\mathcal{T}_{h_k}$ \hspace{-0.1mm}and \hspace{-0.1mm}$\sigma^I_{h_k} \hspace{-0.175em}\subseteq\hspace{-0.175em} \mathcal{S}^I_{h_k}$ \hspace{-0.1mm}are \hspace{-0.1mm}found \hspace{-0.1mm}using~\hspace{-0.1mm}Dörfler~\hspace{-0.1mm}\mbox{marking};
        \item[(iii)] In step (\hyperlink{Refine}{'Refine'}),  newest-vertex-bisection is employed as conforming refinement routine;
    \end{itemize} 
\end{remark}

\subsubsection{Example with unknown primal and dual solution}

\hspace{5mm}In this example, let $m=3$,  $\Omega =(-1,1)^2 \setminus ([0,1] \times [-1,0])$ and $f=1 \in  L^2(\Omega)$.~Then, we distinguish two setups with regard to the insulation of boundary parts of $\Omega$:\enlargethispage{10mm}
\begin{itemize}[noitemsep,topsep=2pt,leftmargin=!,labelwidth=\widthof{$\bullet$}]
    \item[$\bullet$]\hypertarget{setup1}{} \emph{Setup 1: (pure insulation).} Let $\Gamma_D=\Gamma_N=\emptyset$ and $\Gamma_I =\partial\Omega$. In this case, we cannot make a statement about the regularity of the primal solution $u\in K$;
    \item[$\bullet$]\hypertarget{setup2}{} \emph{Setup 2: (mixed boundary conditions).} Let $\Gamma_D=[0,1]\times\{0\}$, $\Gamma_N \coloneqq  \{0\} \times [-1,0]$, and $\Gamma_I \coloneqq \partial \Omega \setminus (\Gamma_D\cup\Gamma_N)$. In this case, since at the origin two boundary conditions meet at the angle $\smash{\frac{\pi}{2}}$, regularity results for the Poisson problem on a polygonal~domain~(\textit{cf}.~\cite{Grisvard2011})~\mbox{imply}~that~${u\hspace{-0.1em}\in\hspace{-0.1em} H^{\smash{\frac{4}{3}}}(\Omega)}$.
\end{itemize}

In these two setups, we make the following observations:
\begin{itemize}[noitemsep,topsep=2pt,leftmargin=!,labelwidth=\widthof{$\bullet$}]
    \item[$\bullet$] \emph{Observation \hspace{-0.15mm}1: \hspace{-0.15mm}(Setup \hspace{-0.15mm}\hyperlink{setup1}{1}).} \hspace{-0.15mm}In \hspace{-0.15mm}Figure \hspace{-0.15mm}\ref{fig:adaptive_rate_example3}, \hspace{-0.15mm}we \hspace{-0.15mm}report \hspace{-0.15mm}the \hspace{-0.15mm}optimal \hspace{-0.15mm}error \hspace{-0.15mm}decay \hspace{-0.15mm}of \hspace{-0.15mm}order \hspace{-0.15mm}${\mathcal{O}(h_k^2)\hspace{-0.15em}=\hspace{-0.15em}\mathcal{O}(N_k^{-1})}$, $k=1,\ldots,30$, for both 
    uniform and adaptive mesh-refinement. For adaptive mesh-refinement, we select $\theta_\mathcal{T} = \frac{1}{4}$, $\theta_\mathcal{S} = 0$ or $\theta_\mathcal{T} = \theta_\mathcal{S} = \frac{1}{8}$. Moreover, in Figure \ref{fig:adaptive_solution}\textit{(left)}, we observe~that~the~adaptive algorithm (\textit{cf}.\ Algorithm \ref{alg:afem}) refines the almost uniformly. All this is an indication for that in Setup \hyperlink{setup1}{1}, %in the case of pure insulation, 
    the unique primal solution satisfies $u\in H^2(\Omega)$;
    
    \item[$\bullet$] \emph{Observation 2: (Setup \hyperlink{setup2}{2}).} In Figure \ref{fig:adaptive_rate_example3}, we report the reduced error decay rate $\mathcal{O}(h_k^{\frac{2}{3}})= \mathcal{O}(N_k^{\smash{-\frac{1}{3}}})$, $k=1,\ldots,6$, for 
    uniform mesh-refinement and the optimal error~decay~rate~${\mathcal{O}(h_k^2)= \mathcal{O}(N_k^{-1})}$, $k=1,\ldots,30$, for adaptive mesh-refinement. For adaptive mesh-refinement, we either select $\theta_\mathcal{T} = \frac{1}{4}$, $\theta_\mathcal{S} = 0$ or $\theta_\mathcal{T} = \theta_\mathcal{S} = \frac{1}{8}$. Moreover, in Figure \ref{fig:adaptive_solution}\textit{(right)}, we~observe~that~the adaptive algorithm (\textit{cf}.\ Algorithm \ref{alg:afem}) refines towards the origin, where we expect a singularity of the primal solution, due to the different touching (with angle $\frac{\pi}{2}$) boundary conditions.  All this is an indication for that in Setup \hyperlink{setup2}{2}, the unique primal solution~indeed~satisfies~$u\in H^{\smash{\frac{4}{3}}}(\Omega)$.\vspace{-1.75mm}
\end{itemize}
\begin{figure}[H]
    \centering
    \includegraphics[width=\linewidth]{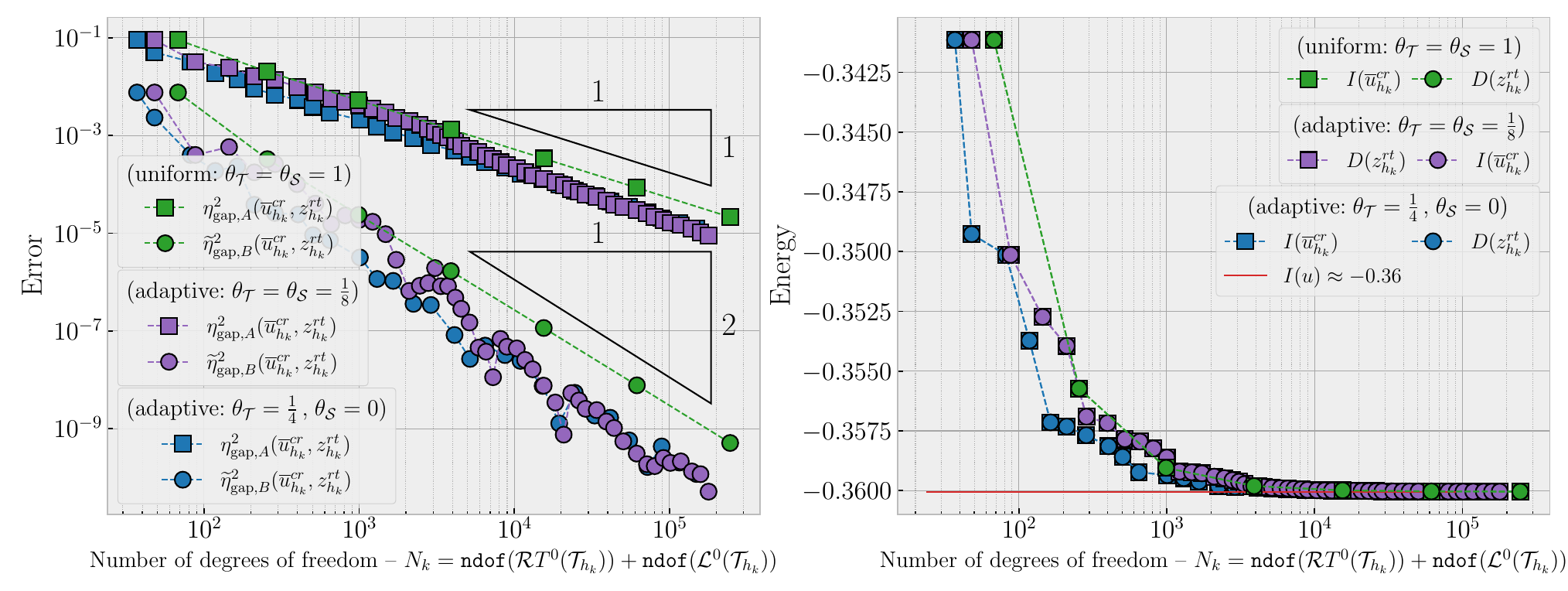}\\\includegraphics[width=\linewidth]{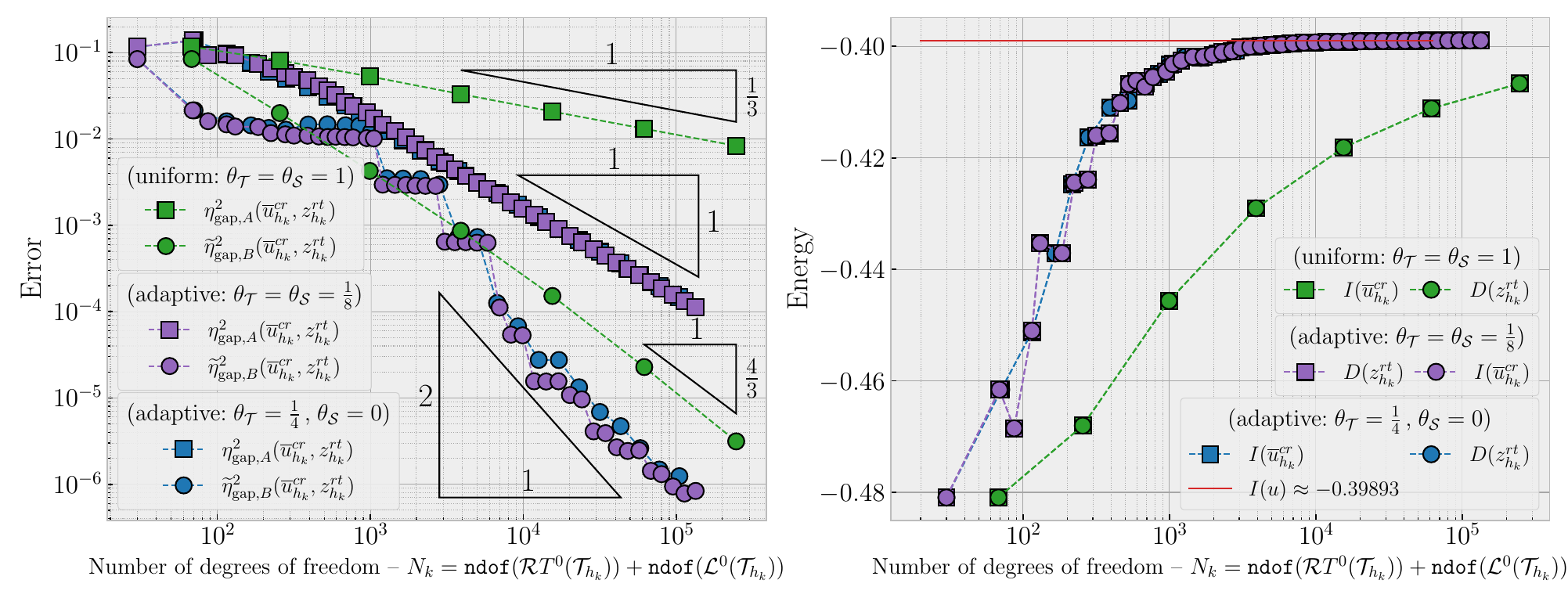}\vspace{-1mm}
    \caption{\textit{top row:} Setup \protect\hyperlink{setup1}{1} (pure insulation); \textit{bottom row:} Setup \protect\hyperlink{setup2}{2} (mixed boundary conditions);
    \textit{left column:} logarithmic plots of $\rho_{\textup{tot}}^2(\overline{u}_{h_k}^{cr},z_{h_k}^{rt})=\eta_{\textup{gap}}^2(\overline{u}_{h_k}^{cr},z_{h_k}^{rt})$; \textit{right column:} logarithmic plots of $I(\overline{u}_{h_k}^{cr})$ and $D(z_{h_k}^{rt})$, where $\overline{u}_{h_k}^{cr}\coloneqq \Pi_{h_k}^{av} u_{h_k}^{cr}\in \mathcal{S}^{1,cr}(\mathcal{T}_{h_k})\cap H^1(\Omega)$; each for $k=0,\ldots,30$, when using adaptive mesh-refinement, and for 
    $k=0,\ldots,6$, when using uniform mesh-refinement.}
    \label{fig:adaptive_rate_example3}
\end{figure}

%\begin{figure}[H]
%    \centering
%    \includegraphics[width=\linewidth]{figures/adaptive_rate_example2.pdf}
%    \caption{Example 3: mixed BC}
%    \label{fig:adaptive_rate_example2}
%\end{figure}

\begin{figure}[H]
    \centering
\includegraphics[width=0.45\linewidth]{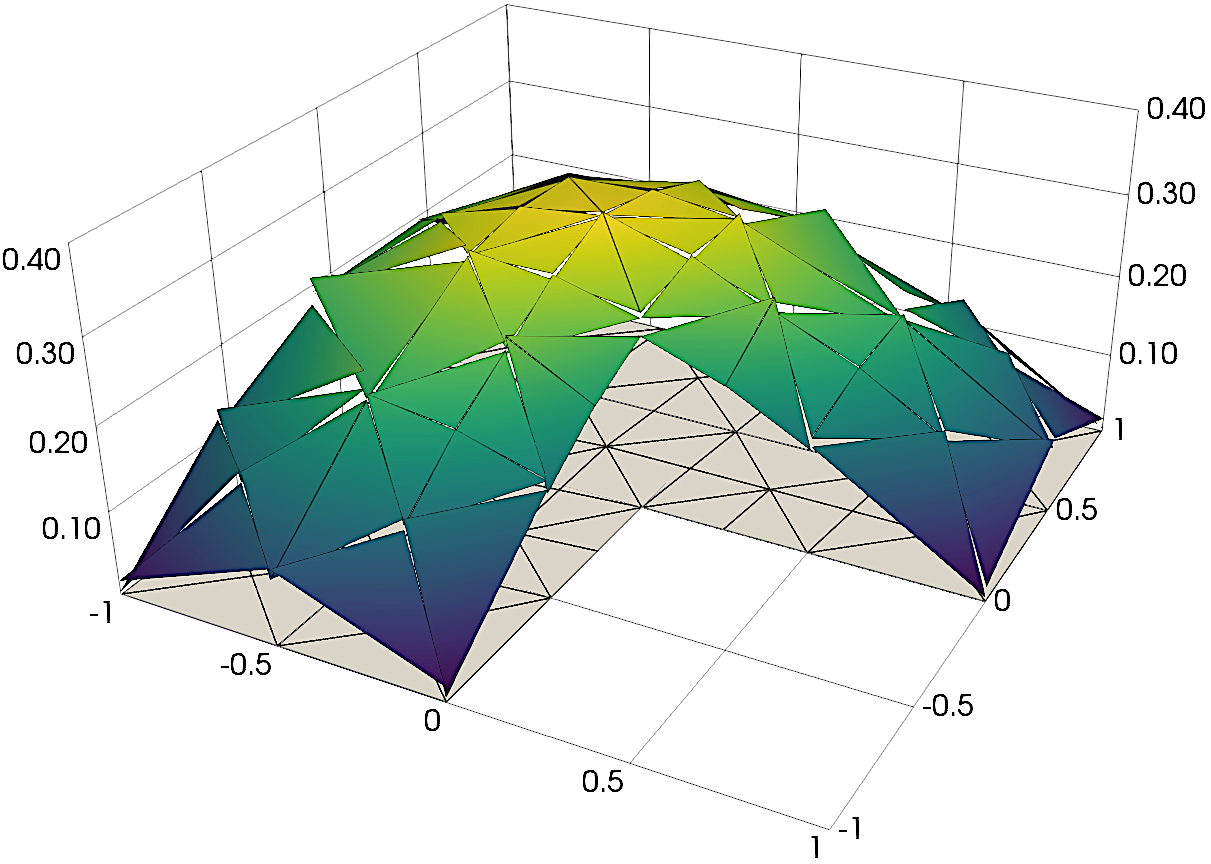}
\includegraphics[width=0.45\linewidth]{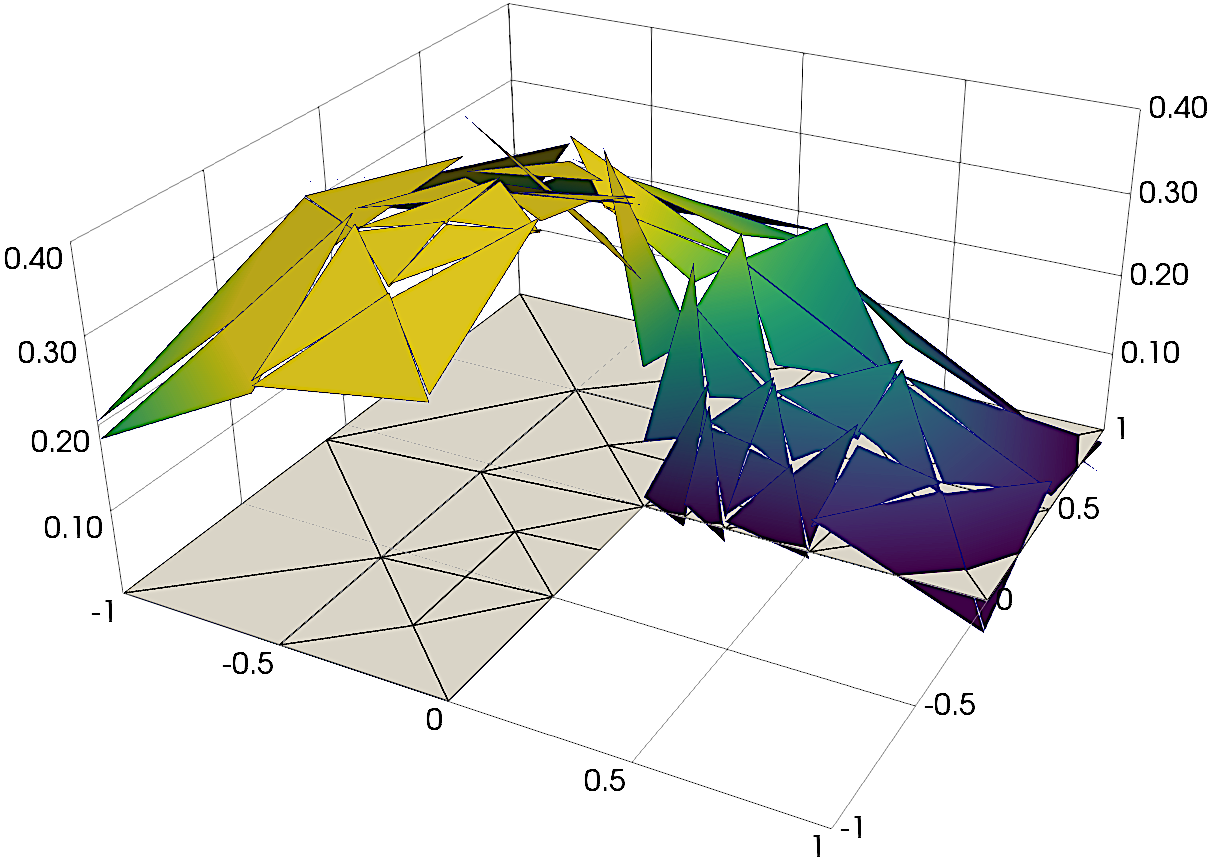} 
\includegraphics[width=0.45\linewidth]{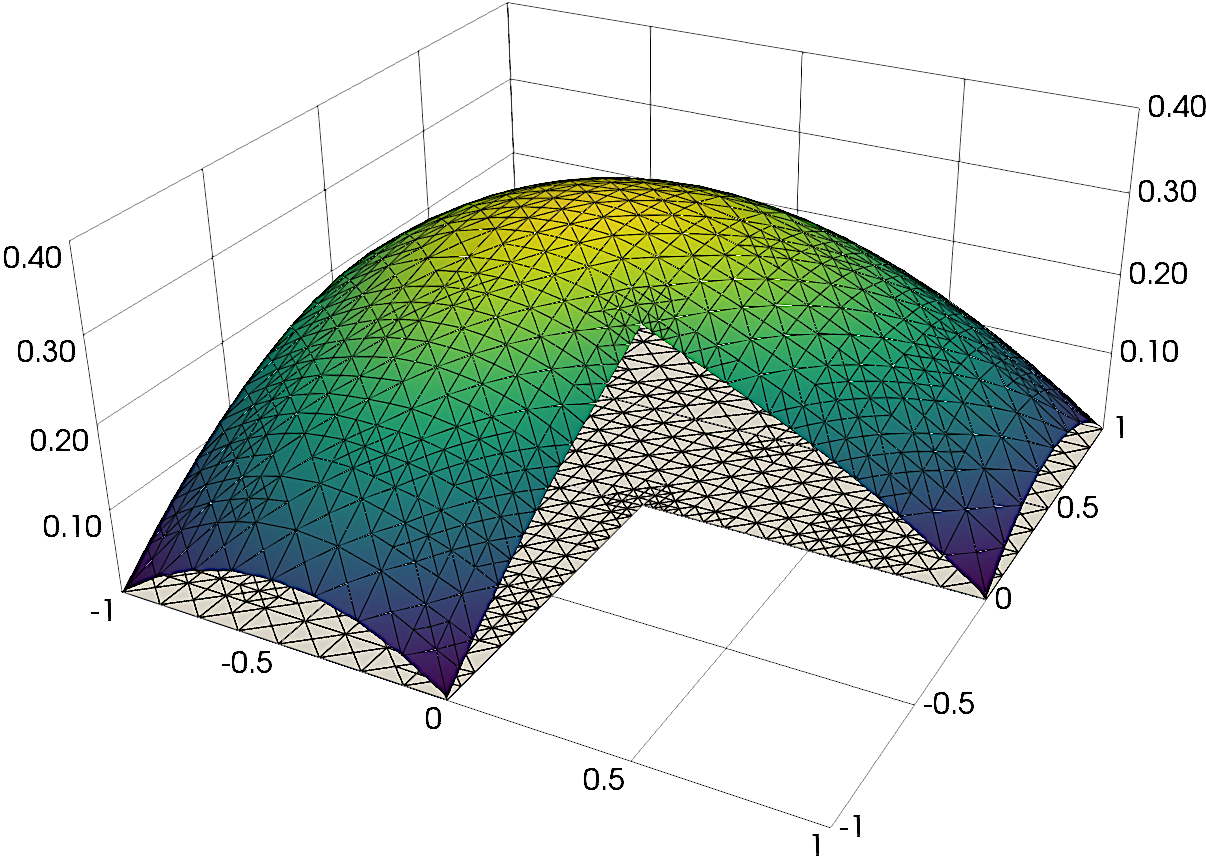}
\includegraphics[width=0.45\linewidth]{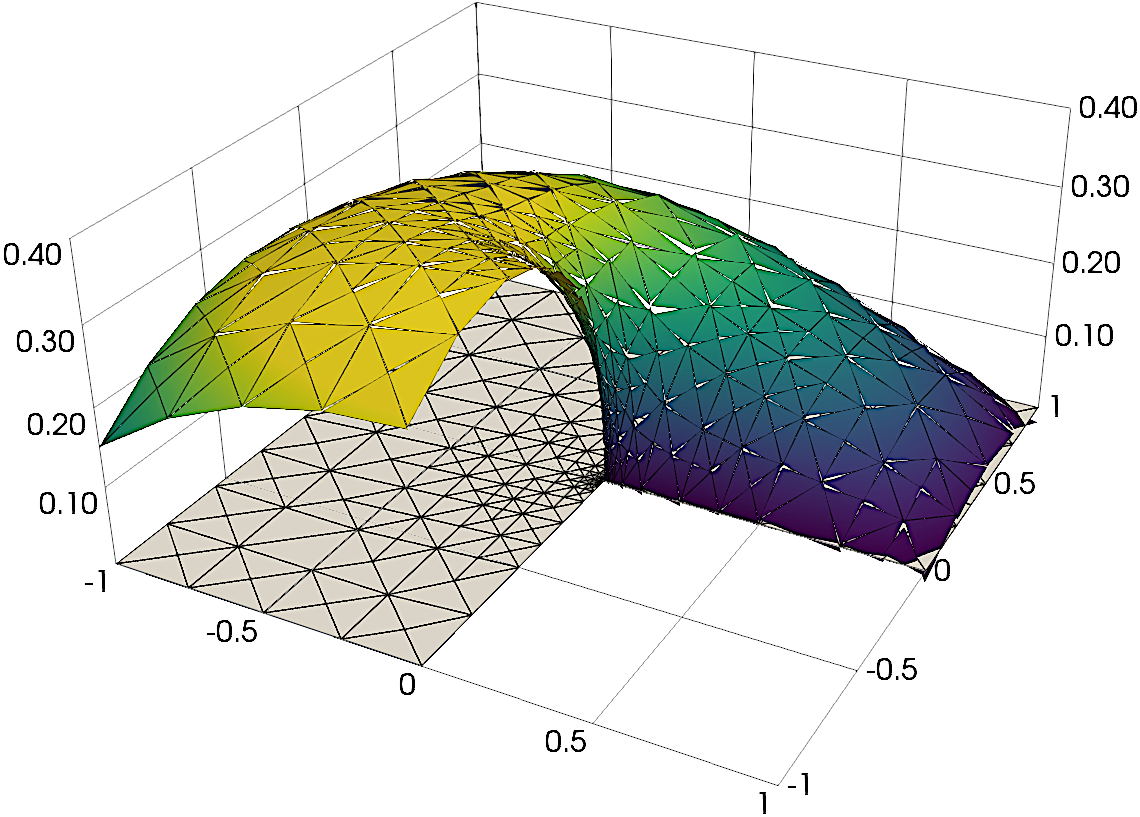}
\includegraphics[width=0.45\linewidth]{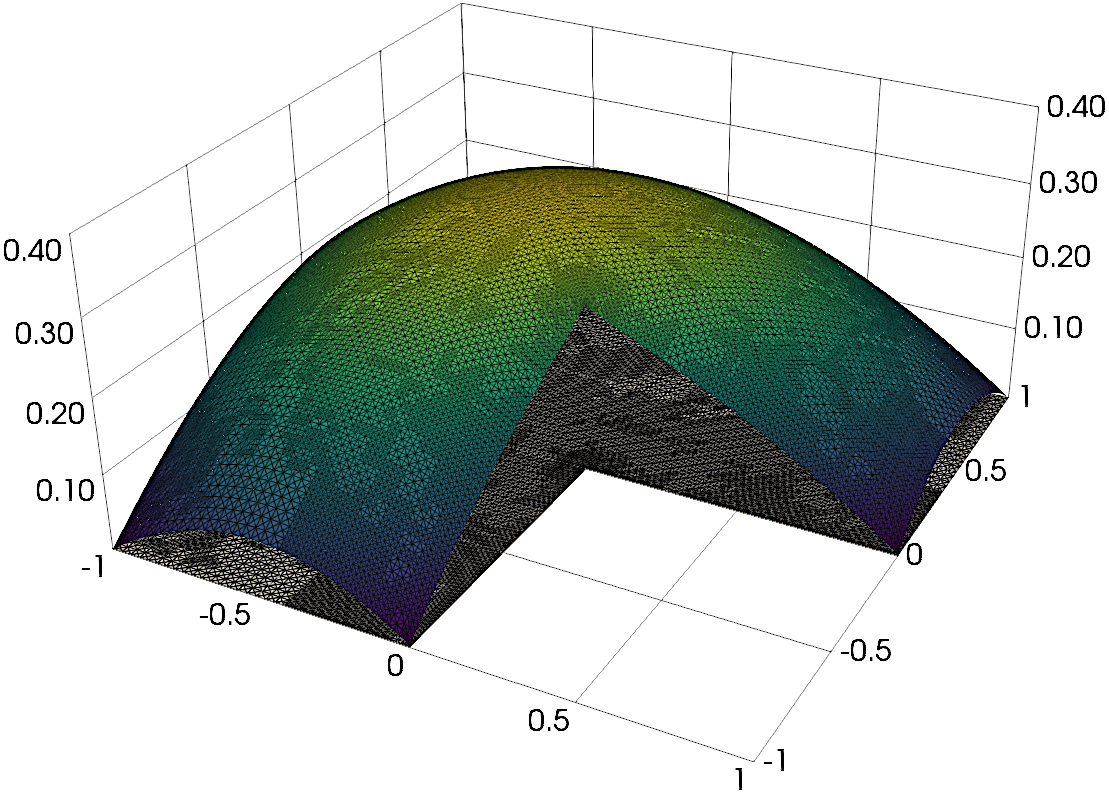}
\includegraphics[width=0.45\linewidth]{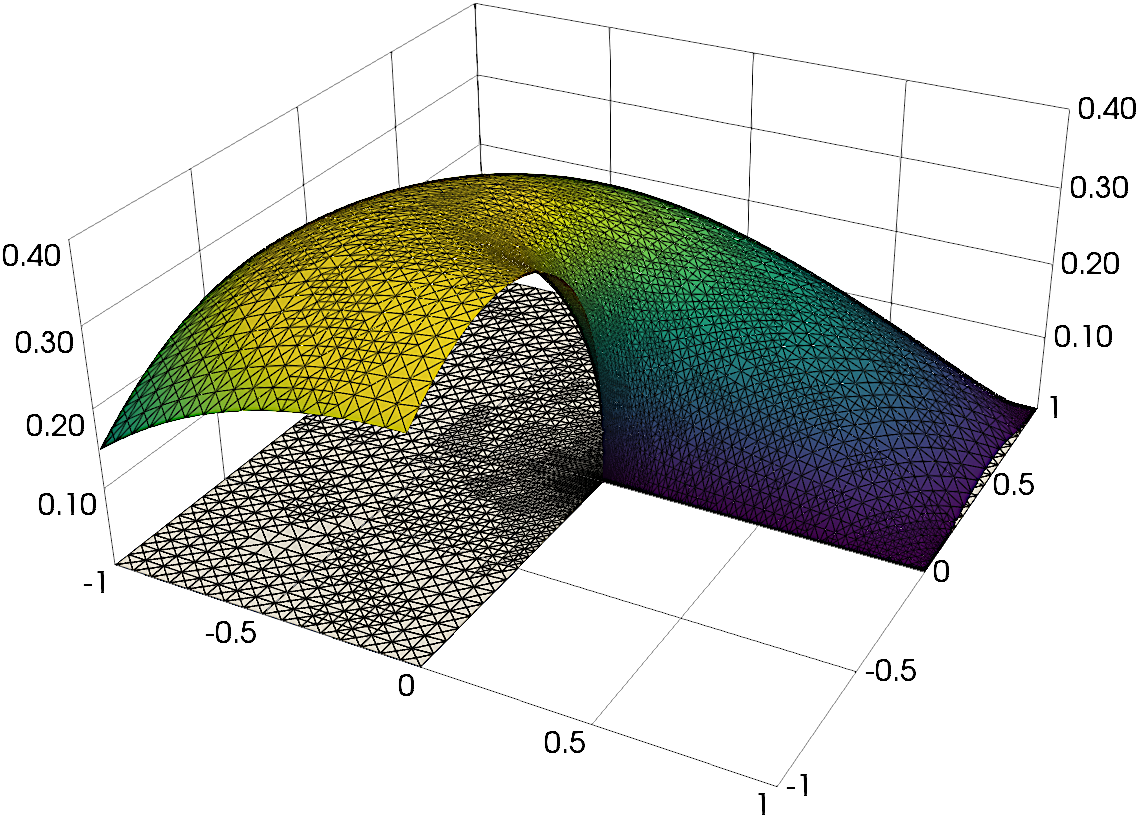}
    \caption{The discrete  primal solution $u_{h_k}^{cr}\in \mathcal{S}^{1,cr}(\mathcal{T}_{h_k})$ and the adaptively refined triangulation $\mathcal{T}_{h_k}$ %using the gap estimator $\eta_{\text{gap,A},h}^2$ 
    pictured at refinement level $k=5$ \textit{(top row)}, $k=15$ \textit{(middle row)}, and $k=25$ \textit{(bottom row)}. 
    The \textit{left column} corresponds to the test case with  purely insulated~\mbox{boundary}~(\textit{cf}.~Setup~\protect\hyperlink{setup1}{1}),~whereas the \textit{right column} corresponds to the test case with mixed~boundary~conditions~(\textit{cf}.~Setup~\protect\hyperlink{setup2}{2}).}
    \label{fig:adaptive_solution}
\end{figure}\enlargethispage{7.5mm}\vspace{-5mm}
 
\if0
\begin{figure}[H]
    \centering
\includegraphics[width=0.45\linewidth,trim={16cm 3cm 16cm 6cm},clip]{figures/marini5.png}
\includegraphics[width=0.45\linewidth,trim={16cm 3cm 16cm 6cm},clip]{figures/marini5_I.png} \\[-1mm]
\includegraphics[width=0.45\linewidth,trim={16cm 3cm 16cm 6cm},clip]{figures/marini15.png}
\includegraphics[width=0.45\linewidth,trim={16cm 3cm 16cm 6cm},clip]{figures/marini15_I.png}\\[-1mm]
\includegraphics[width=0.45\linewidth,trim={16cm 3cm 16cm 6cm},clip]{figures/marini25.png}
\includegraphics[width=0.45\linewidth,trim={16cm 3cm 16cm 6cm},clip]{figures/marini25_I.png}
    \caption{The reconstructed primal solution $u_{h_k}^{cr}$ and the adaptively refined mesh using the gap estimator $\eta_{\text{gap,A},h}^2$ pictured at refinement level $k=5$ (top row), $k=15$ (middle row), and $k=25$ (bottom row). The left column corresponds to the test case with mixed boundary conditions, whereas the right column corresponds to the test case with a purely insulated boundary. }
    \label{fig:adaptive_solution}
\end{figure}\enlargethispage{10mm}
\fi

\subsection{Optimal insulation of a house}\vspace{-1mm}
    \label{ss:experiment_3}

\hspace{5mm}In this experiment, we study the optimal distribution of a given amount of insulating material attached to an insulating body $\Omega\subseteq \mathbb{R}^3$  modelling a simple house with attached garage.~In~doing~so, we assume that the windows, doors, and floors of the house exhibit fixed insulating properties,~\textit{i.e.}, we assign  Neumann boundary conditions to the windows, doors, and floors of the house, on which we prescribe~an~outward~heat~flux. We believe this is a reasonable assumption, as these elements are typically standardized in the construction industry and provided by external manufacturers. 
We do not impose Dirichlet boundary conditions (\textit{i.e.}, $\Gamma_D=\emptyset$), so that the insulated boundary $\Gamma_I\coloneqq \partial\Omega\setminus \Gamma_N$ is given via the roofs and the walls without windows and doors. For~simplicity, we set $f=1\in L^2(\Omega)$ (\textit{i.e.}, the house is uniformly heated) and we prescribe a uniform outward heat flux $g = \frac{1}{5}\in H^{-\frac{1}{2}}(\Gamma_N)$. % on the Neumann boundary $\Gamma_N$ (\textit{i.e.}, the windows, doors, and floors). 
We set the %(normalized) 
total amount of insulating material to be%\textcolor{red}{total insulating power} to be
\begin{align*}
    m\coloneqq  \|h\|_{1,\Gamma_I} = \tfrac{1}{4}|\Gamma_I|\,.
\end{align*} 

In \Cref{fig:house}, the surface temperature field $\pi_h u_h^{cr}\in \mathcal{L}^0(\mathcal{S}_h^{\partial})$ of the house $\Omega$ %~at~thermal~equilibrium 
and the distribution of the insulating material (in the direction of %the outward unit normal vector field 
$n\colon \partial\Omega\to\mathbb{S}^2$ for visualization purposes), \textit{i.e.},
\begin{align}\label{def:discrete_distribution}
    \smash{\widetilde{h}_{u_h^{cr}}\coloneqq \tfrac{m}{\|\pi_hu_h^{cr}\|_{1,\Gamma_I}}\vert \pi_hu_h^{cr}\vert \in \mathcal{L}^0(\mathcal{S}_h^{I})}\,,
\end{align}
are depicted. The surface temperature of the insulated portion $\Gamma_I\subseteq \partial\Omega$ of the house is non-zero, indicating that the inclusion of the insulating material impedes heat transfer at the boundary $\partial\Omega$. Moreover, the  distribution of the insulating material \eqref{def:discrete_distribution} is not uniform, but instead tends to prioritize the placement of insulating material on the roof of the house. This appears physically reasonable in light of Fourier's law, which states that the rate of conductive heat transfer is proportional to the exposed surface area.\vspace{-2.5mm}

% Therefore, to mitigate conductive losses, our results suggest one should prioritize insulating regions with larger exposed surface~area.\vspace{-1.5mm}
% \todo{KK: solution is apparently also curvature dependent, as a uniformly heated annular region leads to more insulation in the interior cut out circle, so exposed surface area isn't the full story...}
\begin{figure}[H] 
    \centering
   \includegraphics[width=0.325\textwidth]{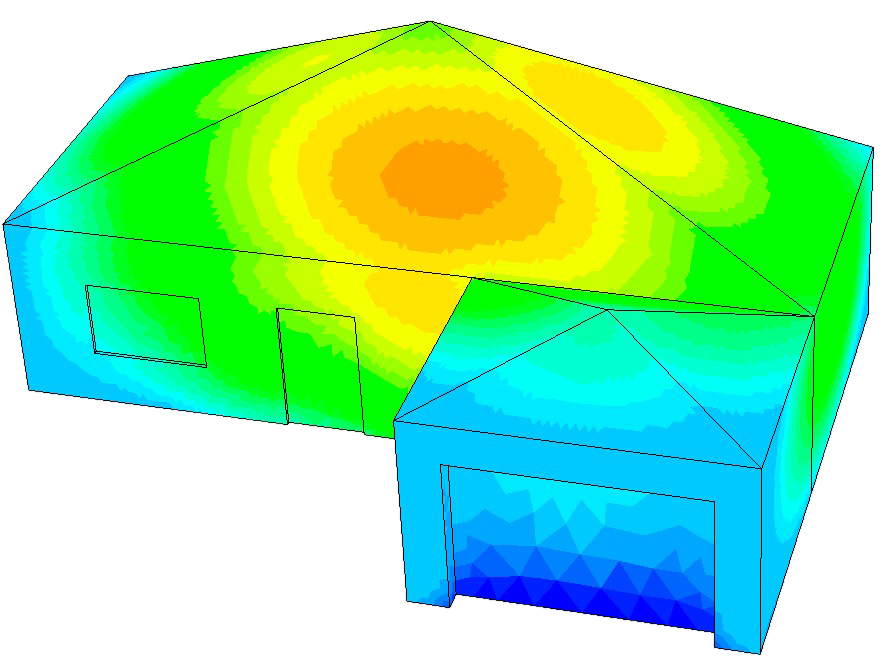} %\hspace{1em}
   \includegraphics[width=0.325\textwidth]{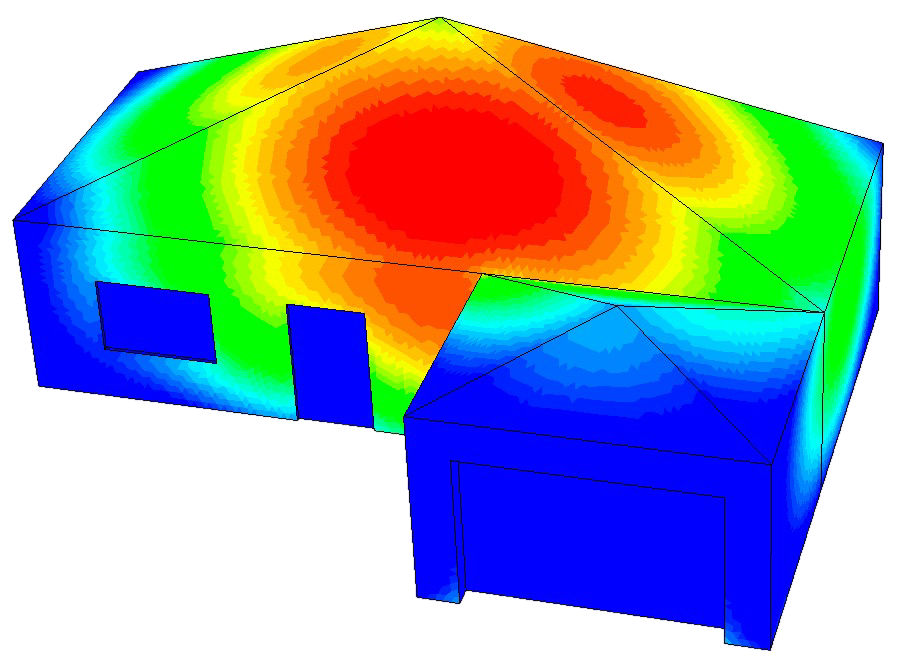} %\hspace{1em}
   \includegraphics[width=0.325\textwidth]{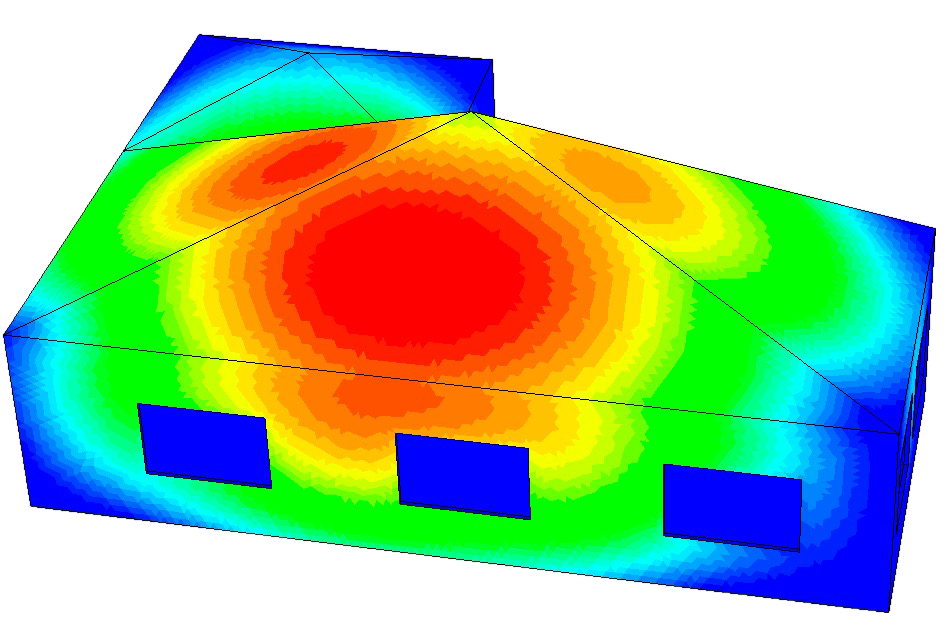}\vspace{-1mm}
    \caption{\textit{left:} surface temperature field $\pi_h u_h^{cr}\in \mathcal{L}^0(\mathcal{S}_h^{\partial})$; \textit{right:} distribution of the insulating material $\widetilde{h}_{u_h^{cr}}  \in \mathcal{L}^0(\mathcal{S}_h^{I})$ (\textit{cf}.\ \eqref{def:discrete_distribution}); each 
    for a uniformly heated home (\textit{i.e.}, $f=1$)  with insulating mass $m = \frac{1}{4}|\Gamma_I|$ and uniform outward heat flux (\textit{i.e.}, $g=\frac{1}{5}$) at the windows,~doors,~and~floors.  The triangulation $\mathcal{T}_h$ consists of $150,370$ tetrahedral elements and the semi-smooth Newton method (\textit{cf}.\ Algorithm \ref{alg:semismooth_Newton}) terminates after $8$ iterations (at the exact discrete solution).}
    \label{fig:house}
\end{figure}\enlargethispage{10mm}\vspace{-5mm}

	{\setlength{\bibsep}{0pt plus 0.0ex}\small 
		%\bibliographystyle{aomplain}
		%\bibliography{references}

        \providecommand{\bysame}{\leavevmode\hbox to3em{\hrulefill}\thinspace}
\providecommand{\noopsort}[1]{}
\providecommand{\mr}[1]{\href{http://www.ams.org/mathscinet-getitem?mr=#1}{MR~#1}}
\providecommand{\zbl}[1]{\href{http://www.zentralblatt-math.org/zmath/en/search/?q=an:#1}{Zbl~#1}}
\providecommand{\jfm}[1]{\href{http://www.emis.de/cgi-bin/JFM-item?#1}{JFM~#1}}
\providecommand{\arxiv}[1]{\href{http://www.arxiv.org/abs/#1}{arXiv~#1}}
\providecommand{\doi}[1]{\url{https://doi.org/#1}}
\providecommand{\MR}{\relax\ifhmode\unskip\space\fi MR }
\providecommand{\MRhref}[2]{%
  \href{http://www.ams.org/mathscinet-getitem?mr=#1}{#2}
}
\providecommand{\href}[2]{#2}

	}

\end{document}